\documentclass[12pt, reqno]{amsart}

\usepackage[margin=1.1in]{geometry}

\usepackage[utf8]{inputenc}
\usepackage[T1]{fontenc}
\usepackage[english]{babel}

\usepackage{graphicx}
\usepackage{tikz-cd}
\usepackage{import}
\usepackage{wrapfig}
\usepackage{tabularx}

\usepackage{etoolbox}

\makeatletter
\let\old@tocline\@tocline
\let\section@tocline\@tocline
\newcommand{\subsection@dotsep}{4.5}
\newcommand{\subsubsection@dotsep}{4.5}
\patchcmd{\@tocline}
  {\hfil}
  {\nobreak
     \leaders\hbox{$\m@th
        \mkern \subsection@dotsep mu\hbox{.}\mkern \subsection@dotsep mu$}\hfill
     \nobreak}{}{}
\let\subsection@tocline\@tocline
\let\@tocline\old@tocline

\patchcmd{\@tocline}
  {\hfil}
  {\nobreak
     \leaders\hbox{$\m@th
        \mkern \subsubsection@dotsep mu\hbox{.}\mkern \subsubsection@dotsep mu$}\hfill
     \nobreak}{}{}
\let\subsubsection@tocline\@tocline
\let\@tocline\old@tocline

\let\old@l@subsection\l@subsection
\let\old@l@subsubsection\l@subsubsection

\def\@tocwriteb#1#2#3{%
  \begingroup
    \@xp\def\csname #2@tocline\endcsname##1##2##3##4##5##6{%
      \ifnum##1>\c@tocdepth
      \else \sbox\z@{##5\let\indentlabel\@tochangmeasure##6}\fi}%
    \csname l@#2\endcsname{#1{\csname#2name\endcsname}{\@secnumber}{}}%
  \endgroup
  \addcontentsline{toc}{#2}%
    {\protect#1{\csname#2name\endcsname}{\@secnumber}{#3}}}%

\newlength{\@tocsectionindent}
\newlength{\@tocsubsectionindent}
\newlength{\@tocsubsubsectionindent}
\newlength{\@tocsectionnumwidth}
\newlength{\@tocsubsectionnumwidth}
\newlength{\@tocsubsubsectionnumwidth}
\newcommand{\settocsectionnumwidth}[1]{\setlength{\@tocsectionnumwidth}{#1}}
\newcommand{\settocsubsectionnumwidth}[1]{\setlength{\@tocsubsectionnumwidth}{#1}}
\newcommand{\settocsubsubsectionnumwidth}[1]{\setlength{\@tocsubsubsectionnumwidth}{#1}}
\newcommand{\settocsectionindent}[1]{\setlength{\@tocsectionindent}{#1}}
\newcommand{\settocsubsectionindent}[1]{\setlength{\@tocsubsectionindent}{#1}}
\newcommand{\settocsubsubsectionindent}[1]{\setlength{\@tocsubsubsectionindent}{#1}}

\renewcommand{\l@section}{\section@tocline{1}{\@tocsectionvskip}{\@tocsectionindent}{}{\@tocsectionformat}}%
\renewcommand{\l@subsection}{\subsection@tocline{2}{\@tocsubsectionvskip}{\@tocsubsectionindent}{}{\@tocsubsectionformat}}%
\renewcommand{\l@subsubsection}{\subsubsection@tocline{3}{\@tocsubsubsectionvskip}{\@tocsubsubsectionindent}{}{\@tocsubsubsectionformat}}%
\newcommand{\@tocsectionformat}{}
\newcommand{\@tocsubsectionformat}{}
\newcommand{\@tocsubsubsectionformat}{}
\expandafter\def\csname toc@1format\endcsname{\@tocsectionformat}
\expandafter\def\csname toc@2format\endcsname{\@tocsubsectionformat}
\expandafter\def\csname toc@3format\endcsname{\@tocsubsubsectionformat}
\newcommand{\settocsectionformat}[1]{\renewcommand{\@tocsectionformat}{#1}}
\newcommand{\settocsubsectionformat}[1]{\renewcommand{\@tocsubsectionformat}{#1}}
\newcommand{\settocsubsubsectionformat}[1]{\renewcommand{\@tocsubsubsectionformat}{#1}}
\newlength{\@tocsectionvskip}
\newcommand{\settocsectionvskip}[1]{\setlength{\@tocsectionvskip}{#1}}
\newlength{\@tocsubsectionvskip}
\newcommand{\settocsubsectionvskip}[1]{\setlength{\@tocsubsectionvskip}{#1}}
\newlength{\@tocsubsubsectionvskip}
\newcommand{\settocsubsubsectionvskip}[1]{\setlength{\@tocsubsubsectionvskip}{#1}}

\patchcmd{\tocsection}{\indentlabel}{\makebox[\@tocsectionnumwidth][l]}{}{}
\patchcmd{\tocsubsection}{\indentlabel}{\makebox[\@tocsubsectionnumwidth][l]}{}{}
\patchcmd{\tocsubsubsection}{\indentlabel}{\makebox[\@tocsubsubsectionnumwidth][l]}{}{}

\newcommand{\@sectypepnumformat}{}
\renewcommand{\contentsline}[1]{%
  \expandafter\let\expandafter\@sectypepnumformat\csname @toc#1pnumformat\endcsname%
  \csname l@#1\endcsname}
\newcommand{\@tocsectionpnumformat}{}
\newcommand{\@tocsubsectionpnumformat}{}
\newcommand{\@tocsubsubsectionpnumformat}{}
\newcommand{\setsectionpnumformat}[1]{\renewcommand{\@tocsectionpnumformat}{#1}}
\newcommand{\setsubsectionpnumformat}[1]{\renewcommand{\@tocsubsectionpnumformat}{#1}}
\newcommand{\setsubsubsectionpnumformat}[1]{\renewcommand{\@tocsubsubsectionpnumformat}{#1}}
\renewcommand{\@tocpagenum}[1]{%
  \hfill {\mdseries\@sectypepnumformat #1}}

\let\oldappendix\appendix
\renewcommand{\appendix}{%
  \leavevmode\oldappendix%
  \addtocontents{toc}{%
    \protect\settowidth{\protect\@tocsectionnumwidth}{\protect\@tocsectionformat\sectionname\space}%
    \protect\addtolength{\protect\@tocsectionnumwidth}{2em}}%
}
\makeatother

\makeatletter
\settocsectionnumwidth{2em}
\settocsubsectionnumwidth{2.5em}
\settocsubsubsectionnumwidth{3em}
\settocsectionindent{1pc}%
\settocsubsectionindent{\dimexpr\@tocsectionindent+\@tocsectionnumwidth}%
\settocsubsubsectionindent{\dimexpr\@tocsubsectionindent+\@tocsubsectionnumwidth}%
\makeatother

\settocsectionvskip{5pt}
\settocsubsectionvskip{0pt}
\settocsubsubsectionvskip{0pt}
    
\settocsectionformat{\bfseries}
\settocsubsectionformat{\mdseries}
\settocsubsubsectionformat{\mdseries}
\setsectionpnumformat{\bfseries}
\setsubsectionpnumformat{\mdseries}
\setsubsubsectionpnumformat{\mdseries}

\let\oldtableofcontents\tableofcontents
\renewcommand{\tableofcontents}{%
  \vspace*{-\linespacing}
  \oldtableofcontents}

\usepackage{amsmath, amsthm, amssymb, amsfonts}
\usepackage{enumitem}
\usepackage{mathrsfs}
\usepackage{nicefrac}

\usepackage{xcolor}
\definecolor{antiquefuchsia}{rgb}{0.57, 0.36, 0.51}
\definecolor{azure}{rgb}{0.0, 0.5, 1.0}
\definecolor{blue(ncs)}{rgb}{0.0, 0.53, 0.74}

\usepackage[colorlinks = true, linkcolor = blue(ncs), citecolor = antiquefuchsia]{hyperref}

\usepackage{comment}

\theoremstyle{plain}

\newtheorem{thm}{Theorem}[section]

\newtheorem{lem}[thm]{Lemma}

\newtheorem{prop}[thm]{Proposition}

\newtheorem{cond}[thm]{Condition}

\theoremstyle{definition}
\newtheorem{definition}[thm]{Definition}

\newtheorem{rem}[thm]{Remark}

\theoremstyle{remark}

\numberwithin{equation}{section}

\renewcommand{\u}{\mathfrak{u}} 
\newcommand{\so}{\mathfrak{so}} 
\renewcommand{\d}{\mathrm{d}} 
\newcommand{\DD}{\mathbb{D}} 
\DeclareMathOperator{\End}{End} 

\DeclareMathOperator{\supp}{supp}

\DeclareMathOperator{\range}{range}
\DeclareMathOperator{\dom}{dom}

\newcommand{\Z}{\mathbb{Z}} 

\newcommand{\X}{\mathcal{X}}
\newcommand{\II}{\mathbb{I}}
\newcommand{\B}{\mathcal{B}}

\newcommand{\I}{I}
\renewcommand{\L}{\mathcal{L}}

\newcommand{\D}{\mathscr{D}} 

\renewcommand{\min}{\mathrm{min}}
\renewcommand{\max}{\mathrm{max}}
\newcommand{\true}{\star}
\newcommand{\truep}{\star}
\newcommand{\iid}{\mathrm{iid}}

\DeclareMathOperator{\ord}{ord}

\DeclareMathOperator{\spn}{span}
\newcommand{\G}{\mathscr G}
\newcommand{\C}{\mathbb C}

\newcommand{\N}{\mathbb N}
\newcommand{\init}{\mathrm{init}}
\newcommand{\MAP}{\mathrm{MAP}}
\newcommand{\mix}{\mathrm{mix}}

\newcommand{\id}{\mathrm{id}} 

\newcommand{\R}{\mathbb R}

\title[Log-concave approximations]{On log-concave approximations of high-dimensional posterior measures and stability properties in non-linear inverse problems}
\date{\today,~~\textit{Department of Pure Mathematics and Mathematical Statistics, University of Cambridge}, email: nickl@maths.cam.ac.uk,  bohr@math.uni-bonn.de}

	\author[J. Bohr \& R. Nickl]{Jan Bohr ~~~~~~~Richard Nickl \\ ~~ \\ University of Cambridge ~\today }

\usepackage{physics}
\usepackage[parfill]{parskip}
\begin{document}

\maketitle

\begin{abstract}
The problem of efficiently generating random samples from high-dimensional and non-log-concave posterior measures arising in nonlinear regression problems is considered.  Extending investigations from \cite{NW20}, local and global stability properties of the model are identified under which such posterior distributions can be approximated in Wasserstein distance by suitable log-concave measures.  This allows the use of fast gradient based sampling algorithms, for which convergence guarantees are established that scale polynomially in all relevant quantities (assuming `warm' initialisation). The scope of the general theory is illustrated in a non-linear inverse problem from integral geometry for which new stability results are derived.
\end{abstract}

\vspace{2em}

\tableofcontents

\section{Introduction}\label{introduction}

The circle of problems we study here evolves around the question of how to efficiently generate random samples from high-dimensional Gibbs-type probability measures 
\begin{equation}\label{gibbs}
\d\Pi(\theta|Z^{(N)}) \propto e^{\ell(\theta\vert Z^{(N)})} \d \Pi ( \theta),~~\theta \in \R^D, \quad D \in \N,
\end{equation}
with \textit{non-concave} `energies' $\theta\mapsto \ell_N(\theta) \equiv \ell(\theta \vert Z^{(N)})$. The maps $\ell_N$ can be thought of as log-likelihood functions from a non-linear regression equation generating a random data vector $Z^{(N)}$ of length $N \in \mathbb N$, and the reference (or \textit{prior}) measure $\Pi$ is assumed to be Gaussian. Sampling algorithms play a fundamental role in scientific applications such as Bayesian inference and inverse problems \cite{CR04, S10, CRSW13, DS16}, data assimilation \cite{CMR05, MH12, RC15} and probabilistic numerics \cite{D88, BOGOS19} and have close connections to computational complexity theory and statistical physics. Since direct computation of the `partition function' (normalising factor) in (\ref{gibbs}) is typically not viable, the sampling task can be computationally hard if not impossible when $D$ and $N$ are large. See, for instance, \cite{RvH15} for the discussion of a prototypical `curse of dimensionality' in a filtering problem, or \cite{BWZ20} in a sampling problem involving sparse PCA. 

\smallskip

In applications, the non-concavity of $\theta \mapsto \ell_N(\theta)$ may arise from a partial differential equation (PDE) underlying the data generating process for $Z^{(N)}$, specifically from the non-linearity of the `parameter to solution' map. In the present article we continue investigations from \cite{NW20} and demonstrate that if the PDE has certain local and global `stability' properties, then `polynomial time' sampling from $\Pi(\cdot|Z^{(N)})$ is possible even in high-dimensional ($D \to \infty$), spiked ($N \to \infty$), non-log-concave situations. More precisely, we devise a general set of analytical hypotheses on the underlying PDE model such that the measure (\ref{gibbs}) can be sampled from within precision $\epsilon$ after $O(N^{b_1} D^{b_2} \epsilon^{-b_3})$ algorithmic iterations (each iteration requiring one evaluation of $\ell_N$, and assuming appropriate initialisation). Our theory applies to the elliptic example with a Schr\"odinger equation considered in \cite{NW20}, but we further demonstrate its applicability in a more demanding transport PDE setting arising with non-Abelian X-ray transforms which have been studied recently in \cite{PSU12, MNP21, PaSa20} (see \cite{nature2, nature1, Nov19} for applications and \cite{Ver91, Nov02, Esk04} for earlier references).  Our main contribution, next to novel stability estimates for X-ray transforms described below, is a \textit{log-concave approximation theorem} which ensures that $\Pi(\cdot|Z^{(N)})$ from \eqref{gibbs} is very close to a log-concave measure in Wasserstein distance, with high probability under the law of the data $Z^{(N)}$.

\smallskip

The most widely used methods to generate samples for \eqref{gibbs} are Markov chain Monte Carlo (MCMC) algorithms, where one runs a Markov chain $(\vartheta_k:k\in \N)$ on $\R^D$ that has $\Pi(\cdot\vert Z^{(N)})$ as invariant measure.  Recent years have seen great progress on computational performance guarantees of such algorithms, of which we want to highlight two advances that are relevant in our PDE setting: The first is the work of Hairer, Stuart and Vollmer \cite{HSV14}, who prove a $D$-independent spectral gap for the so called {\it preconditioned Crank-Nicolson} (pCN) algorithm in settings where $\ell_N$ is uniformly bounded and Lipschitz in $\theta$. While their results imply $D$-independent convergence guarantees, they remain non-quantitative in the magnitude (or `temperature') of the energy $\ell_N$ and hence in $N$ (cf.~the discussion in \textsection 1.1.2 in \cite{NW20}). The second line of advances concerns \textit{log-concave} Gibbs-measures (\ref{gibbs}) and was initiated by Dalalyan \cite{Dal17} and Durmus and Moulines \cite{DM17, DuMo19}. Their work on gradient based {\it (unadjusted) Langevin algorithms} (ULA) implies mixing time bounds of the law of the iterates $\vartheta_k$ towards $\Pi(\cdot|Z^{(N)})$ which scale \textit{polynomially} in both $D, N$. This exploits fundamental ideas of hyper-contractivity of Langevin diffusions and Wasserstein gradient flows going back to Bakry and \'Emery \cite{BE84} and Jordan, Kinderlehrer and Otto \cite{JKO98}, see also the monographs \cite{V09, BGL14}. While these techniques do not apply directly to our non-log-concave setting, the general log-concave approximation theorems proved below provide the crucial link to transfer the fast mixing properties of such gradient based methods to the more complex PDE-driven models considered here, at least if a reasonable initialiser is available. We discuss the difference between our results and traditional Laplace approximations of posterior measures after Theorem \ref{onetrickpony} below. Note that in general, without exploiting  log-concavity, the runtime of sampling algorithms for high-dimensional (even unimodal) posterior distributions scales exponentially in dimension, see the very recent  \cite{BMNW23}.

\medskip

\textbf{Non-linear inverse problems and PDE stability.} We now discuss informally the notions of `local and global stability' of the underlying model that we will employ. It is common (e.g., in statistical inverse problems) to encode PDE structure in a `forward map'
\begin{equation}\label{forward}
\G: \R^D \rightarrow L^2_\lambda(\mathcal{X},V),\quad \theta\mapsto \G(\theta),
\end{equation}
describing vector space $V$-valued regression fields $\G(\theta): \mathcal X \to V$ defined on some probability space $(\mathcal X, \lambda)$. The data vector $Z^{(N)}=((Y_i, X_i) \in  V \times \X, i=1, \dots, N)$ is generated via the regression equation $Y_i = \G(\theta)(X_i) + \varepsilon_i,$ with $\varepsilon_i$ standard i.i.d.~Gaussian noise in $V$ and $X_i \sim^{\iid} \lambda$ independently of the $\varepsilon_i$'s. This induces an `energy' (log-likelihood) associated to $\G$ by
\begin{equation}\label{likely}
\ell_N(\theta)=\ell_N(\theta\vert Z^{(N)}) = -\frac 12 \sum_{i=1}^N \left \vert Y_i -\G(\theta)(X_i) \right\vert_V^2,~~ \theta \in \R^D.
\end{equation}
The non-concavity of $\ell_N(\cdot|Z^{(N)})$ is due to the map $\G$ which we think of here as the composition of a non-linear `parameter to solution map' of some PDE (such as \eqref{pde} below) with a discretisation step that represents an infinite-dimensional parameter $\Phi$ by its approximation $\Phi_\theta$ from a finite-dimensional approximation space $E_D$ with (`Fourier-') coefficients $\theta \in \R^D$.

\smallskip

With $\ell_N$ from (\ref{likely}), the Gibbs-type measure \eqref{gibbs} is the Bayesian \textit{posterior distribution} for a Gaussian prior $\Pi$ arising from the data. Suppose $Z^{(N)}$ is generated from `ground truth' $\Phi_\true$ and let $\Phi_{\theta_{\true, D}}$ with $\theta_{\true,D} \in \R^D$ denote the best approximation  of $\Phi_\true$ from $E_D \cong \R^D$. Then properties of $\G$ near $\theta_{\true,D}$ will play a crucial role in our analysis. Very informally speaking our conditions on $\G$ will comprise the following:

\smallskip

\begin{enumerate}[label=(\alph*)]
\item \label{assumption1} (Global stability) $\G$ is uniformly bounded and Lipschitz on $\R^D$ and satisfies a stability estimate $\|\theta - \theta'\| \lesssim \|\G(\theta) - \G(\theta')\|^\gamma, \gamma>0,$ on bounded sets in $\R^D$.

\smallskip

\item \label{assumption2} (Local regularity) $\G$ is $C^3$ in a neighbourhood of $\theta_{\true,D}$, with $C^{3}$-norms of $\G$ growing at most polynomially in $D$.

\smallskip

\item \label{assumption3} (Local gradient stability) For all $\theta$ near $\theta_{\true,D}$, the gradient $\nabla \G(\theta)$ of $\G$ satisfies
\begin{equation}\label{linearstability}
\Vert \nabla \G(\theta)[h] \Vert_{L^2_\lambda}^2 \gtrsim D^{-\kappa_0} \Vert h \Vert_{\R^D}^2, ~~ h \in \R^D,~~\text{for some } \kappa_0 \ge 0.
\end{equation}
\end{enumerate}

\smallskip

The global requirements in \ref{assumption1} pose strong restrictions on $\G$ but  fall far short of assuming global log-concavity of the Gibbs measure (\ref{gibbs}). Such conditions can often be verified in concrete PDE settings, in particular global inverse continuity (or `stability') estimates for $\G$ have recently emerged as drivers of posterior consistency theorems in non-linear inverse problems see \cite{V13, NS17, N17, NVW18, AbNi19, MNP21, GN20, K21}, \textsection \ref{posthorn} below, and also \cite{N22} for a general account of the theory.  The regularity condition \ref{assumption2} provides sufficient smoothness of the local parameterisation of $\theta$ and is typically satisfied in PDE settings too. 

\smallskip

Assumption \ref{assumption3} ensures that the energy $\theta 
\mapsto \ell_N(\theta\vert Z^{(N)})$ is locally concave near $\theta_{\true,D}$ \textit{on average} under the distribution of $Z^{(N)}$ (see Condition \ref{goldfisch} below). While there is no general recipe to obtain such a stability estimate, it was demonstrated in \cite{NW20} how to use elliptic PDE theory to verify (\ref{linearstability}) in a representative example with a Schr\"odinger equation. Our main example to be introduced below shares some conceptual similarities and we explain these now in more detail, thereby also clarifying the role of the approximation spaces $E_D\cong \R^D$. Consider PDEs of the form
\begin{equation}\label{pde}
(\mathscr D + \Phi_\theta)u = 0 \text{ on } \mathcal{M},
\end{equation}
where $\mathcal{M}$ is a smooth manifold, $\mathscr D$ is a {\it known linear} differential operator and $\Phi_\theta$  is a (possibly matrix valued) potential defined on $\mathcal M$, indexed by the parameter $\theta \in \R^D$ to be determined. Subject to suitable boundary conditions under which  \eqref{pde} admits a unique solution $u=u_\theta$,  it is assumed that one can measure $u_\theta$  on a subset $\mathcal{X}\subseteq \mathcal{M}$, giving the forward map
$ \G(\theta) = u_\theta\vert_\X$. For instance we can consider

\smallskip

\begin{itemize}
\item A {\it Schr{\"o}dinger equation}: Here $\mathcal{M}$ is a bounded domain in $\R^d$ and $\mathscr D= \Delta$ is the Laplace-operator, say with Dirichlet boundary conditions.
\item A {\it transport equation}: Here $\mathcal{M}=\{(x,v)\in \R^2\times \R^2: \vert x \vert \le \vert v \vert =1\}$ (the unit sphere bundle over a disk) and $\mathscr D=v\cdot \nabla_x$ is the \textit{geodesic vector field}. 
\end{itemize}

\smallskip

Now, if we consider a variation $\theta + t h$ ($t\in \R$) in direction $h \in \R^D$ and denote $\dot u_{\theta} = \left( \d/\d t \vert_{t=0}\right)u_{\theta+t h}$ the resulting variation of the solution to \eqref{pde},  then {\it formally}
\begin{equation}\label{variatio}
 \nabla \G(\theta)[h] = \dot u_\theta = - ( \mathscr D+\Phi_\theta)^{-1} \big[ \dot \Phi_\theta[h] u_\theta \big], 
\end{equation}
and a stability estimate as in \eqref{linearstability}  can be expected to follow from {\it isomorphim properties} of $(\mathscr D+\Phi_\theta)^{-1}$ between suitable function spaces. 

\smallskip

In case of the Schr{\"o}dinger equation the powerful machinery of elliptic PDE's implies that $(\Delta + \Phi_\theta)^{-1}$ is an isomorphism from $L^2(\mathcal{M})$ onto the Sobolev space $H_0^2(\mathcal{M})=\{u \in H^2(\mathcal{M}): u \vert_{\partial \mathcal{M}} = 0 \}$ for (non-positive) potentials $\Phi_\theta$. Choosing the approximation space $E_D$ to equal the span of the first $D$ eigenfunctions of the Dirichlet-Laplacian (i.e.\,the eigen-basis corresponding to $\Phi_\theta=0$), the usual Weyl asymptotics imply that on $E_D$ the $H_0^2(\mathcal{M})$-norm is equivalent to $D^{-2/d} \Vert \cdot \Vert_{E_D}$. A duality argument then yields a stability estimate as in \eqref{linearstability} with $\kappa_0 = 4/d$; see Lemma 4.7 in \cite{NW20} for more details.

\smallskip

For the transport equation the operator $\II_\Phi \equiv (\mathscr D+\Phi)^{-1}$ can be interpreted as a linear X-ray transform on the disk $\DD=\{x\in \R^2:\vert x \vert <1\}$ that is `attenuated' by a potential $\Phi=\Phi_\theta$.  Transforms of this kind appear frequently as linearisations for non-linear geometric inverse problems \cite{IM20, PSU22} and are subject of current research.  Mapping properties of $\II_0$ acting on compactly supported functions are well known (see \cite[Thm.\,5.1]{Nat01}) and generally the operator $\mathbb I_0$ is understood to be $1/2$-regularising.  Taking into account the behaviour near $\partial \DD$, it was proved in \cite{Mon20} that $\II_0$ is an isomorphism from a Sobolev-type space $\tilde H^{-1/2}(\DD)$, defined in \eqref{zscaledef} below, onto $L^2_\lambda(\mathcal X)$. Extending recent developments in \cite{MNP20}, we will prove here that this isomorphism property extends to $\II_\Phi$ for a large class of $\Phi \neq 0$ (see Theorem \ref{zernikestability}). The natural discretisation spaces $E_D$ for which we can expect curvature as in (\ref{linearstability}) are defined in terms of the singular value decomposition (SVD) of $\II_0$, and we obtain in Theorem \ref{findimstability} a stability estimate as in \eqref{linearstability} with $\kappa_0=1/2$.

\smallskip

While the precise analysis in the transport case is very different from the elliptic techniques used for Schr\"odinger type equations, we want to emphasise the unifying aspects: The discretisation spaces $E_D$ are based on the SVD  of the linearisation of $\G$ at $\Phi_\theta=0$; then one uses an adequate isomorphism theorem for $\Phi_\theta\neq 0$ to prove the linear stability estimate in \eqref{linearstability} with $\kappa_0$ related to the `local' ill-posedness of the underlying inverse problem.

\smallskip

In Section \ref{genth} we state our general log-concave approximation theorem, while Section \ref{naxray} contains the application to $X$-ray transforms. Proofs and further technical developments are given in later sections.

\section{A general log-concave approximation theorem}\label{genth}

\subsection{Functional setting, spaces, and forward map}

Consider a measurable space $\mathcal Z = (\mathcal Z, \mathscr Z)$ equipped with a measure $\zeta$ on $\mathscr Z$, and a finite-dimensional vector space $W \simeq \mathbb R^{d_W}$ of dimension $d_W$. We denote the resulting $L^p$ spaces of $W$-valued (Borel) vector fields $f=(f_1, \dots, f_{d_W}), f_i \in L_\zeta^p(\mathcal Z, \mathbb R)$, by $L_\zeta^p(\mathcal Z, W) \equiv  \times_{i =1}^{d_W} L_\zeta^p(\mathcal Z, \mathbb R), 1 \le p<\infty,$  while $L^\infty(\mathcal Z, \R), L^\infty(\mathcal Z, W)$ denote the bounded measurable maps on $\mathcal Z$ equipped with the supremum norm $\|\cdot\|_\infty$.  On each coordinate space $L^2_{\zeta}(\mathcal Z, \mathbb R)$, we consider an $L^2_\zeta$-orthonormal system $$\{e_{n}: n \in \mathbb N _0\} \subset L^2_\zeta(\mathcal Z, \mathbb R) \cap L^\infty(\mathcal Z),~~\mathbb N_0 = \mathbb N \cup \{0\}.$$ On the product space $L_\zeta^2(\mathcal Z, W)$ we obtain a corresponding system $\{(e_{n, 1}, \dots, e_{n, d_W}), ~e_{n,i}=e_n ~\forall i, ~n \in \mathbb N_0\}$ that is orthonormal for the Hilbert space inner product $$\langle \cdot, \cdot \rangle_{L^2_\zeta(\mathcal Z, W)} = \sum_{i=1}^{d_W} \langle \cdot_i, \cdot_i \rangle_{L^2_\zeta(\mathcal Z, \mathbb R)},$$ where $\cdot_i$ selects the $i$-th coordinate of the vector field. Next, for a sequence of positive scalars $(\lambda_n : n \in \mathbb N_0)$ such that $0 < \lambda_n \le \lambda_{n+1} \to_{N \to \infty} \infty$, Sobolev-type spaces of real-valued functions for the system $(e_n: n \in \mathbb N_0)$ are defined as
\begin{equation}\label{scale}
\tilde H^s = \tilde H^s (\mathcal Z, \mathbb R) = \Big\{f \in L^2(\mathcal Z, \R): \|f\|_{\tilde H^s(\mathcal Z, \R)}^2 := \sum_{n \in \mathbb N_0} \lambda_{n}^{s} \langle f, e_{n} \rangle_{L_\zeta^2(\mathcal Z, \R)}^2 < \infty \Big\},~s \ge 0.
\end{equation}
Clearly  the imbedding $\tilde H^s \subset L^2(\mathcal Z, \mathbb R)$ is continuous for any $s \ge 0$. The spaces $\tilde H^s (\mathcal Z, W) = \times_{i=1}^{d_W} \tilde H^s (\mathcal Z, \mathbb R)$ of vector fields whose coordinate functions lie in $\tilde H^s(\mathcal Z, \mathbb R)$ are equipped with the natural (Hilbert-space) norm $\|\cdot\|_{\tilde H^s(\mathcal Z, W)}$ arising from inner product $$\langle f, g \rangle_{\tilde H^s} =\sum_{i=1}^{d_W} \sum_n \lambda_n^s \langle f_i, e_n \rangle_{L_\zeta^2} \langle g_i, e_n \rangle_{L_\zeta^2}.$$ We will sometimes just use the notation $\tilde H^s$ when the range $W$ or $\R$ is clear from the context. We note that when $\mathcal Z$ is a bounded domain in $\R^d$, these Sobolev `type' spaces can be `non-standard' depending on the boundary behaviour of the basis $e_n$. The theory that follows encodes all results in terms of regularity of the $\tilde H^s$ scale. 

\bigskip

Finite-dimensional subspaces $E_D \subset \tilde H^s(\mathcal Z, W)$ arise as the linear span of vector fields \begin{equation} \label{eddef3}
E_D = \spn \{(e_{n, 1}, \dots, e_{n, d_W}): n\le D_W \equiv D/d_W\}, 
\end{equation}
isomorphic (by Parseval's theorem) to an Euclidean space of dimension $D  = d_W \sum_{n \le D/d_W}1$. In what follows we assume for notational simplicity that $D$ is a positive integer multiple of $d_W$. The `Euclidean norm' on $E_D$, equal to the $L^2_\zeta$-norm, will be denoted by $\|\cdot\|_{E_D}$.

\begin{cond}\label{linsp} Consider an orthonormal system $(e_n: n \in \mathbb N_0)$ in $L_\zeta^2(\mathcal Z)$ such that
\begin{equation} \label{unifef}
\max_{0 \le n \le D_W}\|e_n\|_\infty \lesssim D^{\tau}, ~~\text{ some } \tau \ge 0, ~ D \in \mathbb N,
\end{equation}
and for some $d \in \mathbb N$ and constants $b_1, b_2$,
\begin{equation} \label{specdim}
b_1 n^{2/d} \le \lambda_n \le b_2 n^{2/d},~ n \in \mathbb N.
\end{equation}
Let $\Theta$ be a (measurable) linear subspace of $L^2_\zeta(\mathcal Z, W)$ such that $\cup_{D \in \mathbb N} E_D \subset \Theta \subset L^\infty(\mathcal Z, W)$.
\end{cond}

While in principle the constant $d$ is arbitrary in the preceding condition, in applications here it will model the dimension of $\mathcal Z$ and (\ref{specdim}) then corresponds to the asymptotic eigenvalue distribution of an elliptic second order differential operator on $\mathcal Z$. Alternatively the `eigen-pairs' $(e_n, \lambda_n)_{n \in \N}$ could arise from the Karhunen-Lo\'eve expansion of an infinite-dimensional Gaussian process prior over $\mathcal Z$ (cf.~also before (\ref{gprior}) below).

\smallskip
 
 When a suitable Sobolev imbedding is available we can take $\Theta= \tilde H^s(\mathcal Z, W)$ for appropriate $s >0$.   On $\Theta$ we consider a measurable `forward' map
\begin{equation} \label{fwdmap}
\mathscr G : \Theta \to L_\lambda^2(\mathcal X, V)
\end{equation}
 into another $L^2$-type space of vector fields; specifically, for an arbitrary \textit{probability} space $(\mathcal X, \mathscr X, \lambda)$ and a finite-dimensional normed vector space $(V, |\cdot|_V)$ with norm $|\cdot|_V$ arising from an inner product $\langle \cdot, \cdot \rangle_V$, we define as before by $L_\lambda^2(\mathcal X, V) \equiv \times_{i=1}^{d_V} L_\lambda^2(\mathcal X, \mathbb R)$ the resulting Hilbert spaces of $V$-valued vector fields, where $d_V=\dim(V)$. 

 \subsection{Gaussian regression model and Gaussian process priors}\label{bayesset}

We now consider a Gaussian regression model with regression maps $\{\mathscr G(\theta): \theta \in \Theta\}$, specifically we observe $(Y_i, X_i)_{i=1}^N$ arising from the equation,
\begin{equation}\label{model}
Y_i = \mathscr G(\theta)(X_i) + \varepsilon_i,~~~ \varepsilon_{i} \sim^{\iid} \mathcal{N}(0, I_V),~i=1, \dots, N,
\end{equation}
where the $X_i$ are i.i.d.~drawn from $\lambda$, and with random noise vectors $\varepsilon_i \sim \mathcal{N}(0,I_V)$ where $I_V$ is the identity matrix on $V$. The joint product law of $(Y_i, X_i)_{i=1}^N$  on $(V \times \mathcal X)^N$ will be denoted by $P_\theta^N=\otimes_{i=1}^N P_\theta$ where $P_\theta$ is the law of a generic observation $(Y,X)=(Y_1,X_1)$. The corresponding expectation operators are denoted by $E_\theta^N$ and $E_\theta$, respectively. The log-likelihood function for given data vector is then (up to additive constants) given 
\begin{equation} \label{llemp}
\ell_N: \Theta \to \R,~~\ell_N(\theta) = \ell_N(\theta, (Y_i, X_i)_{i=1}^N) =  -\frac{1}{2} \sum_{i=1}^N |Y_i - \mathscr G(\theta)(X_i)|_V^2,
\end{equation}
and we also define the log-likelihood `per observation'
\begin{equation} \label{llgen}
\ell: \Theta \to \R, ~~\ell(\theta)= \ell(\theta, (Y,X))= -\frac{1}{2} |Y- \mathscr G(\theta)(X)|_V^2.
\end{equation}
To construct an `$\alpha$-regular' high-dimensional Bayesian model, we consider random Gaussian series
$$\theta = N^{-d/(4\alpha+2d)} \theta', ~\text{ where }~\theta' = \sum_{i \le d_W} \sum_{n \le D_W} \lambda_{n}^{-\alpha/2} g_{n,i} e_{n,i} ,~~ g_{n,i} \sim^\iid\mathcal{N}(0,1),~\alpha>0,$$ and with $d$ from (\ref{specdim}). We denote by $\Pi$ the prior law $\mathcal L(\theta)$ on $\Theta$, i.e., the normal distribution 
\begin{equation}\label{gprior}
\Pi = \mathcal N\big(0, (N\delta_N^2)^{-1}\Lambda_{\alpha}\big) \text{ on } \R^D \simeq E_D,~~\text{where } \delta_N=N^{-\alpha/(2\alpha+d)},
\end{equation}
with covariance matrix $\Lambda_{\alpha} = \mathrm{diag}(\lambda_n^{-\alpha})$ and probability density $\pi=\d \Pi$. Assuming joint measurability of $(x,\theta) \to \G(\theta)(x)$ (implied by Condition \ref{ganzwien}b) below), the posterior distribution on $\Theta$ that arises from this prior is then given by 
\begin{equation}\label{postbus}
\Pi(B|(Y_i, X_i)_{i=1}^N) = \frac{\int_B e^{\ell_N(\theta)} d\Pi(\theta)}{\int_{E_D} e^{\ell_N(\theta)} d\Pi(\theta)},~~ B \subseteq E_D ~\text{ Borel},
\end{equation}
and has a Lebesgue-probability density function on $\R^D \simeq E_D$ equal to
$$\pi(\theta|(Y_i, X_i)_{i=1}^N) \propto \exp \Big\{ -\frac{1}{2} \sum_{i=1}^N |Y_i - \mathscr G(\theta)(X_i)|_V^2 - \frac{1}{2} N \delta_N^2\|\theta\|^2_{\tilde H^\alpha} \Big\},~~\theta \in E_D.$$

\subsection{Analytic hypotheses}\label{vonhintenwievonvorne}

We now formulate some general conditions on uniformly bounded forward maps $\mathscr G$ that will allow us to formulate our main theorems. These conditions are split into `global' and `local' types and we start with the global ones.

\begin{cond}\label{ganzwien}
For $\mathscr G: \Theta \to L^2_\lambda(\mathcal X,V)$ from (\ref{fwdmap}), assume the following:

a) [uniform boundedness]  There exists a constant $U_\mathscr G \ge 1$ such that 
$$\sup_{\theta \in \Theta, x \in \mathcal X}|\mathscr G(\theta)(x)|_V \le U_\mathscr G.$$

\smallskip

b) [global Lipschitz] There exists a fixed constant $C_\mathscr G \ge 1$ s.t.~for all $\theta, \theta' \in \Theta$
$$\|\mathscr G(\theta) - \mathscr G(\theta')\|_* \le C_\mathscr G \|\theta-\theta'\|_+,$$ for norms $\|\cdot\|_* = \|\cdot\|_{L^2_\lambda}, \|\cdot\|_\infty$, $\|\cdot\|_+ = \|\cdot\|_{L^2_\zeta}, \|\cdot\|_\infty$, respectively.

\smallskip

c) [inverse continuity modulus] For every $M$ there exists a constant $L'$ (that may depend on $M$) and $0 <\gamma \le 1$ s.t.~for all $D \in \mathbb N$, all $\delta>0$ small enough and the given $\alpha>0$,  $$\sup \Big\{\|\theta-\theta'\|_{L^2_\zeta}:~ \theta, \theta' \in E_D, \|\theta\|_{\tilde H^{\alpha}} + \|\theta'\|_{\tilde H^{\alpha}} \le M,  \|\mathscr G(\theta)-\mathscr G(\theta')\|_{L^2_\lambda} \le \delta \Big \}  \le L' \delta^\gamma.$$
\end{cond}

The first condition is often satisfied in view of `compactifying' or `energy-preserving' nature of the solution map $\mathscr G$ of an underlying PDE. Condition b) can be expected to hold for PDEs such as those discussed \textsection \ref{introduction} as long as $(\D + \Phi_\theta)^{-1}$ from (\ref{variatio}) has operator norms $\|\cdot\|_{L^p \to L^p}, p=2,\infty,$ uniformly bounded in $\Phi_\theta$. The third condition requires a `stability (inverse continuity) estimate' for the map $\theta \mapsto \mathscr G(\theta)$ on a bounded subset of $\tilde H^\alpha$. Condition \ref{ganzwien} is sufficient to prove  global posterior consistency properties (see \textsection \ref{posthorn}).

\medskip

We now turn to the `local' conditions near a hypothetical `ground truth' $\theta_\true \in \Theta$, in fact near its $L^2_\zeta$-projection $\theta_{\true,D}$ onto $E_D$; for radius $\eta>0$ to be chosen, set
\begin{equation}\label{eq:B:ass}
\mathcal B:= \big\{ \theta\in E_D: \|\theta-\theta_{\true,D}\|_{E_D}< \eta \big \}.
\end{equation}

The following condition is a `forward' regularity condition on the map $\mathscr G$. It is satisfied, for instance, as soon as $\mathscr G$ is $C^3$ on $\mathcal B$ with local $C^3$-norm constants growing at most \textit{polynomially} in dimension $D$. To formulate it let us define the following local $C^{2,1}$-norm for maps $F: \mathcal B \to V$, $F(\theta)=(F_1(\theta), \dots, F_{d_V}(\theta))$,
\begin{align}\label{c2lip}
\|F\|_{C^{2,1}(\mathcal B, V)} :=  \max_{1 \le k \le d_V} \Big\{&\|F_k\|_{\infty} + \|\nabla F_k\|_{L^\infty(\mathcal B, \R^D)} + \|\nabla^2 F_k\|_{L^\infty(\mathcal B, \R^{D\times D})} \\
&~+ \sup_{\theta, \theta' \in \mathcal B, \theta \neq \theta'}\frac{|\nabla^2 F_k(\theta) - \nabla^2 F_k(\theta')|_{\mathrm{op}}}{\|\theta - \theta'\|_{E_D}} \Big\},
\end{align}
where $\|\cdot\|_{L^\infty ( \mathcal B, U)}$ are the supremum norms of maps defined on $\mathcal B$ taking values in a normed vector space $U$, the space $\R^{D \times D}$ is equipped with the usual operator norm $|\cdot|_\mathrm{op}$, and $\nabla=\nabla_\theta$ and $\nabla^2$ denote the gradient and `Hessian' operator with respect to $\theta \in \R^D \simeq E_D$, respectively.

\begin{cond}[Local regularity]\label{u4} Let $\mathcal B$ be given in (\ref{eq:B:ass}) for some $\eta>0$ and suppose that for all $x\in\mathcal X$, the map $\theta \mapsto \mathscr G(\theta)(x)$ from (\ref{fwdmap}) is in $C^{2,1}(\mathcal B, V)$ and satisfies $\sup_{x \in \mathcal X}\big\|\mathscr G(\cdot)(x)\big\|_{C^{2,1}(\mathcal B, V)} \le c_2 D^{\kappa_2}$ for some $c_2 \ge 1$ and $\kappa_2 \ge 0$.
\end{cond}

For the next condition recall the real-valued map $\ell: E_D \to \mathbb R$ from (\ref{llgen}) and consider again the gradient operator $\nabla =\nabla_\theta$ w.r.t.~$\theta \in E_D$. The $D\times D$ Hessian matrix $E_{\theta_\true}[-\nabla^2\ell (\theta, (Y,X))]$ is symmetric and $\lambda_{\min}(A)$ will denote the smallest eigenvalue of a symmetric matrix $A$.

\begin{cond}[Local mean curvature]\label{goldfisch}
	Let $\mathcal B$ be given in (\ref{eq:B:ass}) for some $\eta>0$ and let $\ell: E_D \to \R$ be as in (\ref{llgen}). Assume that for some $c_0>0, c_1 \ge 1$, $\kappa_0, \kappa_1 \ge 0$ and all $D \in \mathbb N$,
			\begin{equation}\label{eq:ass:lb}
		\inf_{\theta\in \mathcal B} \lambda_{\min}\Big(E_{\theta_\true}[-\nabla^2\ell (\theta,(Y,X)))]\Big)\ge c_0 D^{-\kappa_0}~ \text{ and }
		\end{equation} 
		\begin{equation}\label{eq:ass:ub}
		\sup_{\theta\in \mathcal B}\Big[ |E_{\theta_\true}\ell(\theta, (Y,X))|+\|E_{\theta_\true}[\nabla \ell(\theta, (Y,X)) ] \|_{\R^D}+\|E_{\theta_\true}[\nabla^2\ell (\theta, (Y,X))]\|_\mathrm{op} \Big] \le c_1 D^{\kappa_1}.
		\end{equation}
\end{cond}

As (\ref{eq:ass:ub}) is an `average' version of Condition \ref{u4}, it will be automatically satisfied for some $\kappa_1$, but we still list it as a separate condition to allow for $\kappa_1<\kappa_2$. The condition (\ref{eq:ass:lb}) is closely related (see also \eqref{seealsothis}) to a `gradient' stability estimate 
\begin{equation}\label{zauberstab}
\|\nabla \mathscr G(\theta_\true)^Tv\|_{L^2_\zeta}^2 \gtrsim D^{-\kappa_0} \|v\|_{E_D}^2,
\end{equation}
and in fact requires that this estimate propagates to a neighbourhood $\mathcal B$ of $\theta_\true$.  Under the above conditions it was shown in \cite{NW20} (see Lemma \ref{youtube} below) that the negative log-likelihood $-\ell_N(\theta)$ is strongly convex on $\mathcal B$ with $P_{\theta_\true}^N$-probability of order $1-O(\exp\{-CND^{-2\kappa_0 - 4 \kappa_2}\})$,  a key starting point of the analysis. We will require that the preceding conditions hold for radius $\eta$ of $\mathcal B$ of the following order of magnitude.

\begin{cond}\label{bock} Suppose Conditions \ref{u4} and \ref{goldfisch} hold for a choice of radius $\eta$ that satisfies 
\begin{equation} \label{laboum}
\eta \ge (\log N) \tilde \delta_N, ~\text{ where } ~\tilde \delta_N \equiv (\log N)\max\big(\delta_N^\gamma,  D^{\kappa_0/2} \delta_N \big)
\end{equation}
where $\delta_N$ is as in (\ref{gprior}), $\gamma$ is from Condition \ref{ganzwien}c), and $\kappa_0$ from (\ref{eq:ass:lb}).
\end{cond}

The radius $\eta$ required in our main theorem below is thus determined by the `local' and `global' ill-posedness of the map $\G$, measured by the parameters $\kappa_0$ and $\gamma$, respectively.

\subsection{The log-concave approximation theorem}
\label{genapprox}
\smallskip

Broadly speaking the results in this section are modelled after the proof of Theorem 4.14 in \cite{NW20} where the forward map $\G$ arises from a steady state Schr\"odinger equation. The main additional feature accounted for here is to deal with the quantitative interaction of the `local' and `global' ill-posedness of a general map $\G$ described by the parameters $\kappa_0, \gamma$ from (\ref{bock}) and the radius $\eta$ of the region where average curvature holds. In the Schr\"odinger model the `local' ill-posedness dominates which simplifies the proof, but for other inverse problems the various high-dimensional regimes occurring can lead to different outcomes. We note that while \cite{NW20} do provide a concentration result for the empirical Hessian of $\ell_N$ for general $\G$, their log-concave approximation theorem is restricted to the particular Schr\"odinger forward map they consider --  Theorem \ref{onetrickpony} below is thus the first result of this kind for general $\G$. 

\subsubsection{The surrogate posterior measure}\label{star}

\smallskip

Let $\theta_\init$ be any fixed vector in $E_D$ such that 
\begin{equation}\label{initioadabsurdum}
\|\theta_\init-\theta_{\true,D}\|_{E_D} \le \eta/8
\end{equation}
for some $\eta>0$ from Condition \ref{bock}. For the purposes of log-concave approximations one may well take $\theta_\init=\theta_{\true,D}$ but for applications to bounds on mixing times of `feasible' MCMC type algorithms, an explicit initialiser $\theta_\init$ has to be constructed. 

\smallskip

We follow Definition 3.5 in \cite{NW20} and require two auxiliary functions, $g_\eta$ (globally convex) and $\alpha_\eta$ (cut-off function): For some smooth and symmetric (about $0$) function
$\varphi:\R\to [0,\infty)$ satisfying $\text{supp}(\varphi)\subseteq [-1,1]$ and $\int_{\R}\varphi(x)dx=1$, let us define the mollifiers $\varphi_h(x):=h^{-1}\varphi(x/h), h>0$. Then, we define the functions $\tilde \gamma_\eta, \gamma_\eta: \R \to \R$ by
\begin{equation}\label{eq:gamma:tilde}
\begin{split}
\tilde \gamma_\eta (t)~&:=\begin{cases}
0 ~~&\text{if}~~~ t<5\eta/8,\\
(t-5\eta/8)^2 ~~~&\text{if}~~~ t \ge 5\eta/8,
\end{cases}\\
\gamma_\eta (t)~&:=\big[\varphi_{\eta/8}\ast \tilde \gamma_\eta\big](t),
\end{split}
\end{equation}
where $\ast$ denotes convolution. Further, let  $\alpha: [0,\infty)\to [0,1]$ be smooth and satisfy $\alpha (t)=1$ for $t\in [0,3/4]$ and $\alpha(t)=0$ for $t\in [7/8,\infty)$. We then define functions $g_\eta: E_D \to [0,\infty)$ and $\alpha_\eta: E_D \to [0,1]$ as
\begin{equation}\label{eq:alphaeta:def}
g_\eta(\theta):=\gamma_\eta(\|\theta-\theta_\init\|_{E_D}),~~\alpha_\eta(\theta)=\alpha\big(\|\theta-\theta_\init\|_{E_D}/\eta \big).
\end{equation}
Now for $\ell_N$ as in (\ref{llemp}) and $K$ to be specified, define the `surrogate' likelihood function 
\begin{equation} \label{surrogate}
\tilde \ell_N(\theta) = \alpha_\eta (\theta) \ell_N(\theta) - K g_\eta(\theta),~ \theta \in E_D,
\end{equation}
which `convexifies' $-\ell_N$ in the `tails'. The resulting renormalised Gibbs-type probability measure $\tilde \Pi(\cdot|(Y_i, X_i)_{i=1}^N)$ then has probability density 
\begin{equation}\label{surro}
\tilde \pi(\theta|(Y_i, X_i)_{i=1}^N) = \frac{e^{\tilde \ell_N(\theta)}\pi(\theta)}{\int_{E_D}e^{\tilde \ell_N(\theta)}\pi(\theta)},~~\theta \in E_D.
\end{equation}

\subsubsection{Main result}

\smallskip

The following theorem establishes that $\tilde \Pi(\cdot|(Y_i, X_i)_{i=1}^N)$ is strongly log-concave with high probability, and that moreover the posterior measure $\Pi(\cdot|(Y_i, X_i)_{i=1}^N)$ is very well approximated in Wasserstein distance from (\ref{waterstone}) by $\tilde \Pi(\cdot|(Y_i, X_i)_{i=1}^N)$. `In probability' here means for $(Y_i, X_i)_{i=1}^N \sim P_{\theta_\true}^N$ where  $\theta_\true \in \Theta$ is well approximated by its $L^2_\zeta$-projection $\theta_{\true,D}$ onto $E_D$;
\begin{equation}\label{bias}
\|\mathscr G(\theta_\true)- \mathscr G(\theta_{\true,D})\|_{L^2_\lambda} \le \delta_N/2u,
\end{equation}
where $u=2(U_\mathscr G +1)$. The last condition is a natural `bias requirement' and for ground truths $\theta_\true \in \tilde H^\alpha(\mathcal Z)$ will require `high dimensions', i.e., that $D \simeq N^\beta$ diverges, where $\beta=\beta(\alpha,d)>0$.  Also, as the nature of the following theorem is `local' it is natural to assume $\eta \le 1$, which also simplifies several technical bounds in the proofs. 

\smallskip

The Wasserstein distance $W_2$ between probability measures $\mu$ and $\nu$ on $\R^D$ is 
\begin{equation}\label{waterstone}
W_2^2(\mu,\nu) = \inf_\pi \int_{\R^D\times \R^D} \vert \theta - \theta' \vert^2 \d \pi(\theta,\theta'),
\end{equation}
where the infimum is taken over all couplings $\pi$ of $\mu$ and $\nu$.

\begin{thm} [Log-concave Wasserstein-approximation] \label{onetrickpony}
Let $(e_n: n \in \mathbb N_0), \tau, d$ and $\Theta$ be as in Condition \ref{linsp}. Let the posterior distribution $\Pi(\cdot|(Y_i, X_i)_{i=1}^N)$ from (\ref{postbus}) arise from the Gaussian process prior $\Pi$ in (\ref{gprior}) and data $(Y_i, X_i)_{i=1}^N \sim P_{\theta}^N$ from (\ref{model}). Suppose $\mathscr G$ from (\ref{fwdmap}) satisfies Condition \ref{ganzwien} for some $0 < \gamma \le 1$ as well as Condition \ref{bock} for some $\eta \le 1$.  Assume further that $D \le A_1 N^{d/(2\alpha+d)}$ for some $A_1 < \infty$ and that $\theta_\true \in \tilde H^\alpha \cap \Theta$ satisfies (\ref{bias}), for
\begin{equation} \label{towardsinfinity}
\alpha > d, ~\alpha \ge d\big(\tau+\frac{1}{2}\big).
\end{equation}
Let the surrogate posterior measure $\tilde \Pi(\cdot|(Y_i, X_i)_{i=1}^N)$ be as in (\ref{surro}) for the given value of $\eta$ and any $K \ge N (\log N) D^{\kappa_1}/\eta^2$. We then have the following:

\medskip

A) For all $\theta  \in \{\theta \in E_D: \|\theta-\theta_{\true,D}\|_{E_D} \le 3\eta/8\}$ we have $\tilde \ell_N(\theta) =\ell_N(\theta)$ and hence also $$\nabla \log \tilde \pi(\cdot|(Y_i, X_i)_{i=1}^N)(\theta)=\nabla \log \pi(\cdot|(Y_i, X_i)_{i=1}^N)(\theta).$$ 

\medskip

B) On an event $\mathcal E_N \subset (V \times \mathcal X)^N$ of $P_{\theta_\true}^N$-probability at least $1 - O(e^{-CND^{-2\kappa_0 - 4 \kappa_2}})$ for some $C>0$, the probability measure $\tilde \Pi(\cdot|(Y_i, X_i)_{i=1}^N)$ is strongly log-concave on $E_D$, specifically $$\inf_{\theta \in E_D} \lambda_\min(-\nabla^2 \log \tilde \pi(\theta|(Y_i, X_i)_{i=1}^N)) \ge \frac{c_0}{2} N D^{-\kappa_0},$$ and moreover $\nabla \tilde \ell_N$ is globally Lipschitz on $E_D$ with Lipschitz constant $7K$. 

\medskip

C) On an event $\mathcal E_N \subset (V \times \mathcal X)^N$ of $P_{\theta_\true}^N$-probability at least $1 - O(e^{-bN\delta_N^2} - e^{-CND^{-2\kappa_0 - 4 \kappa_2}})$ for some $b>0,C>0$, we have the Wasserstein approximation $$W^2_2\big(\Pi(\cdot|(Y_i, X_i)_{i=1}^N), \tilde \Pi(\cdot|(Y_i, X_i)_{i=1}^N) \big) \leq e^{-N^{d/(2\alpha+d)}}.$$
\end{thm}

\smallskip

The theorem is `non-asymptotic' in the sense that whenever the hypotheses hold for pairs $(D,N) \in \mathbb N^2$, then so do the conclusions (which are informative only when $N \to \infty$).

\smallskip

Some main ideas underpinning the proof of this theorem are as follows: By Condition \ref{goldfisch} and tools from concentration of measure, the function $\ell_N$ and hence also the posterior distribution are \textit{locally} log-concave near $\theta_{\true,D}$, with high probability. Then our proof shows and exploits the fact that the surrogate posterior density $\tilde \pi(\cdot |(Y_i, X_i)_{i=1}^N)$ has as unique maximiser the \textit{maximum a posteriori (MAP) estimate} $\hat \theta_{\MAP}$, again with high probability -- see (\ref{score}). Next we use that under Condition \ref{ganzwien} the posterior measure is statistically consistent under $P_{\theta_\true}^N$ (see \textsection \ref{posthorn}) and hence charges most of its mass precisely to the neighbourhood of $\theta_{\true,D}$ where $\Pi(\cdot|(Y_i, X_i)_{i=1}^N)$ is log-concave. Making the last statement quantitative in Part C) is the main challenge in the proof.

\smallskip

The log-concave approximations provided by the previous theorem are qualitatively different from commonly used Gaussian `Laplace approximations' for numerical integration of Gibbs measures (\ref{gibbs})  (see \cite{W01} and also \cite{BC09, LSW17, SSW20, HK20} for recent references), which are related to Bernstein-von Mises theorems (\cite{BC09, CN13, CN14, R17, N17, NS19, NR20, MNP20} and references therein). In high-dimensional situations these hold only under restrictive analytical conditions (e.g., \cite{NP22}) and in particular with slower convergence rates than the approximations obtained in Theorem \ref{onetrickpony}. The added flexibility of being able to approximate from the (infinite-dimensional) class of \textit{log-concave measures} instead of just from normal distributions is essential in our context. See also Remark 5.1.4 in \cite{N22} for more discussion.

\smallskip

While the surrogate posterior $\tilde \Pi(\cdot|D_N)$ proposed here is constructive (up to selection of $\theta_{\init}$), it involves choices of constants that may not be practical. However, for the key applications of Langevin-type MCMC methods, this is immaterial, as the gradients of the log-posterior and the surrogate log-posterior coincide near $\theta_{\init}$ by Part A) of the preceding theorem. If a Markov chain such as the ones to be discussed below is started in this region, it will rarely (if ever) exit the region where the bulk of the posterior mass lies and hence not enter the region where the surrogate construction and the choice of its constants become relevant. While the last argument can be made rigorous via Markov chain hitting times, we abstain from giving the details here.

\subsection{Applications to Langevin mixing times}\label{mixingsection}

Part A) of Theorem \ref{onetrickpony} implies that queries of any gradient based algorithm exploring the posterior measure $\Pi(\cdot|(Y_i, X_i)_{i=1}^N)$ will remain unchanged under the surrogate approximation as long as $\theta$ remains near the initialiser $\theta_\init$. Part B) says that the surrogate is strongly log-concave and we can hence expect (following work by \cite{Dal17, DM17, DuMo19}) gradient flow based MCMC algorithms to mix very rapidly towards $\tilde \Pi(\cdot|(Y_i, X_i)_{i=1}^N)$. Specifically one can combine Theorem \ref{onetrickpony} with results in \textsection 3 in \cite{NW20} (cf.~also Remark 3.11 in that paper) to give explicit bounds on the Wasserstein-mixing time of the following Langevin scheme: Let $\tilde \pi_N \equiv \tilde \pi(\cdot|(Y_i,X_i)_{i=1}^{N})$ denote the probability density of $\tilde \Pi(\cdot|(Y_i,X_i)_{i=1}^{N})$. For auxiliary variables $\xi_k \sim^{\iid} \mathcal{N}(0, I_{D \times D})$ and step size $\gamma$ to be chosen, define a Markov chain $(\vartheta_k) \subset\R^D$ as
\begin{align}\label{langos}
\vartheta_0 & = \theta_{\init} \notag \\
\vartheta_{k+1} & = \vartheta_k + \gamma \nabla \log \tilde \pi_N (\vartheta_k)  + \sqrt{2 \gamma} \xi_k,~~k=0,1,2,\dots,
\end{align}
initialised at $\theta_\init$ from (\ref{initioadabsurdum}). If we set 
\begin{equation}
\bar m = ND^{-\kappa_0} + D^{2\alpha/d} N \delta_N^2, ~~\Lambda = \frac{N (\log N) D^{\kappa_1} }{\eta^2} + D^{2\alpha/d} N \delta_N^2
\end{equation}
 and $\epsilon>0$ is given, then a choice of \begin{equation}\label{stepsize}
\gamma \simeq \min\Big(\frac{\epsilon^2 \bar m^2}{D \Lambda^2}, \frac{\epsilon \bar m^{3/2}}{D^{1/2} \Lambda^2} \Big)
\end{equation}
allows to obtain, from Theorem 3.7 in \cite{NW20}, on the event $\mathcal E_N$ from (and under the conditions of) Part C) of Theorem \ref{onetrickpony}, the bound
\begin{equation}
W_2^2(\mathcal L(\vartheta_k), \Pi(\cdot|(Y_i, X_i)_{i=1}^N) \le e^{-N^{d/(2\alpha+d)}} + \epsilon
\end{equation}
after a polynomial runtime $k \approx O(N^{\xi_1} D^{\xi_2}\epsilon^{-\xi_3})$ for some $\xi_1, \xi_2, \xi_3>0$.

\section{Main results for non-Abelian X-ray transforms}\label{naxray}

We now describe here the concrete results we obtain when the map $\G$ arises from a \textit{non-Abelian X-ray transform}. While we review the essential details, we refer to \cite{PSU12, MNP21} for general background and references, and adopt their general geometric notation.

\smallskip

Throughout, for $m \ge 2$ we denote by $SO(m) = \{A\in \R^{m\times m}: A^T A = \id, \det A =1 \}$ the special orthogonal group, with Lie-algebra 
$\so(m)=\{A\in \R^{m\times m}: A^T=-A\}$ consisting of skew-symmetric $m\times m$-matrices. We equip the vector spaces $\so(m)$ and $\R^{m\times m}$ with the Frobenius-norm $\vert A \vert_F = \trace(A^TA)^{1/2}$.

\smallskip

Let $\DD\subset \R^2$ be the open unit disk with  boundary  $\partial \DD$ and so called `influx and outflux' boundaries $\partial_+ S\bar \DD$ and $\partial_-S\bar \DD$, defined  by
\begin{equation}
\partial_\pm S\bar \DD = \{(x,v)\in \partial \DD\times S^1: \mp x\cdot v \ge 0\},
\end{equation} 
where $S^1=\{v\in \R^2:\vert v \vert=1\}$ is the unit circle.
Then $\partial_+S\bar \DD$ parametrises the set of line segments $\gamma_{x,v}(t)=x+tv$ in $\DD$, $0\le t \le \tau(x,v)$, where $\tau(x,v)= -2 x\cdot v$ is the hitting time with $\partial \DD$. The space $\partial_+S\bar \DD$ is diffeomorphic to a cylinder and will be equipped with its normalised area measure $\lambda$, which is given in fan-beam coordinates $(\alpha,\beta)\in [-\pi/2,\pi/2]\times \R/ 2\pi \Z$ (described in \textsection \ref{fanbeam}) by
\begin{equation}\label{measurelambda}
\lambda =(2\pi^2)^{-1} \d \alpha \d \beta.
\end{equation}
Let $C(\bar \DD,\so(m))$ be the space of $\so(m)$-valued maps on $\bar \DD$ that are continuous up to the boundary. The \textit{non-Abelian X-ray transform} is the non-linear map
\begin{equation}\label{yetagain}
\Phi\mapsto C_\Phi,~~~C(\bar \DD,\so(m)) \rightarrow L^2_\lambda(\partial_+S\bar \DD,\R^{m\times m}),
\end{equation}
which associates to a matrix field $\Phi$ its {\it scattering data} $C_\Phi$, defined as follows: for every $(x,v)\in \partial_+S\bar \DD$ we can solve a matrix ODE along the line segment $\gamma_{x,v}$,
\begin{equation}
\frac{\d}{\d t} u(t) + \Phi(\gamma_{x,v}(t)) u(t) = 0,~~~~u(\tau(x,v)) = \id,
\end{equation}
formulated as end-value problem by convention; we then define
\begin{equation}
C_\Phi(x,v):=u(0)\in SO(m)
\end{equation}
as the resulting initial value. By standard ODE theory, $C_\Phi$ is a continuous function of $(x,v)\in \partial_+S\bar \DD$ and thus can be viewed as element in the Hilbert space $L^2_\lambda(\partial_+S\bar \DD,\R^{m\times m})$.

\smallskip

Finite dimensional approximation spaces $E_D\subset C(\bar \DD,\so(m))$  are spanned by the classical {\it Zernike polynomials}, which are part of an SVD of the derivative of $\Phi\mapsto C_\Phi$ at $\Phi=0$ \cite{KaBu04}. Postponing a full definition of $E_D$ to (\ref{eddef2}), we mention here that if $D\in \N$ is of the form $D= \dim \so(m) \cdot (D'+1)(D'+2)/2$ for some $D'\in \N_0$, then
\begin{equation}\label{eddef}
E_D=E_D(\DD,\so(m))=\left\{ \begin{array}{ll}
\Phi:\DD\rightarrow \so(m): \text{Each matrix entry } \Phi_{ij}(x_1,x_2)  \\
 \text{is a polynomial in } x_1,x_2 \text{ of order } \le D'
\end{array}
\right \}.
\end{equation}
The space $E_D$ is a $D$-dimensional subspace of  $L^2(\DD, \R^{m \times m})$ and (after choosing an orthonormal basis of $\so(m)$)  can be naturally identified with $\R^D$; we sometimes make this identification explicit via the map
\begin{equation}
\R^D\rightarrow E_D,\quad \theta \mapsto \Phi_\theta
\end{equation}
and refer to $\theta$ as the `Fourier coefficients' of $\Phi_\theta$. 

\smallskip

What precedes  fits into the framework of the \textsection \ref{genth} by making the following choices:  $\mathcal Z = \DD$ and $\zeta$ is the Lebesgue measure; $\mathcal X = \partial_+S\bar \DD$ and  $\lambda$ is the measure from (\ref{measurelambda});
\begin{equation}\label{thetaxray}
\Theta = C(\bar \DD,\so(m)),  \quad W= \so(m), \quad V= \R^{m\times m};
\end{equation}
the approximation spaces $E_D$ from (\ref{eddef3}) are as in (\ref{eddef}), with spanning set $(e_n: n \in \mathbb N_0)$ given by  Zernike polynomials (see \textsection \ref{realval} for details); finally, identifying matrix fields $\Phi=\Phi_\theta \in E_D$ with the parameter $\theta \in \R^D$, the forward map $\G$ is defined as
\begin{equation}
\G(\theta) = C_{\Phi_\theta},\quad \theta\in \R^D \simeq E_D.
\end{equation}

\smallskip

{\bf Local gradient stability for the X-ray transform.} In order to apply the general theory from \textsection \ref{genth}, the analytical hypotheses from \textsection \ref{vonhintenwievonvorne} need to be verified. In the present setting, the main non-trivial properties are a {\it global stability estimate} (Condition \ref{ganzwien}c)), as well as a {\it local gradient stability estimate} (Condition \ref{goldfisch} and equation \eqref{zauberstab}). The former has already been obtained in \cite{MNP21} and we now discuss the latter, local condition.

\smallskip

The Gateaux derivative of the non-Abelian X-ray transform  $\Phi\mapsto C_\Phi$ in \eqref{yetagain} at some $\Phi\in C(\bar \DD,\so(m))$ and in direction  $h\in C(\bar \DD,\so(m))$ is denoted by
\begin{equation}\label{gateaux}
\II_\Phi h := \frac{\d}{\d t}\Big \vert_{t=0} C_{\Phi+t h} \in L^2_\lambda(\partial_+S\bar \DD,\so(m)).
\end{equation}
It can be shown to extend as bounded linear map between $L^2$-spaces,
\begin{equation*}
\II_\Phi\colon L^2(\DD,\so(m))\rightarrow L^2_\lambda(\partial_+S\bar \DD),
\end{equation*}
which can be regarded as linear X-ray/Radon transform with a $\Phi$-dependent weight -- see \textsection \ref{analysissection} for details.  If $\Phi=\Phi_\theta\in E_D$, then $\II_\Phi\vert_{E_D}$ can be viewed as gradient $\nabla \G(\theta)$ of the forward map  $\G$ at $\theta$ -- in particular, the gradient stability condition (\ref{zauberstab}) is verified for the non-Abelian X-ray transform by the following theorem. 
\begin{thm}\label{findimstability} Suppose $\Phi:\DD\rightarrow \so(m)$ is supported in a compact set $K\subset \DD$ and has Sobolev regularity $H^\alpha( \DD)$ for some $\alpha>5$. Then
\begin{equation}\label{weylz}
\Vert \II_\Phi h \Vert_{L^2_\lambda(\partial_+S\bar \DD)}^2 \ge   c D^{-1/2} \cdot  \Vert h \Vert_{L^2(\DD)}^2 \quad  \text{ for all } h\in E_D,
\end{equation}
for a constant $c=c(m, K, M )>0$, where $M$ is any upper bound $M \ge \Vert \Phi \Vert_{H^\alpha(\DD)}$.
\end{thm}

For $\Phi=0$ the inequality (\ref{weylz}) follows directly from the distribution of the singular values of $\mathbb I_0$ from \cite{KaBu04}, but remarkably this stability property extends to all sufficiently regular $\Phi$ (where the Zernike polynomials do \textit{not} constitute the singular vectors for $\II_\Phi$).

\smallskip

The proof of the theorem (see Remark \ref{toolate}) relies on recent developments in \cite{Mon20} and \cite{MNP20} and follows from the infinite dimensional Theorem \ref{zernikestability}, which constitutes our main analytical contribution.   As already discussed after \eqref{variatio}, the underlying isomorphism property can be regarded as X-ray version of the well known result from elliptic PDE theory that, for natural boundary conditions, inverse Schr\"odinger operators $(\Delta + \Phi)^{-1}$ have mapping properties comparable to those of the standard inverse Laplacian $\Delta^{-1}$.

\begin{rem}\label{nonsharp}
We emphasise here that the exponent $1/2$ ($=\kappa_0$ in Condition \ref{goldfisch}) in the preceding theorem is sharp in that it matches exactly the ill-posedness of the problem (e.g., \cite[Theorem 5.1]{Nat01}). A non-sharp result  (with $\kappa_0=1+\epsilon$ for arbitrary $\epsilon>0$) can be obtained directly from the global stability estimate in \cite{MNP21} (and without appealing to Theorem \ref{zernikestability})---see Remark \ref{nonsharp2} below.

\end{rem}

In  \textsection \ref{finalsection} we show that Conditions \ref{ganzwien} up to \ref{bock} are indeed satisfied for the non-Abelian X-ray transform, with precise parameters exhibited at the start of that section.

\smallskip

\textbf{Computation time bounds and approximation of the Gibbs measure.} We now state our main results on log-concave approximation and polynomial time sampling of posterior measures in the context of the non-Abelian X-ray transform.

\smallskip

Analogous to \textsection \ref{bayesset} and with the choices made surrounding (\ref{thetaxray}), we observe data $Z^{(N)}=\left((Y_i, X_i,V_i)\in \R^{m \times m} \times \partial_+S\bar \DD:i=1,\dots, N\right)$ from the regression model
\begin{equation}\label{model0}
Y_i=C_\Phi(X_i,V_i)+\varepsilon_i,  \qquad i=1,\dots, N,
\end{equation}
where $(X_i,V_i)\sim \lambda$ are independently drawn directions in  $\partial_+S\bar \DD$ (corresponding to random line segments in $\DD$) and $\varepsilon_i=(\varepsilon_{ijk}:1\le j,k\le m)$ with $\varepsilon_{ijk}\sim^\iid \mathcal{N}(0,1)$---see also \cite{MNP19a}, where this noise model for the non-Abelian X-ray transform is discussed in detail. The law of $Z^{(N)}$ is now denoted $P_\Phi^N$ (or simply $P_\Phi$ if $N=1$).
The prior $\Pi$ on $\Theta=C(\bar \DD,\so(m))$ is chosen as in \eqref{gprior}, with covariance matrix $\Lambda_\alpha$ determined by the singular values $\lambda_0\le \lambda_1\le \dots$ of $\II_0$ (cf.~\eqref{singularvalues})
The bias condition  \eqref{bias} now takes the form
\begin{equation} \label{phi0cond2}
\Vert C_{\Phi_\truep} -C_{\Phi_{\truep,D}} \Vert_{L^2_\lambda(\partial_+S\bar \DD)} \le \frac{\delta_N}{2 \sqrt{m}+ 1}, \quad \delta_N=N^{-\alpha/(2\alpha+2)}.
\end{equation}

\smallskip

The general Wasserstein approximation result from Theorem \ref{onetrickpony}, applied to the present setting, yields the following theorem.

\begin{thm}\label{xraypony} Let $\alpha >5$ and suppose that $D\le A_1 N^{1/(\alpha+1)}$ for some $A_1>0$. Let $\Pi(\cdot \vert Z^{(N)})$ be the posterior distribution on $\R^D$ described in \textsection\ref{bayesset}, with
$Z^{(N)}$ following model \eqref{model0}. Suppose the true potential $\Phi_\truep \in H^\alpha(\DD, \so(m))$ has compact support in $\DD$ and obeys \eqref{phi0cond2}. Then there exists a log-concave `surrogate posterior' $\tilde \Pi(\cdot \vert Z^{(N)})$ on $\R^D$ which approximates $\Pi(\cdot\vert Z^{(N)})$ in the sense that 
\begin{equation}
W_2^2\big(\tilde \Pi(\cdot\vert Z^{(N)}),\Pi(\cdot\vert Z^{(N)})\big) \le \exp\left(- N^{1/(\alpha +1)}\right)
\end{equation}
on an event $\mathcal{E}_N\subset\left(\R^{m \times m} \times \partial_+S\bar \DD\right)^N$ of probability at least
\begin{equation}\label{event0}
P_{\Phi_\truep}^N(\mathcal{E}_N) \ge 1 - O\big(\exp(-b N^{1/(\alpha+1)}) - \exp(-CND^{-7})\big),
\end{equation}
where $b>0$ and $C>0$.
\end{thm}

Note that for the maximal permitted dimension $D \simeq N^{1/(\alpha+1)}$, the stronger condition $\alpha>6$ is required for the last term in (\ref{event0}) to converge to zero. 

\smallskip

For an explicit construction of the `surrogate posterior' we refer to \textsection \ref{genapprox} above, with `base point' $\theta_{\init}=\theta_{\star, D}$ and choices of constants exhibited in the proof of Theorem \ref{xraypony} (see also  \cite[\textsection 3.6]{NW20}). In fact we obtain quantitative bounds on the log-concavity of $\tilde \Pi(\cdot|Z^{(N)})$ (as in Thm.~\ref{onetrickpony}B) with $\kappa_0=1/2$).


\smallskip

If a feasible initialiser $\theta_{\init}$ (in place of the infeasible $\theta_{\init}=\theta_{\true, D}$) is available, then we can use the gradient-based Langevin Markov chain $(\vartheta_k) \subset\R^D$ from \textsection \ref{mixingsection} to efficiently sample from the posterior $\Pi(\cdot\vert Z^{(N)})$; precisely, we obtain:

\begin{thm}\label{xraypony2} 
Suppose the hypotheses of the previous theorem hold and that there exists an initialiser $\theta_\init$ with $\Vert \theta_\init - \theta_{\true,D} \Vert_{\R^D} \le D^{-2}/\log N$.
Given any precision level $\epsilon>0$, there is an appropriate step size $\gamma=\gamma(\epsilon,D,N)$  such that the law of the iterates $\vartheta_k$ from (\ref{langos}) satisfy, on the events $\mathcal E_N$ from (\ref{event0})
\begin{equation}
W_2^2(\mathcal L(\vartheta_k), \Pi(\cdot|(Z^{(N)})) \le e^{-N^{1/(\alpha+1)}} + \epsilon,\qquad k\ge k_\mix(\epsilon,D,N),
\end{equation}
with mixing time of order  $ k_\mix(\epsilon,D,N)=O\left((1/\epsilon)^{b_1}D^{b_2}N^{b_3}\right)$ for some $b_1, b_2, b_3 > 0$.
\end{thm}

A further consequence is that posterior mean vectors $E[\theta|Z^{(N)}] \in \R^D$ (which are the subject of the statistical consistency results in \cite{MNP21}) can be computed in polynomial time by ergodic MCMC averages along the chain $\vartheta_k$, see Corollary 3.9 in \cite{NW20} for details. Computation time bounds for optimisation based `MAP'-estimates (see Section \ref{posthorn}) could also be obtained from our proofs (as in \cite{NW20}, Thm.2.8).

\smallskip

The question of existence of a suitable initialiser $\theta_\init$ in the X-ray problem (with $D \to \infty$) poses an intriguing challenge that is beyond the scope of the present paper. We note that in high-dimensional settings as considered here, `worst case initialised' MCMC algorithms can in general not be expected to achieve polynomial mixing times, see, e.g., \cite{BWZ20} for such a result, albeit in a different setting.

\section{Mapping properties of X-ray transforms in Zernike scales}\label{analysissection}

The results on the non-Abelian X-ray transform from \textsection \ref{naxray} rely on novel forward- and stability estimates for its linearisation $\II_\Phi$. These are of independent interest and will be discussed and proved in this section; here we pass to a complex valued setting, noting that the group $SO(m)$ in Section \ref{naxray} can at any point be replaced by the unitary group $U(m)$ which has Lie algebra $\u(m)=\{A\in \C^{m\times m}: A^* = -A\}$; moreover, within this section, function spaces with suppressed codomain are understood to contain  $\C$-valued functions (that is, $L^2(\DD)=L^2(\DD,\C)$ etc.).


\smallskip

The linearisation $\II_\Phi$ from \eqref{gateaux} can be expressed in terms of the {\it attenuated X-ray transform}, which is a linear map
 \begin{equation}\label{attenuatedxray}
I_\Phi :L^2(\DD,\C^m)\rightarrow L^2_\lambda(\partial_+S\bar \DD,\C^{m}),
\end{equation} 
defined for $\Phi \in C(\bar \DD,\u(m))$  in the following manner: Consider the \textit{phase space} $S\DD=\DD\times S^1=\{(x,v)\in \DD\times \R^2: \vert v \vert =1\}$, whose elements are tuples $(x,v)$ of points and directions.  Its closure $S\bar \DD$ has as boundary $\partial S\bar \DD = \partial_+S\bar \DD\cup \partial_-S\bar \DD$.
Given a continuous function $f:\bar \DD\rightarrow \C^m$ we define \begin{equation}
I_\Phi f = u_\Phi \vert_{\partial_+S\bar \DD},
\end{equation}
where $u=u_\Phi:S\bar \DD\rightarrow \C^{m}$ is the unique continuous solution of the transport problem
\begin{equation}\label{te}
(X+ \Phi) u = - f\text{ on } S\DD\quad\text{and } \quad  u = 0 \text{ on } \partial_-S\bar \DD.
\end{equation}
Here $X=v\cdot \nabla_x$ is the \textit{geodesic vector field} and $\Phi$ and  $f$ are viewed as functions on $S\DD$ (constant in the $v$-variable). Note that \eqref{te} is equivalent to solving an ODE  along every line $\gamma_{x,v}(t)=x+tv$ and existence of $u$ follows from standard arguments. 
On $L^2(\DD,\C^m)$ the transform $I_\Phi$ is defined by continuous extension,  using well known forward estimates (e.g., Theorem \ref{zernikeforward} for $s=0$).

\smallskip 

The non-Abelian X-ray transform $C_\Phi$ from (\ref{yetagain}) is related to this linear transform via a {\it pseudolinearisation identity} from \cite{PSU12}, see also \cite[Lem.\,5.5]{MNP21}. This takes the form
\begin{equation}\label{pseudolinearisation}
C_\Phi C_\Psi^{-1} - \id = I_{\Xi(\Phi,\Psi)}(\Phi-\Psi), \quad \Phi,\Psi \in C(\bar \DD,\u(m))
\end{equation}
where $\Xi(\Phi,\Psi):\DD\rightarrow \u(m^2)$ is defined as follows: For $x\in \DD$, let $\Xi(\Phi,\Psi)(x):\C^{m\times m}\rightarrow \C^{m\times m}$ be the linear map $\Xi(\Phi,\Psi)(x) A=\Phi(x) A - A \Psi(x)$.  Then $\Xi(\Phi,\Psi)(x)$ is skew-hermitian with respect to the Frobenius inner product and thus may be viewed as element in $\u(m^2)$; the attenuated X-ray transform in the previous display is then to be understood `one level higher', acting on $\C^{m\times m}\cong \C^{m^2}$-valued functions.  By virtue of  \eqref{pseudolinearisation},  injectivity and stability questions for $\Phi\mapsto C_\Phi$ reduce to  linear problems; moreover one can show (cf.\,Lemma \ref{derivatives} below) that  the Gateaux derivative $\II_\Phi$ can be expressed as
\begin{equation}
\II_\Phi h \equiv \frac{\d}{\d t} C_{\Phi+t h} = \big ( I_{\Xi(\Phi,\Phi)} h \big) C_\Phi\in L^2_\lambda(\partial_+S\bar \DD,\u(m)),\quad \Phi,h\in C(\bar \DD,\u(m)).
\end{equation}
Whenever it is convenient, we will think of $I_\Phi$ as the linearisation of $\Phi \mapsto C_\Phi$, bearing in mind that eventually one has to pass to $\II_\Phi=I_\Xi \cdot C_\Phi$.

\subsection{Analytical results} Our results are formulated in a non-standard scale $\tilde H^s(\DD,\C^m)$ of Sobolev-type spaces studied systematically in \cite{Mon20}.  This so called {\it Zernike scale} is defined in terms of the (normalised) {\it Zernike-polynomials} $\hat Z_{nk} : \DD\rightarrow \C$ ($k,n\in \N_0,0\le k \le n$),  defined in \eqref{zdef} below. Precisely, for $s\ge 0$ we have
\begin{equation}\label{zscaledef}
\tilde H^s(\DD,\C^m) = \left\{f\in L^2(\DD,\C^m): \Vert f \Vert_{\tilde H^s(\DD)}^2 = \sum_{n\in \N_0, 0\le k \le n } (1+n)^{2s}\vert f_{n,k} \vert^2 <\infty \right\},
\end{equation}
where $f_{n,k}=\langle f, \hat Z_{n,k}\rangle_{L^2(\DD)}$ are the coefficients of $f$ in the Zernike basis.
The definition can be extended to $s<0$ by duality, or equivalently by relaxing the requirement that $f\in L^2(\DD,\C^m)$ to $f$ being a  {\it supported distribution} in $\dot C^{-\infty}(\bar \DD,\C^{m})\supset L^2(\DD,\C^m)$, as described in \cite[\textsection 6.4.2]{MNP20}.  Further, we write
$
C^k(\bar \DD)
$
$(k\in \N_0)$ for the Banach space of functions on $\DD$, which are $k$-times continuously differentiable up to the boundary.

\begin{thm}[Stability estimate]\label{zernikestability}
Suppose $\Phi:\DD\rightarrow \so(m)$ is of regularity $C^4(\bar\DD)$ and its support is contained in a compact set $K\subset \DD$. Then
\begin{equation*}
\Vert f \Vert_{\tilde H^{-1/2}(\DD)}^2 \le C(\Phi,K ) \cdot \Vert I_\Phi f \Vert^2_{L^2_\lambda(\partial_+S\bar  \DD)}
\end{equation*}
for all $f\in L^2(\DD,\C^m)$ and a constant $C(\Phi,K)>0$ which (for fixed $K$) depends continuously on $\Phi$ in the $C^4$-topology.
\end{thm}

The uniform constant $c>0$ in Theorem \ref{findimstability} is obtained from the continuity statement in Theorem \ref{zernikestability} and the compact embedding $H^\alpha(\DD)\subset C^4(\bar \DD)$ ($\alpha >5$, see  Remark \ref{toolate}).  

\smallskip

The preceding Theorem  is sharp in that it captures the $\frac12$-ill-posedness of the attenuated X-ray transform. For smooth fields $\Phi$ the stability estimate can be seen to follow from the {\it normal operator} $N_\Phi=I_\Phi^* I_\Phi$ being a Hilbert space isomorphism from $ L^2(\DD,\C^m)$ onto  $ \tilde H^1(\DD,\C^m)$, which is the content of \cite[Thm.\,6.18]{MNP20}. 
Our contribution here lies in extending the result to lower regularity $\Phi$, and in the associated continuity statement.  The challenge is to bypass the (smooth) microlocal analysis methods that underpin the results in \cite{MNP20} to obtain mapping properties of $N_\Phi$ for $\Phi$ being of \textit{finite} regularity as is relevant in the statistical results from \textsection \ref{naxray}.  We achieve this by analysing certain singular integral operators `by hand' (requiring {\it four derivatives} on $\Phi$ in order to regularise) and establish Lipschitz type bounds for the maps $\Phi\mapsto N_\Phi$. This allows to pass to a finite regularity setting by means of an approximation argument. 

\smallskip

We emphasise here that the codomain of $I_\Phi$ (relevant to define $N_\Phi$) is the $L^2$-space with respect to the canonical area form $\lambda$, \eqref{measurelambda}.  Mapping properties of $I_\Phi$ are also frequently studied with respect to the `symplectic measure' on $\partial_+S\bar \DD$ (e.g.\,in \cite{MNP19a}) and we refer to  \cite[Remark 2.3]{MNP20} for a discussion of the choice of measure in the present setting.

\smallskip

The stability result is complemented by a fitting forward estimate in the Zernike scale over $\DD$.   Note that $\partial_+S\bar \DD$ is a compact manifold with boundary, hence there is a standard scale of Sobolev spaces $H^s(\partial_+S\bar \DD)$ ($s\in \R$, see e.g.\,\cite[\textsection 4.4]{Tay11}).

\begin{thm}[Forward estimate]\label{zernikeforward}Let $s\ge 0$ and suppose $\Phi$ is of regularity $C^k(\bar \DD)$, where $k=2\lceil s/2\rceil$ (the smallest even integer $\ge s$). Then for all $f\in \tilde H^{s}(\DD,\C^m)$ we have
\begin{equation*}
\Vert I_\Phi f \Vert_{H^s(\partial_+S\bar \DD)} \le C \cdot \big(1 + \Vert \Phi \Vert_{C^k(\bar \DD)}^2 \big)^k\cdot \Vert f \Vert_{\tilde H^{s}(\DD)}
\end{equation*}
for a constant $C=C(s)>0$. For $s=1$ one may also take $k=1$.
\end{thm}
For zero attenuation this can be strenghtened to the mapping property $I_0:\tilde H^{s}(\DD)\rightarrow H^{s+1/2}(\partial_+S\bar \DD)$ \cite[Prop.\,16]{Mon20}.  At present it is not clear how to obtain $\frac12$-smoothing forward estimates for general potentials $\Phi$.



\subsection{Preliminaries} Below we use the usual notation $ A\lesssim B$
as shorthand for `$A \le C B$, with a constant $C>0$ independent of the relevant quantities'. If we want to emphasise a certain dependence $C=C(a,b,\dots)$, we write $A\lesssim_{a,b,\dots} B$.

\subsubsection{Zernike scale and the operator $\L$}\label{zernikefacts} The Zernike polynomials (and their normalisation) are defined in a complex coordinate $z\in \DD$ by 
 \begin{equation}\label{zdef}
Z_{nk}(z) = \frac{1}{k!} \partial_{z}^k\left(z^n(\frac1z-\bar z)^k\right)\quad \text{ and } \quad \hat Z_{nk}(z) = \frac{Z_{nk}(z)}{\Vert Z_{nk} \Vert_{L^2(\DD)}},
\end{equation}
where $n\in \N_0$ and $0\le k \le n$ and we follow the labelling convention from \cite{KaBu04}. They form a complete orthogonal system of $L^2(\DD)$ and are part of a well known SVD of the unattenuated X-ray transform $I_0$ \cite{KaBu04}.  More recently it was observed that Zernike polynomials form an eigenbasis of a certain $2$nd order differential operator $\L$ on $\DD$ \cite{Mon20}, given by
\begin{equation}\label{xieta}
\L = (4\pi^2)^{-1}\left( \xi^* \xi + \eta^*\eta +1 \right)
\end{equation}
in terms of the complex vector fields
$
\xi = \sqrt{\vert z \vert^{-2}-1}\cdot (z\partial_z + \bar z \partial_{\bar z})$ and 
$\eta = i \vert z \vert^{-1}\cdot (z \partial_z - \bar z \partial_{\bar z})
$
as well as their $L^2$-adjoints $\xi^*$ and $\eta^*$.  We use the same prefactor as in \cite{MNP20} and refer to there for a description of $\L$ in polar coordinates. A computation shows that indeed $\L Z_{nk} = \left(\frac{1+n}{4\pi}\right)^2 Z_{nk}
$, such that alternatively the Zernike scale is defined by
\begin{equation}
\tilde H^s(\DD,\C) = \dom \L^{s/2},\quad s\ge 0
\end{equation}
where $\L$ is viewed as self-adjoint operator in $L^2(\DD,\C)$ with domain $\dom \L = \tilde H^2(\DD,\C)$ and $\L^{s/2}$, together with its natural domain, are defined via the spectral theorem. 

\smallskip

The operator $\L$ is elliptic in $\DD$, but degenerates at $\partial \DD$ as $\xi$ vanishes at the boundary. As consequence the Zernike scale  is non-standard near $\partial \DD$. Nevertheless it shares many features with usual Sobolev spaces. We refer to \cite[Lem.\,6.14]{MNP20} for a Rellich compactness theorem and interpolation properties. For the comparison with the standard Sobolev scale see \cite[Lem.\,6.15]{MNP20} (stating that $\tilde H^s(\DD)$ and $H^s(\DD)$ are equivalent on compactly supported functions) and \cite[Lem.\,14]{Mon20},  which implies the embeddings (strict for $s>0$)
\begin{equation}\label{zernikeembedding}
 \tilde H^{2s+\epsilon}(\DD,\mathbb{C}) \subset  H^s(\DD,\mathbb{C}) \subset \tilde H^{s}(\DD,\mathbb{C}),\quad s\ge 0, \epsilon >0.
\end{equation} 

For future reference we record that
\begin{equation} \label{znorms}
\Vert  Z_{nk} \Vert_{L^2(\DD)}=\left(\pi/(n+1)\right)^{1/2}\quad \text{ and } \quad \Vert Z_{nk} \Vert_{L^\infty(\DD)} =1,
\end{equation}
(see e.g.\,\cite[Appendix A.2]{Mon20}) and, as can easily be observed from \eqref{zdef},
\begin{equation}\label{zsymmetry}
\overline{Z_{nk}}=(-1)^nZ_{n,n-k}\quad \text{ and }\quad Z_{nk}(uz)=u^{n-2k}Z_{nk}(z)~(u\in S^1,z\in \DD),
\end{equation}
where $u$ is viewed as complex number to make sense of its exponents.

\subsubsection{Bounds on integrating factors}\label{boif}
Given an attenuation $\Phi\in C(\bar \DD,\u(m))$, we write $R_\Phi:S\DD\rightarrow U(m)$ for the unique continuous solution (differentiable along $X$) of
\begin{equation}
(X+\Phi) R_\Phi =0\text{ on }S\bar \DD\quad \text { and } \quad R_\Phi = \id \text{ on } \partial_-S\bar \DD.
\end{equation}
This is an example of an {\it integrating factor} for $\Phi$ and enters several computations below; we thus collect here some useful bounds. Using a similar notation for a second attenuation $\Psi$ and writing $R^{\pm 1}$ for $R$ or its inverse $R^{-1}$, we have
\begin{equation}\label{horse1}
\Vert R^{\pm 1}_\Phi \Vert_{L^\infty(S \DD)} \lesssim_m 1 \quad \text{ and } \quad \Vert R^{\pm 1}_\Phi - R^{\pm 1}_\Psi \Vert_{L^\infty(S \DD)} \lesssim_m \Vert \Phi - \Psi \Vert_{L^\infty(\DD)}.
\end{equation}
Further, assuming that $\Phi$ and $\Psi$ are of regularity $C^k$ ($k\ge 1$) and have their support contained in a compact set $K\subset \DD$, also $R_\Phi,R_\Psi$ are of regularity $C^k$ and obey 
\begin{equation}\label{horse2}
\begin{split}
\Vert R^{\pm 1}_\Phi \Vert_{C^k(S\bar \DD)} &\lesssim_{m,k,K} (1+\Vert \Phi \Vert_{C^k(\bar \DD)})^k \\
\Vert R^{\pm 1}_\Phi - R^{\pm 1}_\Psi \Vert_{C^k(S\bar \DD)} &\lesssim_{m,k,K} (1 +\Vert \Phi \Vert_{C^k(\bar \DD)}+ \Vert \Psi \Vert_{C^k(\bar \DD)} )^k\cdot  \Vert \Phi - \Psi \Vert_{C^k(\bar  \DD)}.
\end{split}
\end{equation}
Here $C^k(S\bar \DD)$ ($k\in \N_0$) denotes the Banach space of functions on $S\DD$ which are $k$-times continuosly differentiable up to the boundary. 
In the compactly supported case, the previous two displays are stated as equation (5.17) in \cite{MNP21} and Proposition 2.2\,in \cite{Boh20} and are both consequences of Lemma 5.2 of the first cited paper; one  easily observes that the support condition can be dropped for $k=0$.

\begin{rem}\label{fullsupport}
In general, one can avoid compact support conditions on $\Phi$ and $\Psi$  by dropping the requirement that $R_\Phi,R_\Psi = \id$ on $\partial_-S\bar \DD$. Indeed, one can extend the matrix fields to compactly supported fields in a slightly larger disc $\DD'$ and impose the corresponding boundary condition at $\partial_-S\overline{\DD'}$ instead.  
\end{rem}

Further, any pair of maps $V,W:\DD\rightarrow \C^{m\times m}$ satisfying
\begin{equation}
(X+\Phi)V = (X+\Psi) W(=: -2h) \text{ on } S\bar \DD\quad \text{ and } \quad V = W = 0 \text{ on } \partial_-S\bar \DD
\end{equation}
(for $\Phi$,$\Psi$ assumed to be continuous and of full support), also obeys
\begin{equation}\label{horse3}
\Vert V \Vert_{L^\infty(\DD)}\lesssim_m \Vert h \Vert_{L^\infty(\DD)}, \quad \Vert V- W \Vert_{L^\infty(S \DD)} \lesssim_m \Vert \Phi - \Psi \Vert_{L^\infty( \DD)} \cdot \Vert h \Vert_{L^\infty(\DD)}.
\end{equation}
This can be seen by writing $V-W = R_\Phi \cdot \left(  (R_\Phi^{-1}V-R_\Psi^{-1} W) + (R_{\Psi}^{-1} - R_\Phi^{-1}) W) \right)$ and using 
\eqref{horse1} together with Lemma 5.2 in \cite{MNP21} for $G=R_\Phi^{-1}V-R_\Psi^{-1} W$ and $G'=W$, which obey the equations $XG = - 2(R_\Phi^{-1} - R_\Psi^{-1}) h $ and $(X+\Psi)G'=-2h$ respectively.

\begin{rem}
The estimates above will often be used `one level higher' without further mention. This means that $\Phi:\DD\rightarrow \u(m)$ is replaced by an attenuation $\Xi:\DD\rightarrow \u(m^2)\subset \End(\C^{m\times m})$, typically given as commutator $\Xi=[\Phi,\cdot]$.
\end{rem}

\subsubsection{Derivatives of the non-Abelian X-ray transform}
\begin{lem} \label{derivatives} Let $\Phi, h\in C(\bar \DD,\u(m))$. Then the non-Abelian X-ray transform, viewed as map into $L^2_\lambda(\partial_+S\bar \DD)$ or $L^\infty(\partial_+S\bar \DD)$,  is twice Gateaux differentiable at $\Phi$ in direction $h$, with derivatives equal to
\begin{eqnarray}
\dot C_\Phi[h] &\equiv (\d/\d t)_{t=0} C_{\Phi + t h}& = I_\Xi h \cdot C_\Phi \label{dot} \\
\ddot C_\Phi[h] &\equiv (\d/\d t)^2_{t=0} C_{\Phi + t h } &= I_{\Xi}[ hV_h]  \cdot C_\Phi. \label{ddot}
\end{eqnarray}
Here $\Xi :\DD\rightarrow \u(m^2)\subset \End(\C^{m\times m})$ is defined pointwise by $\Xi(x)=[\Phi(x),\cdot]$ (= the commutator of matrices in $\C^{m\times m}$) and $V_h:S\bar \DD\rightarrow \C^{m\times m}$ is the unique solution to the problem $(X+\Xi)V = -2 h$ on $S\bar \DD$ with $V = 0 $ on $\partial_-S\bar \DD$. 
\end{lem}

\begin{proof}
The pseudolinearisation formula \eqref{pseudolinearisation} immediately yields  \begin{equation}\label{pfdot}
C_{\Phi + t h} = C_\Phi + t I_{\Xi + t h } h \cdot C_\Phi.
\end{equation}
We note that $I_{\Xi + th} h = u_t\vert_{\partial_+S\bar \DD}$, where $u_t=u(t,\cdot)\in C^\infty(S\DD,\C^{m\times m})$ solves
\begin{equation} \label{pfdot2}
X u_t + [\Phi,u_t] + th u_t = - h \text{ on } S \DD\quad \text{ and } \quad u_t = 0\text{ on } \partial_-S\bar\DD.
\end{equation}
In fact, the function $u(t,x,v)$ is in $C^\infty(\R\times S \DD,\C^{m\times m})$ and, taking $t$-derivatives in the previous display, we find that $\partial_t u(t,\cdot) = \dot u_t$ satisfies
\begin{equation} \label{pfdot3}
X\dot u_t + [\Phi,\dot u_t] + t h \dot u_t=  - h u_t \text{ on } S \DD\quad \text{ and } \quad \dot u_t = 0 \text{ on } \partial_-S\bar \DD.
\end{equation}
Write
$
I_{\Phi + t h }h =( u_0 + t \dot u_0 +  v_t)\vert_{\partial_+S\bar \DD},$ where $v_t = u_t - u_0 - t\dot u_0 $; next note that the previous display implies $\dot u_0 \vert_{\partial_+S\bar \DD} = I_{\Xi}(hu_0)=\frac 12 I_\Xi (h V_h)$, with $V_h$ as in the statement of the lemma. We can thus further expand \eqref{pfdot} to 
\begin{equation}
C_{\Phi+th} = C_\Phi + t I_\Xi h\cdot C_\Phi + \frac12 t^2 I_{\Xi}(h V_h)\cdot C_\Phi + t (v_t\vert_{\partial_+S\bar \DD})\cdot C_\Phi.
\end{equation}
In order to prove \eqref{dot} and \eqref{ddot}, it remains to show that for $p=2$ and $p=\infty$ we have $\Vert v_t\vert _{\partial_+S\bar \DD} \Vert_{L^p_\lambda(\partial_+S\bar \DD)} = o(t)$ as $t\rightarrow 0$ (the additional factor $C_\Phi$ in the remainder can be estimated in $L^\infty$-norm). Equations \eqref{pfdot2} and \eqref{pfdot3} imply that
\begin{equation}
(X+\Xi)v_t = -th(u_t-u_0)\text{ on } S \DD,\quad \text{ and } \quad v_t = 0 \text{ on } \partial_-S\bar \DD
\end{equation}
and thus \cite[Lem.\,5.2]{MNP21} yields $\Vert v_t\vert_{\partial_+S\bar \DD} \Vert_{L^p_\lambda(\partial_+S\bar \DD)}\le C t \Vert u_t - u_0 \Vert_{L^p(S\DD)}\Vert h \Vert_{L^\infty}$ for an absolute constant $C>0$. Similarly $(X+\Xi)(u_t - u_0)= t u_t h$ on $S\DD$, whence another application of the just cited lemma yields $\Vert u_t - u_0 \Vert_{L^p(S\DD)} \le C t \Vert u_t h \Vert_{L^p(S \DD)} \le C t \Vert u_t \Vert_{L^2(S\DD)} \Vert h \Vert_{L^\infty(S\DD)}$. Finally, 
\begin{equation}\label{pfdot4}
\Vert u _t \Vert_{L^p(S\DD)} \le C \Vert h \Vert_{L^p(\DD)},
\end{equation}
which follows from \cite[Lem.\,5.2]{MNP21}, this time applied to $u_t$,  satisfying \eqref{pfdot2}.\end{proof}

\subsection{Proof of forward estimate} 

The attenuated X-ray transform can be written as
\begin{equation}
I_\Phi f =R_\Phi \I(R_\Phi^{-1}f),\quad \text{ on } \partial_+S\bar \DD
\end{equation}
where $\I$ here denotes the general, unattenuated X-ray transform on $S\DD$
(i.e.\,$
Iu(x,v) = \int_{0}^{\tau(x,v)} u(x+tv,v) \d t$ for $u\in C(S\bar \DD,\C^m)
$)
and $R_\Phi:S\bar\DD\rightarrow U(m)$ is an integrating factor for $\Phi$, that is, a solution to $(X+\Phi)R_\Phi =0$ on $S\bar \DD$. In our context, when $\Phi$ has regularity $C^k(\bar \DD)$,  the integrating factor can be chosen of regularity $C^k(S\bar \DD)$ (cf.\,Remark \ref{fullsupport}). 

On $S \DD$ we define a scale, which captures `Zernike-regularity' in the spatial variable and standard Sobolev regularity in the phase variable. Precisely, for $s\ge 0$ let
\begin{equation*}
\tilde H^s(S \DD,\C) 
=\Big \{u\in L^2(S\DD,\C): \Vert u \Vert_{\tilde H^s(S \DD)}^2 = \sum_{\ell\in \Z} \langle \ell\rangle^{2s} \Vert u_\ell \Vert_{L^2(\DD)}^2 + \Vert u_\ell \Vert_{\tilde H^s(\DD)}^2 <\infty \Big\}.
\end{equation*}
where $u_\ell:\DD\rightarrow \C$ is the $\ell$th vertical Fourier mode, given by 
\begin{equation*}
u_\ell(x) = \int_{0}^{2\pi} e^{-i\ell t} u\left(x,(\cos t,\sin t)\right) \d t,\quad u\in L^2(S\bar \DD,\C)
\end{equation*}
 and $\langle \ell \rangle = \left( 1 + \ell^2\right)^{1/2}$. Here $\tilde H^0(S\DD) = L^2(S\DD)$ is the $L^2$-space with respect to the natural Lebesgue measure on $S\DD = \DD\times S^1$. As usual, one defines a version of $\C^m$ and $\C^{m\times m}$-valued spaces.


\smallskip

Theorem \ref{zernikeforward} follows from forward estimates of the general transform $I$ on $\tilde H^s(S\bar \DD)$, together with operator norm bounds for $R_\Phi^{\pm 1}$, acting via multiplication. These assertions are the content of the next two lemmas, formulated in the scalar case for convenience.

\begin{lem}\label{lem_phaseforward} For all $s\ge 0$ and $u\in \tilde H^s(S\DD)$ we have
\begin{equation*}
\Vert \I u \Vert_{ H^s(\partial_+S\bar \DD)} \lesssim_s  \Vert u \Vert_{\tilde H^s(S\DD)}.
\end{equation*}
\end{lem}

Note that the transform $\I:L^2(S\bar \DD)\rightarrow L^2_\lambda(\partial_+S\bar \DD)$, acting on functions on $S\DD$ (rather than $\DD$),
is {\it not} $1/2$-regularising, resulting in the mapping property in Theorem \ref{zernikeforward}.


\begin{lem}\label{lem_product}
Let $s\ge 0$,  and $k=2\lceil s/2 \rceil\in 2 \mathbb{N}_0$ (if $s=1$ also $k=1$ is allowed). Then multiplication is continuous in the following settings:
\begin{align}
C^k(\bar \DD) \times \tilde H^{s}(\DD) &\rightarrow \tilde H^s(\DD)\qquad && \label{product1} \\
C^k(S \bar \DD) \times \tilde H^{s}(S \DD) &\rightarrow \tilde H^s(S\DD)\qquad && \label{product2}
\end{align}
\end{lem}

The lemmata are proved below, first we complete the proof of Theorem \ref{zernikeforward}. 

\begin{proof}[Proof of Theorem \ref{zernikeforward}] Let $f\in \tilde H^s(\DD,\C^m)$ ($s\ge 0$) and $k$ as in Lemma \ref{lem_product}. Then 
\begin{equation}
\begin{split}
\Vert I_\Phi f \Vert_{H^s(\partial_+S\bar \DD)}& \le \Vert R_\Phi \Vert_{C^k(\partial_+S\bar \DD)} \Vert \I (R_\Phi^{-1} f ) \Vert_{H^s(\partial_+S\bar \DD)}\\
& \lesssim \Vert R_\Phi \Vert_{C^k(\partial_+S\bar \DD)} \cdot \Vert R_\Phi^{-1} f \Vert_{\tilde H^s(S\DD)}\\
&\lesssim\Vert R_\Phi \Vert_{C^k(\partial_+S\bar \DD)}  \cdot  \Vert R_\Phi^{-1} \Vert_{C^k(S\bar \DD)} \cdot \Vert f \Vert_{\tilde H^s(S\DD)},
\end{split}
\end{equation}
where we used that $C^k$-functions act continuously (via multiplication) on $H^s(\partial_+S\bar \DD)$ and then applied Lemma \ref{lem_phaseforward} and Lemma \ref{lem_product}. Now, as $f$ is a function on $\DD$, its Fourier modes $f_\ell = 0$ for $\ell\neq 0$, such that the Zernike-norm on $S\DD$ turns into $\Vert f \Vert_{\tilde H^s(\DD)}$. Finally, we have
\begin{equation}
\Vert R_\Phi^{\pm  1} \Vert_{C^k(\partial_+S\bar \DD)} \le \Vert R_\Phi^{\pm  1} \Vert_{C^k(S \DD)} \lesssim (1+\Vert \Phi \Vert_{C^k(\bar \DD)})^k,
\end{equation}
by \eqref{horse2}, which completes the proof.
\end{proof}

\subsubsection{Forward estimate on phase space}\label{fanbeam} In order to prove Lemma \ref{lem_phaseforward} we recall that $\partial_+S\bar \DD$ can be represented in fan-beam coordinates as $[-\pi/2,\pi/2]_\alpha\times (\R/ 2\pi \Z)_\beta$ via the map
\begin{equation}
(\alpha,\beta) \mapsto (e^{i\beta},e^{i(\alpha + \beta + \pi)}) = (x,v)\in\partial_+S\bar \DD.
\end{equation}
Here $(x,v)\in \partial_+S\bar \DD$ is viewed as pair of complex numbers with $\vert x \vert, \vert v \vert = 1$ and $\Re(\bar x v) \le 0$. In these coordinates we define a collection of functions, indexed by $n\in \N_0$ and $k\in \Z$:
\begin{align}
\psi_{nk}^+(\alpha,\beta) &= \frac{(-1)^{n}}{4\pi }e^{i(n-2k)(\alpha+\beta)}\left(e^{i(n+1)\alpha} + (-1)^{n} e^{-i(n+1)\alpha}\right)\\
\psi^{-}_{nk}(\alpha,\beta) &= e^{i(\alpha+\beta+\pi)} \psi^+_{nk}(\alpha,\beta).
\end{align}
If we normalise them to have $L^2$-norm equal to $1$, we obtain an orthonormal basis $(\hat \psi^\bullet_{nk}:n\in \N_0,k\in \Z,\bullet = \pm)$ of $L^2_\lambda(\partial_+S\bar \DD)$.  This basis was constructed  in \cite{Mon16}\footnote{The positive version $\psi^+_{nk}$ corresponds to $\psi_{nk}$ in  \cite{Mon20} (defined in equation (24)). Up to a multiplicative factor and after a change of indices $p=n-2k$,$q=n-k$ it equals $u'_{pq}$ in \cite{Mon16} defined after equation (17).
The negative version $\psi^{-}_{nk}$ equals, up to a multiplicative factor and a change of indices $p=n-2k+1$, $q=n-k+1$,  the functions $v_{pq}$ in \cite{Mon16}. }, where the classification into $+$ and $-$ corresponds to a natural splitting $L^2_\lambda(\partial_+S\bar \DD)=L^2_+(\partial_+S\bar\DD)\oplus L^2_-(\partial_+S\bar\DD)$. In particular, the $+$ version is part 
of an SVD of the unattenuated transform $I_0:L^2(\DD)\rightarrow L^2_+(\partial_+S\bar \DD)$ (cf.\,\cite{KaBu04}\cite[Thm.\,10]{Mon20}), which is crucially used below.

\begin{proof}[Proof of Lemma \ref{lem_phaseforward}] Note that $L^2(S\DD)$ has as orthormal basis the collection of functions $(x,v)\mapsto \frac{1}{\sqrt {2\pi}}v^\ell \hat Z_{nk}(x)$ ($0\le k \le n$, $\ell\in \Z$), where we view $v\in S^1$ as complex number to make sense of the powers $v^\ell$.  To find the action of $\I$ on such a basis element note that \begin{equation}
v^\ell\psi^+_{nk} = 
\begin{cases}
\psi^+_{n,k-j} &\ell=2j \in 2\Z\\
\psi^{-}_{n,k-j} &\ell=2j+1\in 2\Z+1,
\end{cases}
\end{equation}
 as  $v=e^{i(\alpha+\beta+\pi)}$ in fan-beam coordinates.  In particular, using the SVD of $I_0$,
\begin{equation}
\I(v^\ell \hat Z_{nk}) = v^{\ell} I_0(\hat Z_{nk}) = \sqrt{\frac{4 \pi}{n+1 }} \times
\begin{cases}
 \hat \psi^+_{n,k-j} & \ell=2j\in 2\Z\\
\hat \psi^{-}_{n,k-j}& \ell=2j+1 \in 2\Z+1,
\end{cases}
\end{equation}
such that for all
$
u \in L^2(S\bar \DD)
$ with coefficients $u_{nk\ell}=(2\pi)^{-1/2}\langle u, v^\ell \hat Z_{nk}\rangle_{L^2(S\bar \DD)}$ and all indices $n\ge 0,k\in \Z$ the following holds true:
\begin{equation}\label{almostsvd}
\langle \I u, \hat \psi^\bullet_{nk}\rangle_{L^2_\lambda(\partial_+S\bar \DD)} = \sqrt{\frac{2}{n+1}} \times 
\begin{cases}
	\sum_{j=0}^n u_{n,j,2(j-k)} & \bullet = +\\
		\sum_{j=0}^n u_{n,j,2(j-k)+1} & \bullet = -\\
\end{cases}.
\end{equation}
Next, consider on $\partial_+S\bar \DD$ the differential operator $P= -\left((\partial_\alpha-\partial_\beta)^2+\partial_\beta^2\right)$, which is elliptic up to the boundary and thus generates the standard Sobolev scale.  A computation shows that $P$  has $\psi^\pm_{nk}$ as eigenbasis:
\begin{equation}\label{eigenp}
P\psi^\pm_{nk} = (a_{nk}^\pm)^2 \cdot \psi^{\pm}_{nk},\quad a_{nk}^\pm=\left((1+n)^2+(n-2k+\epsilon_\pm)^2\right)^{1/2},
\end{equation}
where we use the notation $\epsilon_+=0$ and $\epsilon_-=1$. Thus for $s\in 2 \mathbb{N}_{0}$ we have
\begin{equation}
\Vert \I u \Vert_{H^s(\partial_+S\bar \DD)}^2 = \Vert P^{s/2} \I u\Vert_{L^2_\lambda(\partial_+S\bar \DD)}^2 = \sum_{\substack{n\ge 0, j\in \Z\\\bullet=\pm}} \vert \langle P^{s/2}\I u ,\psi^\bullet_{nj}\rangle_{L^2_\lambda(\partial_+S\bar \DD)}\vert^2.
\end{equation}
We can integrate by parts without collecting boundary terms, as all eigenfunctions  $\psi^\bullet_{nj}$ vanish identically on the boundary $\partial_0S\bar \DD \equiv\{\alpha=\pm\pi/2\}$ of $\partial_+S\bar \DD$.  Further using \eqref{eigenp}, \eqref{almostsvd}  and Jensen's inequality, the previous display continuous with
\begin{equation}\label{jensen}
\begin{split}
=  \sum_{\substack{n\ge 0, j\in \Z\\\bullet=\pm}} (a_{nj}^\bullet)^{2s} \vert \langle \I f ,\psi_{nj}^\bullet\rangle \vert^2 
&= 2  \sum_{\substack{n\ge 0, j\in \Z\\\bullet=\pm}} (n+1)(a_{nj}^\bullet)^{2s} \big \vert\frac{1}{n+1} \sum_{k=0}^n u_{n,k,2(k-j)+\epsilon_\bullet}\big \vert^2\\
& \le 2  \sum_{\substack{0\le k \le n\\j\in \Z, \bullet=\pm}} (a_{nj}^\bullet)^{2s} \vert u_{n,k,2(k-j)+\epsilon_\bullet} \vert^2.
\end{split}
\end{equation}
We now relabel $2(k-j)+\epsilon_\pm \leadsto \ell\in \Z$, noting that $+$ and $-$ corresponds to even and odd values for $\ell$. Thus, writing
\begin{equation}
b_{nk\ell}  = \left((n+1)^2+(n-2k+\ell)^2\right)^{1/2}=\begin{cases}
a^+_{n,k-i}& \ell=2i\in 2\Z\\
a^{-}_{n,k-i}& \ell= 2i+1\in 2\Z+1 
\end{cases}
\end{equation}
we have
$
b_{nk\ell}^{2s} \lesssim_s (n+1)^{2s} + \langle \ell \rangle^{2s}
$
and thus
\begin{equation}
\Vert \I u \Vert_{H^s(\partial_+S\bar \DD)}^2 \le2 \sum_{\substack{0\le k \le n\\\ell\in \Z}} b_{nk\ell}^{2s} \vert u_{nk\ell}\vert^2 \lesssim_s \sum_{\substack{0\le k \le n\\\ell\in \Z}} \vert (n+1)^s u_{nk\ell}\vert^2 + \langle \ell \rangle^{2s} \vert u_{nk\ell}\vert^2,
\end{equation}
which gives the desired estimate, at least for $s\in 2\N_0$. The general case $s\in [0,\infty)$ follows by interpolation.
\end{proof}

\subsubsection{Estimate on products} The proof of Lemma \ref{lem_product} relies on the following three identities for Zernike polynomials:
\begin{align}\label{identity1}
zZ_{nk}&=\left(1-\nicefrac{k}{n+1}\right) Z_{n+1,k} - (\nicefrac{k}{n+1}) Z_{n-1,k-1}\\
(1- z\bar z) \partial_{\bar z} Z_{nk} &= (n-k+1)(z Z_{nk} - Z_{n+1,k}) \label{identity2} \\
(z\partial_z - \bar z \partial_{\bar z})  Z_{nk}  & =  i (n-2k) Z_{nk}  \label{identity3},
\end{align}
where we set $Z_{nk}=0$ if $n< 0$ or $\vert n-2k \vert > n$. The second and third identity were kindly communicated to us by François Monard; all three can be verified by a straightforward computation that we omit here.

\smallskip 
The virtue of the identities is that they immediately give mapping properties for certain model operators in the Zernike scale. For example,  \eqref{identity1}  implies that multiplication by $z$ gives a continuous map $\tilde H^s(\DD)\rightarrow \tilde H^s(\DD)$.  Using this together with the remaining two identities, we see that all of the operators 
\begin{equation}\label{modelops}
(1-\vert z \vert^2)\partial_{\bar z}, \quad (1-\vert z \vert^2)\partial_z,  \quad \text{ and }\quad  z\partial_z - \bar z \partial_{\bar z}
\end{equation}
map $\tilde H^{s+1}(\DD)\rightarrow \tilde H^s(\DD)$ ($s\ge 0$) continuously. Here, the assertion for the second operator, which is the  conjugate of the first one, follows from symmetry property \eqref{zsymmetry}.

\begin{proof}[Proof of Lemma \ref{lem_product}] The proof of \eqref{product1} will be explained in detail, the one for \eqref{product2} is very similar and the necessary changes are discussed at the very end.
Recall from \eqref{xieta} the representation of the operator $\L$ in terms of the first order operators
\[
\xi = \sqrt{\vert z \vert^{-2}-1}\cdot (z\partial_z + \bar z \partial_{\bar z})  \quad \text{ and } \quad 
\eta = i \vert z \vert^{-1}\cdot (z \partial_z - \bar z \partial_{\bar z})
\]
and their $L^2$-adjoints $\xi^*,\eta^*$. Let $u,v\in C^\infty(\DD)$ and assume for now that $u$ vanishes for $\vert z \vert \le \delta$. Then $\xi^*(uv) = (\xi^*u)v - u\xi v$ (similarly for $\eta$), such that
\begin{equation}\label{product3}
\begin{split}
\L(uv) =&~ (\L u) v + u \L v - \frac{1}{4\pi^2}\left[uv + 2 \xi u \cdot \xi v  +2  \eta u \cdot \eta v\right].
\end{split}
\end{equation}
We want to show by induction over $s\in 2\N_0$, that
\begin{equation}\label{product4}
\Vert uv \Vert_{\tilde H^s(\DD)} \lesssim_s \Vert u \Vert_{C^s(\bar\DD)} \cdot \Vert v \Vert_{\tilde H^s(\DD)}.
\end{equation}
This is trivially true for $s = 0$,  so we may assume that it holds for some $s\in 2\N_0$ and prove it for $s+2$.  By \eqref{product3} we have
\begin{equation*}
\begin{split}
\Vert uv \Vert_{\tilde H^{s+2}(\DD)} = \Vert  \L(uv)\Vert_{\tilde H^s(\DD)} \le &~ \Vert (\L u)v \Vert_{\tilde H^s(\DD)} +  \Vert u( \L v) \Vert_{\tilde H^s(\DD)} +  \Vert uv \Vert_{\tilde H^s(\DD)}\\
& + \Vert \xi u \cdot \xi v \Vert_{\tilde H^s(\DD)} +  \Vert \eta u \cdot  \eta v \Vert_{\tilde H^s(\DD)}
\end{split}
\end{equation*}
and by the induction hypothesis the first three terms on the right hand side can be bounded by (a constant times) $\Vert u \Vert_{C^{s+2}(\bar \DD)} \cdot \Vert v\Vert_{\tilde H^{s+2}(\DD)}$, as desired. To bound the remaining terms note that
\begin{equation}
\begin{split}
\xi u \cdot \xi v& = (\vert z \vert^{-2} -1) \cdot \left(z\partial_z u + \bar z \partial_{\bar z} u\right)\left(z\partial_z v + \bar z \partial_{\bar z} v\right) \\
& = \left[\frac{z}{\bar z} \partial_z u + \partial_{\bar z} u \right]\cdot (1-\vert z \vert^2)\partial_z v + \left[ \partial_z u + \frac{\bar z}{ z}\partial_{\bar z} u \right]\cdot (1-\vert z \vert^2)\partial_{\bar z} v,
\end{split}
\end{equation}
whence the mapping properties discussed after \eqref{modelops} (together with the induction hypothesis) can be used to obtain
\begin{equation}\label{product5}
\Vert \xi u \cdot \xi v \Vert_{\tilde H^s(\DD)} \lesssim_{s,\delta} \Vert u \Vert_{C^{s+1}(\bar \DD)} \cdot \Vert v \Vert_{\tilde H^{s+1}(\DD)}.
\end{equation}
Finally, $\eta u \cdot \eta v =-  \vert z \vert^{-2} (z \partial_z - \bar z \partial_{\bar z})u \cdot (z \partial_z - \bar z \partial_{\bar z}) v $ and since we have also derived mapping properties for $(z \partial_z - \bar z \partial_{\bar z})$,  \eqref{product5} remains true with $\xi$ replaced by $\eta$. We have thus established \eqref{product4} for $u$ vanishing in a disk $\vert z\vert \le \delta$.\\
The restriction that $u$  shall vanish in a neighbourhood of zero is easily removed by a cutoff procedure, as the scales $\tilde H^s$ and $H^s$ are equivalent on compactly supported functions \cite[Lemma 6.15]{MNP20} and the analogue of \eqref{product4} for the standard scale is well known.\\
Mapping property \eqref{product1} for $s\in [0,\infty)$ follows from an interpolation argument, noting that the Zernike scale satisfies the expected interpolation property \cite[Lemma 6.14]{MNP20}.  The remark about the special case $s=1$ easily follows from
\begin{equation}
\Vert uv \Vert_{\tilde H^1(\DD)}^2 = \Vert \xi(uv) \Vert_{L^2(\DD)}^2 + \Vert \eta(u v) \Vert_{L^2(\DD)}^2 + \Vert uv \Vert_{L^2(\DD)}^2
\end{equation}
together with the Leibniz rule and the obvious estimates, noting that both $\xi$ and $\eta$ have coefficients which are bounded away from a neighbourhood of zero.\\
Finally, we discuss how the argument can be adapted to phase space $S \DD$. The scale $\tilde H^s(S \DD)$ is generated by the operator $\L_V=\L - V^2+1$, where $V\equiv \d/\d \theta$ is the vertical derivative (i.e.\, $\tilde H^s(S\bar \DD)$ is the natural domain of $\L_V^{s/2}$, defined via the spectral theorem). The argument from above can thus be applied almost verbatim, with the only difference that expansion \eqref{product3} also contains vertical derivatives which are easily bounded.
\end{proof}

\subsection{Proof of stability estimate} \label{s_pfstability}
The proof of Theorem \ref{zernikestability} is based on a careful analysis of the normal operator associated to $I_\Phi$. Recall that this is defined by
\begin{equation}
N_\Phi = I_\Phi^*I_\Phi : L^2(\DD,\C^m)\rightarrow L^2(\DD,\C^m),
\end{equation}
where $I_\Phi^*$ is the $L^2$-adjoint of $I_\Phi:L^2(\DD,\C^m)\rightarrow L^2_\lambda(\partial_+S\bar \DD,\C^m)$. 

\smallskip

We introduce the following notation:  For $K\subset \DD$ compact and $k\in \N_0$, we write $C_K^k(\DD)$ for the space of $k$-times continuously differentiable functions on $\DD$ with support contained in $K$ and equip it with the obvious Banach space norm.  For two Hilbert spaces $H_1$ and $H_2$ we write $\B(H_1,H_2)$ for the space of continuous linear operators between them, equipped with the operator norm; $\B^\times(H_1,H_2)$ denotes the open subset of invertible operators.

\begin{thm}[Mapping properties of $N_\Phi$] \label{isoprop} Let $K\subset \DD$ be compact.
\begin{enumerate}[label=(\roman*)]
\item \label{isoprop1} For $\Phi\in C^2_K(\DD,\u(m))$ the normal operator $N_\Phi$ maps $L^2(\DD,\C^m)$ into $\tilde H^1(\DD,\C^m)$. The following map is Lipschitz continuous on bounded sets:
\begin{equation}
\Phi \mapsto N_\Phi, \text{ as map }  C_K^2(\DD,\u(m))\rightarrow \B(L^2(\DD),\tilde H^1(\DD))
\end{equation}
\item \label{isoprop2} For $\Phi\in C_K^4(\DD,\u(m))$ the normal operator $N_\Phi$ is an isomorphism from $L^2(\DD,\C^m)$ onto $\tilde H^1(\DD,\C^m)$.  The following map is continuous:
\begin{equation}
\Phi \mapsto N_\Phi^{-1} \text{ as map } C_K^4(\DD,\u(m))\rightarrow \B^\times(\tilde H^1(\DD),L^2(\DD)).
\end{equation}
\end{enumerate}
\end{thm}

The preceding theorem is proved in three steps:
\begin{enumerate}[label=(\Alph*)]
\item For smooth $\Phi$, the operators $N_\Phi$ and  $K_\Phi=N_\Phi - N_0$ are pseudodifferential operators of order $-1$ and $-2$, respectively, admitting singular integral representation in terms of a smooth kernel $A(x,v,t)$.  For operators of this type we derive norm bounds in terms of the kernel $A(x,v,t)$. Using a Calder{\'o}n-Zygmund type argument,  we only require bounds of the $2$nd and $4$th order derivatives of $A$.
\item Using these norm bounds, we follow the localisation procedure from \cite[Lem.\,6.17]{MNP20} to obtain (local) Lipschitz estimates for 
$\Phi \mapsto N_\Phi$ and $\Phi \mapsto K_\Phi$, where the domain is equipped with $C^2$-norm and $C^4$-norm, respectively.
\item By means of the local Lipschitz estimates and an approximation argument, we obtain an analogue of \cite[Lem.\,6.17]{MNP20} for $\Phi$ of regularity $C^4$.  Additionally using stability estimate from \cite[Thm.\,5.3]{MNP21} to get injectivity of $N_\Phi$ and $N_\Phi^*$ for $\Phi \in C_K^4(\DD,\u(m))$, we can then imitate the proof of \cite[Thm.\,6.18]{MNP20}.
\end{enumerate} 

\medskip
The next ingredient in the proof of Theorem \ref{zernikestability} is a classical result of Heinz and L{\"o}wner on the `operator monotonicity' of  the square root $T\mapsto \sqrt{T}$ (cf.\,\cite[Satz 3]{Hei51} or \cite[\textsection V.4, Thm.\,4.12]{Kat52}), combined with a duality argument.  The result is purely functional analytic,  and we formulate it on an arbitrary Hilbert space $H$ and with respect to a scale
\begin{equation}
H_L^s=\dom(L^{s/2}), \quad (s\ge 0)
\end{equation}
where $L$ is a non-negative, self-adjoint operator in $H$ and $L^{s/2}$ is defined via the spectral theorem. We further assume that $L$ is injective, such that 
$
\Vert x \Vert_{H_L^s} = \Vert L^{s/2} x \Vert_H 
$ ($x\in H^s_L$)
defines a norm that turns $H^s_L$ into a Hilbert space in its own right.

\begin{lem} \label{heinz} Suppose that  $N\in \B(H)$ is a  self-adjoint, non-negative operator. Assume additionally that $N\in \B^\times(H,H^1_L)$ and set $\underline n = \Vert N^{-1} \Vert_{\B(H^1_L,H)}$.
\begin{enumerate}[label=(\roman*)]
\item \label{heinz1} Then $\sqrt N \in \B^\times(H,H^{1/2}_L)$ and $\Vert \sqrt N^{-1} \Vert_{\B(H^{1/2}_L,H)} = \sqrt {\underline n}$.
\item \label{heinz2} There is a stability estimate of the form
\begin{equation}
\Vert x \Vert^2_{(H^{1/2}_L)^*} \le \underline n \cdot \langle N x, x\rangle_H \quad \text{ for all } x\in H.
\end{equation} 
\end{enumerate}
\end{lem}

\begin{proof} Given non-negative, self-adjoint operators $S,T$ in $H$, let us write $S\le T$ if $\dom T^{1/2} \subset \dom S^{1/2}$ and $\Vert S^{1/2} x \Vert_H \le \Vert T^{1/2} x \Vert_H$ for all $x\in \dom T^{1/2}$.  Operator monotonicity of the square root is then the assertion that
\begin{equation}
0\le S \le T \quad \Longrightarrow\quad  0\le \sqrt S \le \sqrt T.
\end{equation}
We use this in the operator inequalities  $0 \le \underline n^{-2} N^{-2} \le L$ and $0\le L \le \bar n^2 N^{-2}$ (where $\bar n = \Vert N \Vert_{\B(H,H_L^1)}$), both of which follow from the asserted isomorphism property. Taking square roots we get $0\le \underline n^{-1} N^{-1} \le L^{1/2} \le \bar n N^{-1}$. In particular  $ \range N^{1/2} = \dom N^{-1/2} = H_L^{1/2}$, which implies $N^{1/2}\in \B^\times (H,H_A^{1/2})$ (injectivity is clear); further $\Vert N^{-1/2} \Vert_{\B(H_A^{1/2},H)} \le \underline{n}^{1/2}$, which proves part \ref{heinz1}.
For \ref{heinz2} write $B$ for the unit ball in $H_L^{1/2}$, then by Cauchy-Schwarz
\begin{equation*}
\Vert  x \Vert_{(H_L^{1/2})^*} = \sup_{y\in B} \langle x, y \rangle_H = \sup_{y\in B} \langle N^{1/2} x, N^{-1/2} y \rangle_H \le \sup_{y\in B} \Vert N^{-1/2} y \Vert_H \cdot \langle Nx,x\rangle_H^{1/2}
\end{equation*}
for all $x\in H$. By part \ref{heinz1},  we have $\sup_{y\in B} \Vert N^{-1/2} y \Vert_H \le \underline n$, as desired.
\end{proof}

Combining the previous two results we can prove Theorem \ref{zernikestability}.

\begin{proof}[Proof of Theorem \ref{zernikestability}]
Put $H=L^2(\DD,\C^m)$,  $L=\L$ and let $\Phi \in C_K^4(\DD,\u(m))$. Then by Theorem \ref{isoprop}, the normal operator $N_\Phi$ satisfies the requirements of Lemma \ref{heinz} and $\underline n(\Phi) = \Vert N^{-1}_\Phi \Vert_{\B(\tilde H^1(\DD),L^2(\DD))}$ depends continuously on $\Phi$ in the $C^4$-topology.  As $\tilde H^{-1/2}(\DD)$ is the dual space of $\tilde H^{1/2}(\DD)$,  we have
\begin{equation}
\Vert f \Vert_{\tilde H^{-1/2}(\DD)}^2 \le \underline n(\Phi) \langle N_\Phi f ,f \rangle_{L^2(\DD)} =  \underline n(\Phi) \Vert I_\Phi f \Vert_{L^2_\lambda(\partial_+S\bar \DD)}^2
\end{equation}
for all $f\in L^2(\DD,\C^m)$, which is the desired stability estimate.
\end{proof}

In the remaining section we carry out the program outlined above, ultimately leading to the proof of Theorem \ref{isoprop}.

\medskip 

{\bf (Part A) Norm bounds for certain integral operators.} For a smooth function $A:\DD\times S^1\times \R \rightarrow \C$ we denote with  $T_A$ the singular integral operator defined by
\begin{equation}\label{defta}
T_Af(x) = \int_{S^1} \int_\R A(x,v,t) f(x+tv) \d t \d v,\quad x\in \DD.
\end{equation}
Here it is understood that $f\in C_c^\infty(\DD)$ (a smooth function with compact support in $\DD$) is extended by zero outside of $\DD$. From \cite[Lemma B.1]{DPSU07} we know that  $T_A$ is a classical pseudodifferential operator on $\DD$ of order $-1$ with principal symbol
\begin{equation}
\sigma_{A}(x,\xi) = 2 \pi (A(x,\xi^\perp,0) + A(x,-\xi^\perp,0)),\quad x\in \DD,\xi\in S^1.
\end{equation}
The present section provides bounds on the operator norm of $T_A$ in different functional analytic settings in terms of the norms
\begin{equation}\label{defnorm}
\Vert A \Vert_{C^k} = \sup \{ \Vert \partial_{x_1}^{\alpha_1}\partial_{x_2}^{\alpha_2}\partial_v^{\alpha_3}\partial_t^{\alpha_4} A\Vert_{L^\infty(\DD\times S^1 \times \R)}: \alpha\in \N_0^4, \vert \alpha \vert \le k\}.
\end{equation}
Precisely,  we establish (writing $\ord$ for the order of a differential operator):

\begin{prop}\label{lemeco} 
Let $P,Q$ be differential operators on $\DD$ with coefficients in $C^\infty(\bar \DD)$ and put $m=\ord P + \ord Q$. Suppose that either $0\le m \le 1$ or that $\sigma_A\equiv 0$ and $0\le m \le 2$. Then
\begin{equation}
\Vert PT_A Q f \Vert_{H^s(\DD)} \le C \Vert A \Vert_{C^{2m}}\cdot \Vert f \Vert_{H^s(\DD)},\quad 0\le s \le m, f\in C_c^\infty(\DD)
\end{equation}
for a constant $C>0$ which only depends on $P$ and $Q$.
\end{prop}

The proposition remains true  if instead $A$ takes values in $\C^{m\times m}$ and $f\in C_c^\infty(\DD,\C^m)$; we confine ourselves to the scalar case for notational simplicity.

\smallskip

The strategy to prove Proposition \ref{lemeco} is to rewrite expressions like $\partial_{x_j} (T_A f)$  and $T_A(\partial_{x_j} f)$  in a form where all derivatives fall on $A$. To this end we 
introduce a type of Hilbert transform that features in the next Lemma and is witness to the singular nature of $T_A$.

\begin{definition}[Hilbert transform] For a function $f\in C_c^\infty(\DD)$ we define \begin{equation}
Hf(x,v)= \lim_{\epsilon\rightarrow 0} \frac 1 \pi \int_{\vert t \vert>\epsilon} f(x+tv) \frac{\d t }{t} \equiv \lim_{\epsilon\rightarrow 0} H_\epsilon f(x,v),\quad (x,v)\in \DD \times S^1. \label{hdef}
\end{equation}
\end{definition}
The transform $Hf(x,v)$ corresponds to the standard Hilbert transform (up to a sign in the argument) along rays parallel to the line $\{sv:s\in \R\}\subset \R^2$. The following facts are thus derived as in the standard case and will be given without proof: For $f\in C^\infty_c(\DD)$ the limit in \eqref{hdef} exists pointwise; further the transform $Hf(x,v)$ is continuous and satisfies
\begin{equation} \label{hilbertbound}
\sup_{v\in S^1}\Vert Hf(\cdot, v) \Vert_{L^2(\DD)} \le \Vert f \Vert_{L^2(\DD)} .
\end{equation}

\begin{lem} \label{lempi} Let $A:\DD\times S^1\times \R\rightarrow \C$ be smooth and $f\in C_c^\infty(\DD)$. Then 
\begin{equation}\label{lempi1}
T_A(\partial_{x_1}f)(x) =\pi  \int_{S^1} \tilde A_0(x,v) Hf(x,v) \d v + T_{\tilde A_1} f(x),\quad  x\in \DD
\end{equation}
where $\tilde A_0(x,v)$ and $\tilde A_1(x,v,t)$ are derived from a Taylor expansion $A(x,v,t) = A_0(x,v) + tA_1(x,v,t)$ via the formulas
\begin{equation}
\tilde A_0 =  \partial_v\left( v_2 A_0\right ) \quad \text{and} \quad \tilde A_1 = \partial_v(v_2 A_1)- \partial_t(tv_1 A_1).
\end{equation}
In particular we have $\Vert \tilde A_0 \Vert_{C^k} \lesssim_k \Vert A \Vert_{C^{k+1}}$ and $\Vert \tilde A_1 \Vert_{C^k} \lesssim_k \Vert A\Vert_{C^{k+2}}$ ($k\in \N_0$).
\end{lem}

Let us remark on a few aspects of the lemma. First, one has $A_0(x,v) = A(x,v,0)$ and $A_1(x,v,t) = \int_0^1 \partial_t A(x,v,st) \d s$ and thus all of $A_0,A_1,\tilde A_0,\tilde A_1$ are smooth functions. Further, if the principal symbol $\sigma_A\equiv 0$, then  $\tilde A_0(x,v)$ is an even function  in the sense that $\tilde A_0(x,v)=\tilde A_0(x,-v)$. As $Hf(x,v)$ is odd,  the first term in \eqref{lempi1} vanishes. This reflects the fact that the operator $T_A$ is then a pseudodifferential operator of order $-2$ and implies that the Lemma may be applied twice.
We tacitly use the Lemma also for $x_2$-derivatives, where it is proved analogously.


\begin{proof}[Proof of Lemma \ref{lempi}]  Using the expansion $A(x,v,t)=A_0(x,v) + tA_1(x,v,t)$, decompose
\begin{equation} \label{lempi2}
T_A(\partial_{x_1}f)(x) =  T_{A_0}(\partial_{x_1 }f)(x) +  T_{A_1}(t\partial_{x_1} f)(x).
\end{equation}
We investigate the two terms separately, starting with the second one.\ Writing $\tilde f (x,v,t) = f(x+tv)$, a straightforward calculation reveals that
\begin{equation}\label{lempi3}
t\partial_{x_1} \tilde f (x,v,t) =t  v_1 \partial_t \tilde f(x,v,t) - v_2\partial_v\tilde f(x,v,t)
\end{equation}
and thus
\begin{equation*}
\begin{split}
T_{A_1}(t\partial_{x_1} f)(x) = \int_{S^1}\int_\R tv_1 A_1(x,v,t)  \partial_t \tilde f (x,v,t)\d t \d v  - \int_{S^1}\int_\R v_2 A_1(x,v,t)  \partial_v \tilde f (x,v,t)\d t \d v.
\end{split}
\end{equation*}
We can now partially integrate with respect to $t$ and $v$ respectively without collecting boundary terms, as the integrand is compactly supported. This yields
\begin{equation*}
= \int_{S^1}\int_\R \left[\partial_v(v_2 A_1) - \partial_t(tv_1 A_1)\right](x,v,t) f(x+tv) \d t \d v = T_{\tilde A_1} f(x),
\end{equation*}
i.e. the second term in \eqref{lempi1}. To examine the first term in \eqref{lempi2}, we define
\begin{equation*}
I_\epsilon(x) = \int_{S^1}\int_{\vert t \vert>\epsilon } A_0(x,v) \partial_{x_1}\tilde f(x,v,t) \d t \d v,\quad x\in \DD,\epsilon >0
\end{equation*}
and note that $I_\epsilon(x)\rightarrow T_{A_0}(\partial_{x_1}f)(x)$ $(\epsilon \rightarrow 0$) for all $x\in \DD$ due to the dominated convergence theorem. Use \eqref{lempi3} once more to write
\begin{equation*}
\begin{split}
I_\epsilon(x) = \int_{S^1} v_1 A_0(x,v) \int_{\vert t \vert>\epsilon} \partial_t \tilde f(x,v,t) \d t \d v 
 +  \int_{S^1} \int_{\vert t \vert>\epsilon} \left[ -v_2 A_0(x,v)\right]  \partial_{v}\tilde f(x,v,t) \frac{\d t }{t} \d v
\end{split}
\end{equation*}
and call the two summands $I_\epsilon^1(x)$ and $I^2_\epsilon(x)$ respectively. By the fundamental theorem of calculus, the first summand also equals
\begin{equation*}
I_\epsilon^1(x) = \int_{S^1} v_1 A_0(x,v) [f(x-\epsilon v) - f(x+\epsilon v)] \d v 
\end{equation*}
which vanishes in the limit $\epsilon\rightarrow 0$ by the dominated convergence theorem. To rewrite the second summand, we integrate by parts in $v$ to obtain
\begin{equation*}
I^2_\epsilon(x) = \int_{S^1} \partial_v\left( v_2 A_0(x,v)\right )\left ( \int_{\vert t \vert >\epsilon} \frac{f(x+tv)}{t} \d t \right) \d v = \pi \int_{S^1} \tilde A_0(x,v)  H_\epsilon(x,v) \d v ,
\end{equation*}
which approaches the first term in \eqref{lempi1} as $\epsilon\rightarrow 0$. The Lemma is proved.
\end{proof}

\begin{proof}[Proof of Proposition \ref{lemeco}]
Let $f\in C_c^\infty(\DD)$. By Minkowski's integral inequality
\begin{equation*}
\Vert T_A f \Vert_{L^2(\DD)} \le \int_{S^1} \int_\R \Vert A(\cdot, v,t) f(\cdot + tv) \Vert_{L^2(\DD_x)} \d t \d v.
\end{equation*}
For $\vert t \vert>2$ (the diameter of $\DD$), the integrand vanishes. For $\vert t \vert \le 2$ we bound it from above by   $\Vert A \Vert_{L^\infty} \cdot \Vert f(\cdot + t v ) \Vert_{L^2(\DD_x)} \le \Vert A \Vert_{L^\infty} \cdot \Vert f \Vert_{L^2(\DD_x)}$. This yields 
\begin{equation} \label{pfeco1}
\Vert T_A f \Vert_{L^2(\DD)} \le 8 \pi  \cdot \Vert A \Vert_{L^\infty}\cdot \Vert f \Vert_{L^2(\DD)}
\end{equation}
and thus proves the Proposition for the case $0=s=m$. By means of the preceding Lemma  we will bootstrap this inequality to prove
\begin{equation}\label{pfeco2}
\Vert \partial_{x}^\alpha \partial_x^\beta T_A \partial^\gamma_x f \Vert_{L^2(\DD)} \lesssim_{\alpha,\beta,\gamma} \Vert A \Vert_{C^{2m}} \cdot \Vert f \Vert_{H^s(\DD)}
\end{equation}
for multi-indices $\alpha,\beta,\gamma\in \N_0^2$ with $\vert \alpha \vert \le s$ and $\vert \beta \vert + \vert \gamma \vert \le m$.  To prove \eqref{pfeco2}, we use the Leibniz rule to write
\begin{equation}
 \partial_{x}^\alpha \partial_x^\beta T_A \partial^\gamma_x f(x) = \sum {\alpha + \beta \choose \mu } T_{(\partial^\mu_x A )} (\partial_x^{\nu + \gamma}f)(x),
\end{equation}
where the sum runs over multi-indices $\mu,\nu$ with $\mu + \nu = \alpha + \beta$. We will estimate the summands separately, distinguishing the cases $ m \le \vert \mu \vert  \le 2m$ and $0\le \vert \mu \vert \le m-1$. In the first case we necessarily have $\vert \nu  + \gamma \vert \le s$ and thus, using \eqref{pfeco1}, the respective term can be estimated as
\begin{equation}\label{pfeco3}
\Vert T_{(\partial^\mu_x A )} \partial_x^{(\nu + \gamma)}f \Vert_{L^2(\DD)} \le 8\pi \Vert \partial_x^\mu A \Vert_{L^\infty} \cdot \Vert \partial_x^{\nu + \gamma} f \Vert_{L^2(\DD)} \le 8 \pi \Vert A \Vert_{C^{2m}} \Vert f \Vert_{H^s(\DD)}.
\end{equation}
From now on we assume that $0\le \vert \mu \vert \le m-1$; the remaining proof consists of a case study $m=1$ vs. $m=2$ (note that $m=0$ is already proved).\\
Suppose that $m=1$, then $\vert \mu \vert=0$ and $\vert \nu + \gamma \vert \le s+1$. We may assume that $\vert \nu \vert =1$ or $\vert \gamma \vert =1$ (if both are zero, then \eqref{pfeco1} suffices to estimate the respective term). If $\vert \nu \vert =1$, we can use Lemma \ref{lempi} to write
\begin{equation}
T_{A}\partial_{x}^{\nu  + \gamma}f(x) = \pi \int_{S^1} \tilde A_0(x,v) H(\partial^\gamma_xf) (x,v) \d v + T_{(\tilde A_1)} \partial_x^\gamma f(x)
\end{equation}
and \eqref{pfeco1} as well as \eqref{hilbertbound} to estimate
\begin{equation} 
\begin{split}
\Vert T_{A}\partial_{x}^{\nu  + \gamma}f \Vert_{L^2(\DD)} &\le 2\pi^2 \Vert \tilde A_0 \Vert_{L^\infty} \Vert \partial_x^\gamma f \Vert_{L^2(\DD)} + 8\pi \Vert \tilde A_1 \Vert_{L^\infty} \Vert \partial_x^\gamma f \Vert_{L^2(\DD)}\\
& \lesssim   \Vert A \Vert_{C^2} \cdot \Vert f \Vert_{H^s(\DD)}.
\end{split}
\end{equation}
If $\vert \gamma \vert = 1$ one can use the same method thus the case $m=1$ is proved.\\
Next suppose that $m=2$ and $\sigma_A\equiv 0$. If $ \vert \mu \vert =1$, then again we have $\vert \nu + \gamma \vert \le s+1$ and, similar to \eqref{pfeco3} we obtain
\begin{equation}
\Vert T_{(\partial_x^\mu A)} \partial^{\nu +\gamma} f \Vert_{L^2(\DD)} \lesssim \Vert \partial_x^\mu A \Vert_{C^2} \cdot \Vert f \Vert_{H^s(\DD)},
\end{equation}
which gives the desired bound as $\Vert \partial_x^\mu A \Vert_{C^2} \le \Vert A \Vert_{C^4}$. It remains to deal with the case $\vert \mu \vert =0$, where we merely have $\vert \nu + \gamma \vert \le s + 2$. Now {\it two} derivatives have to be shuffled over to $A$, but this is possible as $\sigma_A\equiv 0$ and thus {Lemma~\ref{lempi}} may be applied twice. Indeed, writing $\nu + \gamma = \nu'+\gamma' + \delta$ with $\vert \delta \vert\le 2$ and $\vert \nu ' + \gamma' \vert\le s$ we have
\begin{equation}
T_A\partial_x^{\nu + \gamma}f(x) = \pi \int_{S^1} \tilde B_0(x,v) H(\partial_x^{\nu'+\gamma'}f)(x,v) \d v + T_{\tilde B_1} (\partial_x^{\nu'+\gamma'}f)(x),
\end{equation}
where now $\Vert \tilde B_0 \Vert_{L^\infty} \lesssim \Vert A \Vert_{C^3}$ and $\Vert \tilde B_1\Vert_{L^\infty} \lesssim \Vert A \Vert_{C^4}$. Similar to above we  find
\begin{equation}
\Vert T_A \partial_x^{\nu + \gamma} f\Vert_{L^2(\DD)} \lesssim \Vert A \Vert_{C^4} \cdot \Vert \partial_x^{\nu'+\gamma'}f\Vert_{L^2(\DD)} \le \Vert A \Vert_{C^4} \cdot \Vert f\Vert_{H^s(\DD)}.
\end{equation}
Combining the preceding estimates, inequality \eqref{pfeco2} follows. As the coefficients of $P$ and $Q$ as well as their derivatives are easily  bounded, we get
\begin{equation}
\Vert \partial^\alpha PT_A Qf \Vert_{L^2(\DD)} \lesssim_{P,Q}\Vert A \Vert_{C^{2m}} \Vert f \Vert_{H^s(\DD)},\quad 0\le \vert \alpha \vert\le s \le m
\end{equation}
which proves the Proposition for integral $s$. General values for $0\le s \le m$ can be obtained by an interpolation argument and the proof is complete.
\end{proof}

\medskip 
{\bf (Part B) Local Lipschitz estimates.} Recall that we had defined  $N_\Phi = I_\Phi^*I_\Phi$ and $K_\Phi = N_\Phi - N_0$.  We bound their operator norms as follows:

\begin{prop}[Local Lipschitz estimates] \label{lipschitz} Let $K\subset \DD$ be compact. Then for all $\Phi,\Psi \in C_K^\infty(\DD,\u(m))$ and  $0\le s \le 2$ we have
\begin{eqnarray}\
\Vert N_\Phi - N_\Psi \Vert_{\B(L^2(\DD), \tilde H^1(\DD))} \le  c_1(\Phi,\Psi) \cdot \Vert \Phi - \Psi \Vert_{C_K^2(\DD)} \label{lip1} \\
\Vert \L  K_\Phi - \L K_\Psi \Vert_{\B(\tilde H^s(\DD))} \le  c_2(\Phi,\Psi) \cdot \Vert \Phi - \Psi \Vert_{C^4_K(\DD)} \label{lip2} \\
\Vert K_\Phi\L - K_\Psi\L \Vert_{\B(\tilde H^s(\DD))} \le  c_2(\Phi,\Psi) \cdot \Vert \Phi - \Psi \Vert_{C^4_K(\DD)} \label{lip3}
\end{eqnarray}
where 
$c_i(\Phi,\Psi) = C (1+ \Vert \Phi \Vert_{C^{2i}(S\bar \DD)} \vee \Vert \Psi \Vert_{C^{2i}(S\bar \DD)} )^{6i}$ for some $C=C(K)>0$ ($i=1,2$).
\end{prop}

Of course we have $K_\Phi - K_\Psi = N_\Phi - N_\Psi$; the formulation above points to the fact that $\L K_\Phi = \L K_\Phi - \L K_0$ itself lies in $\B(\tilde H^s(\DD))$, while $\L N_\Phi$ does not.


\begin{proof} Let $R_\Phi$ be an integrating factor for $\Phi$ as in \textsection \ref{boif}. Then $R_\Phi$ is smooth on $S\DD$ (in fact constant $\equiv \id$ near the `glancing' region $\partial_+S\bar \DD \cap \partial_-S\bar \DD)$) and allows the following integral representation of the normal operator (cf.\,display above (6.20) in \cite{MNP20}):
\begin{equation*}
N_\Phi f(x) = \int_{S^1}(\mu^\sharp(x,v))^{-1} R_\Phi(x,v) \int_{-\tau(x,-v)}^{\tau(x,v)} R_\Phi^{-1}(x+tv,v) f(x+tv) \d t \d v,
\end{equation*}
where $\mu(x,v) = - x\cdot v$ and $\mu^\sharp$ is an extension of $\mu \in C^\infty(\partial_+S\bar\DD)$ to all of $S\DD$,  constant along the geodesic vector field $X$.
Given $f\in C_c^\infty(\DD,\C^m)$ we can thus write 
\begin{equation}\label{defs}
\begin{split}
(N_\Phi  - N_\Psi)f(x) &= \int_{S^1}\int_\R (\mu^\sharp(x,v))^{-1} S(x,v,t) f(x+tv) \d t \d v,\\
\text{ with} ~~ S(x,v,t) &= R_\Phi(x,v)  R_\Phi^{-1}(x+tv,v) - R_\Psi(x,v)  R_\Psi^{-1}(x+tv,v),
\end{split}
\end{equation}
where $R_\Psi$ is defined analogously to $R_\Phi$. 

\smallskip 
 
Our first objective is to prove \eqref{lip1} and \eqref{lip2},\eqref{lip3} with right hand side replaced by (a positive multiple of)  $\Vert S\Vert_{C^{2i}}$ for $i=1$ and $i=2$ respectively. 
To this end we will make repeated use of Proposition \ref{lemeco} above, noting that $N_\Phi-N_\Psi$ is a singular integral operator as in  \eqref{defta} with `integral kernel'
\begin{equation}
A(x,v,t) = \mu^\sharp(x,v)^{-1} S(x,v,t),\quad (x,v,t)\in S\DD\times \R =\DD\times  S^1\times \R,
\end{equation}
which is smooth and has all of its derivative bounded. Indeed, if we choose $K'\subset S\DD$ compact, such that lines through $(x,v)\in S\DD\backslash K'$ avoid the set $K$, then $S(x,v,t)=0$ for $(x,v)\in S\DD\backslash K'$. But on $K'$, the function $\mu^\sharp(x,v)$ is bounded below by some $\mu_K>0$, which means that $A(x,v,t)$ is smooth and satisfies
\begin{equation}
\Vert A \Vert_{C^k} \lesssim_{k,K} \Vert S \Vert_{C^k},\quad k\in \N_0
\end{equation}
where the norm $\Vert \cdot \Vert_{C^k}$ is defined as in \eqref{defnorm}.\\
Proposition \ref{lemeco} (for $P=\partial_{x_1},\partial_{x_2}$ and $Q=\id$) implies for $f\in C_c^\infty(\DD)$ that
\begin{equation}
	\Vert \partial_{x_j }(N_\Phi -N_\Psi) f \Vert_{L^2(\DD)}= \Vert \partial_{x_j} T_A f\Vert_{L^2(\DD)}  \lesssim \Vert A \Vert_{C^2} \Vert f \Vert_{L^2(\DD)}\quad (j=1,2), 
\end{equation} thus $\Vert N_{\Phi} - N_\Psi \Vert_{\B(L^2(\DD),\tilde H^1(\DD))} \lesssim_K \Vert S \Vert_{C^2}$, as desired. Here we used that $H^1(\DD)\subset \tilde H^1(\DD)$.

\smallskip

Next, we consider \eqref{lip2} and \eqref{lip3}. It suffices to prove bounds for $s=0$ and $s=2$, as all intermediate values follow by interpolation. If $s=0$, then again Proposition \ref{lemeco} (this time for $P=\L,Q=\id$ and $P=\id,Q=\L$, noting that $\sigma_A\equiv 0$) yields
\begin{equation}
\Vert  \L  N_\Phi - \L N_\Psi \Vert_{\B(L^2(\DD))} \vee \Vert N_\Phi\L - N_\Psi\L \Vert_{\B(L^2(\DD))} \lesssim_K \Vert S \Vert_{C^4}.
\end{equation}
Inequality \eqref{lip3} for $s=2$ follows in fact from the previous display: For  $f\in \tilde H^2(\DD)$, 
\begin{equation}
\Vert (N_\Phi - N_\Psi) \L f \Vert_{\tilde H^2(\DD)} = \Vert \L(N_\Phi - N_\Psi) (\L f)\Vert_{L^2(\DD)} \lesssim_K\Vert S \Vert_{C^4}\cdot  \Vert \L f \Vert_{L^2(\DD)}
\end{equation}
and as $\Vert \L f \Vert_{L^2(\DD)} = \Vert f \Vert_{\tilde H^2(\DD)}$, this yields $\Vert N_\Phi \L - N_\Psi \L\Vert_{\L(\tilde H^2(\DD))} \lesssim_K \Vert S \Vert_{C^4}$, as desired.\\
In order to prove \eqref{lip2} for $s=2$, more work is needed, as $\L$ needs to be commuted with the normal operators.
To this end we localise $N_\Phi-N_\Psi$ as in \cite{MNP20}: Cover $\bar \DD$ by opens $U_1,\dots, U_n$ such that if $U_i\cap U_j\neq \emptyset$ and either set touches the boundary, then $U_i\cup U_j\subset B$ for a ball $B$ which is disjoint from $K$. Let $\psi_1,\dots,\psi_n$ be a subordinate smooth partition of unity, then $N_\Phi - N_\Psi = \sum_{ij} \psi_i(N_\Phi - N_\Psi) \psi_j = \sum_{ij} T_{A_{ij}}$, where 
\begin{equation}
A_{ij}(x,v,t) =  \psi_i(x) (\mu^\sharp(x,v))^{-1} S(x,v,t) \psi_j(x+tv).
\end{equation}
Note that if $U_i\cap U_j \neq \emptyset$ and either set touches the boundary, then  lines that connect points in $U_i$ and $U_j$ never enter $K$. But along such lines the integrating factors $R_\Phi$ and $R_\Psi$ both equal $\id$ which means that $S(x,v,t)$ vanishes and thus $A_{ij}\equiv 0$. The non-trivial contributions to $N_\Phi - N_\Psi$ thus stem from the following two cases: (I) $U_i \cap U_j \neq \emptyset$ and $\overline{U_i\cup U_j}\subset \DD$ or (II) $U_i\cap U_j  = \emptyset$.  We claim that in both cases
\begin{equation}
 \Vert \L^2 T_{A_{ij}} f \Vert_{L^2(\DD)} \lesssim_{K} \Vert A_{ij} \Vert_{C^4}\cdot  \Vert \L f \Vert_{L^2(\DD)},\quad f\in C^\infty(\bar \DD),
\end{equation}
which, when combined, yields \eqref{lip2}.\\
In case (I) we may choose a cut-off function $\chi \in C_c^\infty(\DD)$ with $\chi \equiv 1$ on $\overline{ U_i \cup U_j}$
and apply Proposition \ref{lemeco} (for $P= \L\chi,Q=\id$) to the effect that
\begin{equation}
\begin{split}
\Vert \L^2 T_{A_{ij}} f \Vert_{L^2(\DD)} &= \Vert \L^2 (\chi T_{A_{ij}} \chi f )\Vert_{L^2(\DD)}  \le \Vert \L (\chi  T_{A_{ij}} \chi f) \Vert_{H^2(\DD)}\\
& \lesssim \Vert A_{ij} \Vert_{C^4}\cdot \Vert \chi f \Vert_{H^2(\DD)},
\end{split}
\end{equation}
which gives the desired bound, as $\Vert \chi f \Vert_{H^2} \lesssim \Vert \L f \Vert_{L^2}$ (\cite[Lem.\,6.15]{MNP20}). In case (II),
\begin{equation}
\L^2 T_{A_{ij}}f(x) = T_{A_{ij}}\L (\L f)(x) + [\L^2,T_{A_{ij}}] f(x)
\end{equation}
The first term could be bounded with Proposition \ref{lemeco}, but in fact we treat both summands on the same footing by considering the relevant Schwartz kernels.  The Schwartz kernel of $T_{A_{ij}}$ equals (cf.\cite[(B.3)]{DPSU07})
\begin{equation}
k_{ij}(x,y) = 2 \vert x-y \vert^{-1}  A_{ij}\left(x,\frac{x-y}{\vert x- y \vert},\vert x- y \vert\right)
\end{equation}
which is in $C^\infty(\bar \DD\times \bar \DD)$ and satisfies $\Vert k_{ij} \Vert_{C^4(\bar \DD\times \bar \DD)} \lesssim \Vert A_{ij} \Vert_{C^4}$.  Now the Schwartz kernels of $T_{A_{ij}}\L$ and $[\L^2,T_{A_{ij}}]$ are given by
\begin{equation}
\L_y k_{ij}(x,y)\quad \text{ and } \quad \L_x^2 k_{ij}(x,y) - \L^2_yk_{ij}(x,y)
\end{equation}
respectively, where the subscripts indicate which slot $\L$ acts on. Hence
\begin{equation*}
\begin{split}
\Vert \L^2 T_{A_{ij}} f \Vert_{L^2(\DD)} &\le  \Vert \L_y k_{ij} \Vert_{L^\infty(\bar \DD\times \bar \DD)} \Vert \L f \Vert_{L^2(\DD)} +  \Vert \L_x^2 k_{ij} - \L^2_y k_{ij}\Vert_{L^\infty(\bar \DD\times \bar \DD)} \Vert f \Vert_{L^2(\DD)}\\
& \lesssim \Vert A_{ij} \Vert_{C^4} \cdot \Vert \L f \Vert_{L^2(\DD)},
\end{split}
\end{equation*}
and we have finished the first part of the proof.

\smallskip 

So far we have established \eqref{lip1} and \eqref{lip2}\eqref{lip3} with right hand side replaced by (a constant times)
$ \Vert S \Vert_{C^{2i}}$ for $ i=1,2$ respectively, where $S=S_{\Phi,\Psi}$ was defined in \eqref{defs}. We complete the proof by showing that
\begin{equation}\label{pflip}
\Vert S_{\Phi,\Psi} \Vert_{C^{k}} \lesssim \left(1+\Vert \Phi \Vert_{C^k_K(\DD)} \vee \Vert \Psi \Vert_{C_K^k(\DD)} \right)^{2k} \cdot \Vert \Phi - \Psi \Vert_{C_K^k(\DD)},\quad k\in \N_0.
\end{equation}
To this end put $\tilde R_\Phi(x,v,t) = R_\Phi(x+tv,v)$ (where we think of $R_\Phi$ as extended to a global solution $(X+\Phi)R_\Phi =0$ on $S\R^2=\R^2\times S^1$). Then $S_{\Phi,\Psi}$ can be rewritten as
\begin{equation}
S_{\Phi,\Psi} = (R_\Phi - R_\Psi)\cdot \tilde R_\Phi^{-1} + R_\Psi\cdot(\tilde R_\Phi^{-1} - \tilde R_{\Psi}^{-1}).
\end{equation}
and the result follows from estimate \eqref{horse2} above.
\end{proof}

\medskip

{\bf (Part C) Isomorphism property for non-smooth potentials.}  As final ingredient of the proof of Theorem \ref{isoprop} we record how the Lipschitz estimates above can be used to get mapping properties for non-smooth potentials $\Phi$.

\begin{lem} \label{ncor} Let $K\subset \DD$ be compact.
\begin{enumerate}[label=(\roman*)]
 \item \label{ncor1} For $\Phi \in C^2_K(\DD,\u(m))$ we have $N_\Phi \in \B(L^2(\DD),\tilde H^1(\DD))$
 \item \label{ncor2} For $\Phi \in C^4_K(\DD,\u(m))$ we have $K_\Phi \L,\L K_\Phi \in \B(\tilde H^s(\DD))$ for $0\le s\le 2$.
\end{enumerate}
\end{lem}

\begin{proof}
We first show that for  $\Phi \in C(\bar \DD,\u(m))$ we have $N_\Phi \in \B(L^2(\DD))$ and that  $\Phi \mapsto N_\Phi$ is continuous as map $C(\bar \DD,\u(m)) \rightarrow \B(L^2(\DD))$.  Indeed,  if $\Phi,\Psi \in C(\bar \DD,\u(m))$ and $R_\Phi$ and $R_\Psi$ are integrating factors as in \textsection \ref{boif}, then for $f\in L^2(\DD,\C^m)$ with $\Vert f \Vert_{L^2(\DD)} \le 1$ we have
\begin{equation*}
\begin{split}
 \Vert I_\Phi f - I_\Psi f \Vert_{L^2_\lambda(\partial_+S\bar \DD)} & =  \Vert\I\left((R_\Phi^{-1} - R_\Psi^{-1}) f\right)\Vert_{L^2_\lambda(\partial_+S\bar \DD)}\\
 & \lesssim \Vert R_\Phi^{-1} - R_\Psi^{-1} \Vert_{L^\infty(\DD)} \lesssim \Vert \Phi - \Psi \Vert_{L^\infty(\DD)},
 \end{split}
\end{equation*}
where we used boundedness of $\I:L^2(S\DD)\rightarrow L^2_\lambda(\partial_+S\bar \DD)$ (Lem.\,\ref{lem_phaseforward}) together with the estimate \eqref{horse1} from above. This implies that $\Phi \mapsto I_\Phi$ is a continuous map into the space $\B(L^2(\DD),L^2_\lambda(\partial_+S\bar \DD))$ and as taking adjoints and multplication of operators are continuous operations, also $\Phi\mapsto N_\Phi$ is continuous in the desired topologies.
\\
For part \ref{ncor1} approximate $\Phi \in C_K^2(\DD,\u(m))$ by a sequence $(\Phi_n)\subset C^\infty_K(\DD,\u(m))$, then Proposition \ref{lipschitz} implies that $(N_{\Phi_n})\subset \B(L^2(\DD),\tilde H^1(\DD))$ is Cauchy and thus has a limit $N_\infty \in  \B(L^2(\DD),\tilde H^1(\DD))$. At the same time $N_{\Phi_n} \rightarrow N_\Phi$ in $\B(L^2(\DD))$, as follows from the previously discussed continuity properties. Viewing $\B(L^2(\DD),\tilde H^1(\DD))$ as continuously embedded subspace of $\B(L^2(\DD))$, we must thus have $N_\Phi  = N_\infty \in \B(L^2(\DD),\tilde H^1(\DD))$. Part \ref{ncor2} is proved by the same argument, again viewing $\B(H^s(\DD))$ as continuously embedded subspace of $\B(L^2(\DD))$, which is possible as $\tilde H^s(\DD)\subset L^2(\DD)$ is dense.
\end{proof} 

\begin{proof}[Proof of Theorem \ref{isoprop}] Note that part \ref{isoprop1} follows  immediately from the local Lipschitz estimates in Proposition \ref{lipschitz} and the preceding lemma. For part \ref{isoprop2} we first prove injectivity and start with the observation that
\begin{equation}\label{iso1}
\ker N_\Phi \cap H^1(\DD) = 0\quad \text{ for } \Phi \in C^1_c(\DD,\u(m)),
\end{equation}
which follows from a recently established $L^2$-$H^1$-stability estimate \cite[Theorem 5.3]{MNP21} by means of an approximation argument.
In order to upgrade this to injectivity on $L^2(\DD)$ we employ a bootstrapping argument, which requires $\Phi \in C_c^4(\DD,\u(m))$. Indeed, for an attenuation of that regularity Lemma \ref{ncor} implies that $\L^{1/2}K_\Phi$ continuously maps $L^2(\DD)\rightarrow \tilde H^1(\DD)$, $\tilde H^1(\DD)\rightarrow \tilde H^2(\DD)$ and $\tilde H^2(\DD)\rightarrow \tilde H^3(\DD)$. Hence, if $f\in L^2(\DD)$ satisfies $N_\Phi f= 0$, then
\begin{equation}
- f= - \L^{1/2} N_0 f = \L^{1/2} K_\Phi f
\end{equation}
and we can bootstrap to obtain $f\in \tilde H^3(\DD)$. As $\tilde H^3(\DD)\subset H^1(\DD)$ by \eqref{zernikeembedding}, it follows from \eqref{iso1} that $f$ must be zero.\\
We proceed as in the proof of \cite[Theorem 6.18]{MNP20}, showing that the range of $N_\Phi$ is both closed and dense. Lemma \ref{ncor} implies that $K_\Phi:L^2(\DD)\rightarrow \tilde H^2(\DD)\subset \tilde H^1(\DD)$ is compact for $\Phi \in C_c^4(\DD,\u(m))$, thus turning $N_\Phi = N_ 0 +K_\Phi \in \B(L^2(\DD),\tilde H^1(\DD))$ into a Fredholm operator; in particular its range is closed. Next, consider $f\in (\range N_\Phi)^\perp=\ker N_\Phi^* \subset \tilde H^{-1}(\DD)$, where ${}^*$ denotes the adjoint in the dual pairing $\tilde H^{-1}$-$\tilde H^1$. Then similar to above we have $-f = \L^{1/2} K_\Phi^* f $; but  $K_\Phi \L^{1/2}$ maps $L^2(\DD)\rightarrow \tilde H^1(\DD)$ and thus 
$\L^{1/2}K_\Phi^* = (K_\Phi\L^{1/2})^*$ maps $\tilde H^{-1}(\DD)\rightarrow L^2(\DD)$, which implies $f\in L^2(\DD)$. Then $N_\Phi f = N_\Phi^* f = 0$ and $f$ must be zero by the previously established injectivity. This proves that the range is closed we have established the isomorphism property. \\ The continuity statement in \ref{isoprop2} follows directly from the one in part \ref{isoprop1} and the fact that taking inverses is a continuous operation. This completes the proof.
\end{proof}

\section{Proofs for Section \ref{genth}}

\subsection{Posterior Asymptotics}\label{posthorn}

We first establish both posterior contraction and convergence of any of its global maximisers towards the projection $\theta_{\true,D}$ onto $E_D$ of the ground truth, for forward maps $\mathscr G$ satisfying only the `global' Condition \ref{ganzwien}. The proofs follow ideas developed in \cite{MNP21, GN20, NW20, NVW18} for specific inverse problems and earlier work \cite{V00, vdVvZ08, GV17} for direct problems.

\begin{thm}\label{recycle}
Let $(e_n: n \in \mathbb N_0)$ and $\Theta$ be as in Condition \ref{linsp}. Let the posterior distribution $\Pi(\cdot|(Y_i, X_i)_{i=1}^N)$ from (\ref{postbus}) arise from the Gaussian process prior $\Pi$ in (\ref{gprior}) with $\alpha$ satisfying (\ref{towardsinfinity}) and data $(Y_i, X_i)_{i=1}^N \sim P_{\theta}^N$ from (\ref{model}). Suppose $\mathscr G$ from (\ref{fwdmap}) satisfies Condition \ref{ganzwien} for some $0 < \gamma \le 1$, that $D \le A_1 N^{d/(2\alpha+d)}$ and that $\theta_\true \in \tilde H^\alpha \cap \Theta$ satisfies (\ref{bias}). 

For every $c>0$ there exist $\bar C, \bar c$ large enough such that any maximiser $\hat \theta_\MAP$ of the posterior density $\pi(\cdot|(Y_i, X_i)_{i=1}^N)$ over $E_D$ (cf.~(\ref{maptheta}))  necessarily satisfies 
\begin{equation}\label{mapinvrate}
P_{\theta_\true}^N \left(\|\hat \theta_\MAP - \theta_{\true,D}\|_{L^2_\zeta} \le \bar C \delta_N^{\gamma}, ~\|\hat \theta_\MAP\|_{\tilde H^\alpha} \le \bar C  \right) \ge 1 - \bar c e^{-c N \delta_N^2}.
\end{equation}
Moreover, for every $b>0$ we can choose $\bar C$ large enough such that also
\begin{align}\label{postinvrate}
P_{\theta_\true}^N &\left(\big\{\Pi(\theta \in E_D: \|\theta-\theta_{\true,D}\|_{L^2_\zeta} \le \bar C \delta_N^\gamma, \|\theta\|_{\tilde H^{\alpha}} \le \bar C \big|(Y_i, X_i)_{i=1}^N) \ge 1- e^{-b N \delta_N^2}\big\}, \mathcal C_{N,K} \right) \\
&~~~~~~~~~~\ge 1- \bar c e^{-\bar b N \delta_N^2} \notag
\end{align}
for some $\bar b>0$, where, for some large enough $K$,
\begin{equation}\label{CN}
\mathcal C_{N, K} = \left\{\int_{E_D} e^{\ell_N(\theta)-\ell_N(\theta_\true)} d\Pi(\theta) \ge \Pi(B_N)  \exp\{-KN\delta_N^2\}\right\},
\end{equation} 
and where, for $u$ from Lemma \ref{getreal}, $B_N$ can be taken to be any measurable subset of $$ \big\{\theta \in E_D: \|\mathscr G(\theta)-\mathscr G(\theta_\true)\|_{L^2_\lambda} \le \delta_N/u \big\}.$$ 
\end{thm}

\begin{rem}[Global consistency] \normalfont
In (\ref{mapinvrate}) and (\ref{postinvrate}) one can replace $\theta_{\true,D}$ by $\theta_\true$ if either $\|\theta_\true- \theta_{\true,D}\|_{L^2_\zeta} = O(\delta_N^\gamma)$ (true if $\theta_\true \in \tilde H^\alpha$ and $D \simeq N \delta_N^2$), or if Condition \ref{ganzwien}c) is strengthened to hold for all $\theta, \theta' \in \Theta$ rather than only $\theta, \theta' \in E_D$. This is natural for statistical consistency results but for the log-concave approximations from Theorem \ref{onetrickpony}, contraction near $\theta_{\true,D} \in E_D$ is relevant.
\end{rem}

\subsubsection{Convergence rate for the MAP estimate (proof of (\ref{mapinvrate}))}

For $(Y_i, X_i)_{i=1}^N$ arising from (\ref{model}) and forward map satisfying Condition \ref{ganzwien}, consider any maximiser
 \begin{equation}\label{maptheta}
\hat \theta_\MAP \in \arg \max_{\theta \in E_D} \left[-\frac{1}{2N}\sum_{i=1}^N|Y_i - \mathscr G(\theta)(X_i)|_V^2 - \frac{\delta_N^2}{2} \|\theta\|^2_{\tilde H^\alpha} \right],~~~\delta_N = N^{-\frac{\alpha}{2\alpha+d}},
\end{equation}
of the posterior density of the posterior measure (\ref{postbus}). [Existence of $\hat \theta_\MAP$ will be addressed later, see after (\ref{score}).] We will write, for $\Lambda_\alpha$ defined after (\ref{gprior}), $$I(\theta):=\frac{1}{2}\|\theta\|_{\tilde H^\alpha}^2 = \frac{1}{2} \theta^T \Lambda_\alpha \theta,~\theta \in E_D,$$ and notationally identify $\mathbb R^D \simeq E_D$ in what follows. We denote the empirical measure on $V \times \mathcal X$ induced by the $(Y_i, X_i)$'s as
\begin{equation} \label{empm}
P_N= \frac{1}{N}\sum_{i=1}^N \delta_{(Y_i, X_i)}, ~~\text{so that } \int h dP_N = \frac{1}{N} \sum_{i=1}^N h(Y_i, X_i)
\end{equation}
for any measurable map $h: V \times \mathcal X \to \mathbb R$. We denote the marginal $dy d\lambda(x)$-densities (for the Lebesgue measure $dy$ on $V \simeq \mathbb R^{d_V}$) of the laws $P_\theta$ generating the data in (\ref{model}) by
\begin{equation}\label{densities}
p_{\theta}(y,x) = (2\pi)^{-d_V/2} \exp\Big\{-\frac{1}{2} |y-\mathscr G(\theta)(x)|_V^2 \Big\},~~ y \in V, x \in \mathcal X.
\end{equation}
\begin{lem} \label{basiclemma}
Let $\hat \theta_\MAP$ be any maximiser in (\ref{maptheta}), and denote by $\theta_{\true,D}$ the $L^2_\zeta$-projection of $\theta_\true$ onto $E_D$. We have ($P^N_{\theta_\true}$-a.s.)
$$\frac{1}{2}\|\mathscr G(\hat \theta_\MAP)- \mathscr G(\theta_\true)\|_{L_\lambda^2}^2 +   \delta_N^2 I(\hat\theta_\MAP) \le  \int \log \frac{p_{\hat \theta_\MAP}}{p_{\theta_{\true,D}}} d(P_N-P_{\theta_\true}) +  \delta_N^2 I(\theta_{\true,D}) + \frac{1}{2}\|\mathscr G(\theta_{\true,D}) -\mathscr G(\theta_\true)\|_{L_\lambda^2}^2.$$
\end{lem}
\begin{proof}
The proof is elementary and the same as the one of Lemma 4.10 in \cite{NW20}.
\end{proof}

\begin{prop}\label{fwdrate}
Let $\alpha>d$. Suppose $\|\theta_\true\|_{\tilde H^\alpha} \le \bar c$ and that (\ref{bias}) holds. Then, for any $c \ge 1$ we can choose $C=C(c, \bar c,u)$ large enough so that every $\hat \theta_\MAP$ maximising (\ref{maptheta}) satisfies
\begin{equation}
P^N_{\theta_\true}\left(\frac{1}{2}\|\mathscr G(\hat \theta_\MAP)- \mathscr G(\theta_\true)\|_{L_\lambda^2}^2 +   \delta_N^2 I(\hat\theta_\MAP) > C\delta_N^2  \right) \lesssim  e^{-cN\delta_N^2}.
\end{equation}
\end{prop}
\begin{proof}
We define functionals $$\tau(\theta, \theta') = \frac{1}{2}\|\mathscr G(\theta)-\mathscr G(\theta')\|_{L^2_\lambda}^2 + \delta_N^2 I(\theta),~~\theta \in E_D, \theta' \in \tilde H^\alpha,$$ and empirical processes 
\begin{equation}
W_N(\theta) =  \int \log \frac{p_{\theta}}{p_{\theta_{\true,D}}} d(P_N-P_{\theta_\true}), ~W_{N,0}(\theta) =  \int \log \frac{p_{\theta}}{p_{\theta_\true}} d(P_N-P_{\theta_\true}), ~\theta \in E_D,
\end{equation}
 so that 
 \begin{equation} \label{jansquery}
 W_N(\theta)= W_{N,0}(\theta) - W_{N,0}(\theta_{\true,D}),~~~\theta \in E_D.
 \end{equation}
Using the previous lemma it suffices to bound
\begin{equation*}
P^N_{\theta_\true}\left(\tau(\hat \theta_\MAP, \theta_\true) > C\delta_N^2, W_N(\hat \theta_\MAP) \ge \tau(\hat \theta_\MAP, \theta_\true) -  \delta_N^2 I(\theta_{\true,D}) - \|\mathscr G(\theta_{\true,D}) -\mathscr G(\theta_\true)\|_{L^2}^2/2 \right).
\end{equation*} 
Since $$I(\theta_{\true,D}) =  \|\theta_{\true,D}\|_{\tilde H^\alpha}^2/2 \le \|\theta_\true\|_{\tilde H^\alpha}^2/2 \le \bar c^2/2 \text{ and }\|\mathscr G(\theta_{\true,D}) -\mathscr G(\theta_\true)\|_{L^2_\lambda}^2 \le \delta_N^2/4u^2$$ by hypothesis, using (\ref{jansquery}) we can choose $C$ large enough so that the last probability is bounded by 
\begin{align}
&P^N_{\theta_\true}\left(\tau(\hat \theta_\MAP, \theta_\true) > C\delta_N^2, |W_N(\hat \theta_\MAP)| \ge \tau(\hat \theta_\MAP, \theta_\true)/2 \right) \notag \\
& \le \sum_{s =1}^\infty P^N_{\theta_\true}\left(\sup_{\theta \in E_D: 2^{s-1}C \delta_N^2 \le \tau(\theta, \theta_\true) \le 2^sC \delta_N^2}|W_{N,0}(\theta)| \ge  2^{s}C \delta_N^2/8 \right) +P^N_{\theta_\true}\big(|W_{N,0}(\theta_{\true,D})| \ge  C \delta_N^2/8 \big) \notag \\
&  \le 2\sum_{s=1}^\infty P^N_{\theta_\true}\left(\sup_{\theta \in\Theta_s}|W_{N,0}(\theta)| \ge  2^{s}C \delta_N^2/8 \right),  \label{peel}
\end{align} 
where, for $s \in \mathbb N$ we have defined
\begin{align}
\Theta_s &:= \big\{\theta \in E_D: \tau(\theta, \theta_\true) \le 2^sC \delta_N^2 \big\} \\
&= \big\{\theta \in E_D: \|\mathscr G(\theta) - \mathscr G(\theta_\true)\|^2_{L^2_\lambda} + \delta_N^2 \|\theta\|_{\tilde H^\alpha}^2 \le 2^{s+1}C \delta_N^2\big\}, \notag
\end{align}
and where we have used that $\theta_{\true,D} \in \Theta_1$ (as $C>0$ is large). To proceed, notice that $$NW_{N,0}(\theta)=\ell_N(\theta) - \ell_N(\theta_\true) - E_{\theta_\true}[\ell_N(\theta) -  \ell_N(\theta_\true)]$$ and that,
for $(Y_i, X_i) \sim^\iid P_{\theta_\true}$,
\begin{align}\label{twoproc}
\ell_N(\theta) - \ell_N(\theta_\true) &= -\frac{1}{2}\sum_{i=1}^N\big[|\mathscr G(\theta_\true)(X_i)-\mathscr G(\theta)(X_i) + \varepsilon_i|_V^2 - |\varepsilon_i|_V^2 \big] \notag \\
&=- \sum_{i=1}^N \langle \mathscr G(\theta_\true)(X_i)-\mathscr G(\theta)(X_i)),\varepsilon_i\rangle_V -\frac{1}{2}\sum_{i=1}^N|\mathscr G(\theta_\true)(X_i)-\mathscr G(\theta)(X_i)|_V^2,
\end{align}
and we deal with these two empirical processes separately. For the `multiplier type' process, by a union bound,
\begin{align}\label{sum1}
\sum_{s =1}^\infty P^N_{\theta_\true}\left(\sup_{\theta \in \Theta_s}\Big|\frac{1}{\sqrt N}\sum_{j=1}^{d_V}\sum_{i=1}^N h_{\theta,j}(X_i)\varepsilon_{i,j}\Big| \ge  \sqrt N 2^{s}C \delta_N^2/16 \right) \notag \\
\le d_V \max_{1 \le j \le d_V}\sum_{s =1}^\infty P^N_{\theta_\true}\left(\sup_{\theta \in \Theta_s}|Z_{N,j}(\theta)| \ge  \sqrt N 2^{s}C \delta_N^2/16d_V \right)
\end{align}
where, for $j$ indexing the vector field entries in the basis of $V$ for which $\varepsilon_{ij} \sim^{\iid} \mathcal{N}(0,1)$, $$Z_{N,j}(\theta)=\frac{1}{\sqrt N}\sum_{i=1}^N h_{\theta,j}(X_i)\varepsilon_{i,j},~~ h_{\theta, j} = \mathscr G(\theta_\true)_j - \mathscr G(\theta)_j,~~\theta \in \Theta = \Theta_s, s \in \mathbb N,$$ is a process of the type considered in Lemma \ref{mixchain} with $\mathcal H =\{h_{\theta,j}: \theta \in \Theta_s\}$ for $j$ fixed. We will apply that lemma with $P^X=\lambda$ and bounds
\begin{equation}\label{wimpy}
E^Xh^2_{\theta,j}(X) = \int_\mathcal X [\mathscr G(\theta_\true)_j(x) - \mathscr G(\theta)_j(x)]^2d\lambda(x) \le \|\mathscr G(\theta)-\mathscr G(\theta_\true)\|_{L_\lambda^2}^2 \le 2^{s+1}C \delta_N^2 =: \sigma_s^2,
\end{equation}
and for envelope constant $U=2U_\mathscr G$ from Condition \ref{ganzwien},
\begin{equation}\label{envelope}
\|h_{\theta,j}\|_\infty \le 2\sup_\theta \|\mathscr G(\theta)\|_\infty \le U<\infty 
\end{equation}
uniformly in all $\theta \in \Theta_s$. To construct suitable bounds for the entropy functionals $J_2, J_\infty$ in Lemma \ref{mixchain}, we note that on each slice $\sup_{\theta \in \Theta_s}\|\theta\|_{\tilde H^{\alpha}} \le \sqrt 2 (C2^s)^{1/2}$, which by Proposition \ref{mettop} implies that for some constant $K$,
\begin{equation*}
\log N\big(\Theta_s, \|\cdot\|_{L^2_\zeta}, \rho \big) \le K\Big(\frac{\sqrt C 2^{s/2}}{\rho}\Big)^{d/\alpha},~~\rho>0,
\end{equation*}
as well as, for $\bar \tau$ from Proposition \ref{mettop},
\begin{equation*}
\log N\big(\Theta_s, \|\cdot\|_{\infty}, \rho \big) \le D^{\bar \tau} K\Big(\frac{\sqrt C 2^{s/2}}{\rho}\Big)^{d/\alpha},~~\rho>0.
\end{equation*}
Since the maps $\theta \mapsto \mathscr G(\theta)$ and then also $\theta \mapsto \mathscr G(\theta)_j, j =1, \dots, d_V$, are $C_\mathscr G$-Lipschitz for the $\|\cdot\|_{L^2}$ and $\|\cdot\|_\infty$-norms by Condition \ref{ganzwien}, the coverings of $\Theta_s$ induce respective coverings for the $h_{\theta,j}$'s with bounds
\begin{equation}
\log N\big(\{h_{\theta,j}=\mathscr G(\theta)_j-\mathscr G(\theta_\true)_j: \theta \in \Theta_s\}, \|\cdot\|_{L^2_\lambda}, \rho \big) \le K'\Big(\frac{\sqrt C 2^{s/2}}{\rho}\Big)^{d/\alpha},~~\rho>0,
\end{equation}
as well as
\begin{equation}
\log N\big(\{h_{\theta,j}=\mathscr G(\theta)_j-\mathscr G(\theta_\true)_j: \theta \in \Theta_s\}, \|\cdot\|_{\infty}, \rho \big) \le D^{\bar \tau} K'\Big(\frac{\sqrt C 2^{s/2}}{\rho}\Big)^{d/\alpha},~~\rho>0,
\end{equation}
for any $j$, where $K'=K'(K,C_\mathscr G)$. As a consequence, for $\alpha>d$ and $J_2(\mathcal H), J_\infty(\mathcal H)$ defined in Lemma \ref{mixchain},
\begin{align}\label{entbd}
&J_2(\mathcal H) \lesssim \int_0^{4\sigma_s}\Big(\frac{\sqrt C 2^{s/2}}{\rho}\Big)^{d/2\alpha}d\rho \lesssim (C2^s)^{d/4\alpha}\sigma_s^{1-\frac{d}{2\alpha}}, \\ \notag
& J_\infty(\mathcal H) \lesssim D^{\bar \tau} \int_0^{4U}\Big(\frac{\sqrt C 2^{s/2}}{\rho}\Big)^{d/\alpha}d\rho \lesssim D^{\bar \tau}(C2^s)^{d/2\alpha} U^{1-\frac{d}{\alpha}}.
\end{align}
The sum in (\ref{sum1}) can now be bounded by Lemma \ref{mixchain} with $x=cN2^s \delta_N^2$ and the choices of $\sigma_s, U$ in (\ref{wimpy}), (\ref{envelope}), for $C=C(c, d_V)>0$ large enough,
\begin{equation}\label{sum11}
\sum_{s \in \mathbb N} P^N_{\theta_\true}\left(\sup_{\theta \in \Theta_s}|Z_{N,j}(\theta)| \ge  \sqrt N \sigma_s^2/(32d_V) \right) \le 2 \sum_{s \in \mathbb N} e^{-c2^sN\delta_N^2} \lesssim e^{-c N\delta_N^2}
\end{equation}
since then, by definition of $\delta_N$, for $\alpha>d$, every $c>0$ and $C$ large enough, the  quantities 
\begin{equation}\label{gaussiantail}
\mathcal J_2(\mathcal H) \lesssim  (C2^s)^{d/4\alpha} (2^{s/2} \sqrt C\delta_N)^{1-\frac{d}{2\alpha}} \lesssim (1/\sqrt C) \sqrt N \sigma^2_s,~~\sigma_s \sqrt x  \le \frac{\sqrt c}{\sqrt C} \sqrt N \sigma_s^2,
\end{equation}
 and, using also the hypotheses $\bar \tau \le 1, D \le A_1 N\delta_N^2$,
 \begin{equation}\label{exponentialtail}
 \frac{1}{\sqrt N}\mathcal J_\infty(\mathcal H) \lesssim D^{\bar \tau} \frac{(C2^s)^{d/2\alpha}}{\sqrt N} \lesssim C^{(d/2\alpha)-1} \sqrt N \sigma^2_s,~~\frac{xU}{\sqrt N} = \frac{cU}{C} \sqrt N \sigma_s^2
 \end{equation}
can all be made smaller than any constant multiple of $\sqrt N \sigma_s^2$. As these upper bounds do not depend on $j$ they hold as well for the full expression in (\ref{sum1}) after adjusting the constants.

We now turn to the second term in (\ref{twoproc}), bounded by
\begin{equation}\label{sum2}
\sum_{s \in \mathbb N} P^N_{\theta_\true}\left(\sup_{\theta \in \Theta_s}|Z'_N(\theta)| \ge  \sqrt N 2^{s}C \delta_N^2/(32d_V) \right)
\end{equation}
where $Z'_N$ is now the centred empirical process 
$$Z'_N(\theta) = \frac{1}{\sqrt N} \sum_{i=1}^N(h_\theta-E^X h_\theta(X)\big),~~\text{ with }~  \mathcal H = \{ h_\theta=|\mathscr G(\theta)-\mathscr G(\theta_\true)|_V^2 : \theta \in \Theta_s\}$$
to which we will again apply Lemma \ref{mixchain}. Just as in (\ref{envelope}) the envelopes of this process are uniformly bounded by a fixed constant $U=U(U_\mathscr G, d_V)$, which implies in particular that the bounds (\ref{entbd}) also apply to $\mathcal H$ as then, for some constant $c_U>0$, $$ \|h_\theta- h_{\theta'}\|_\infty \leq c_U \|\mathscr G(\theta)-\mathscr G(\theta')\|_\infty$$ and similarly if $\|\cdot\|_\infty$ are replaced by $L^2_\lambda$-norms. Moreover, using (\ref{envelope}) on each slice $\Theta_s$ the weak variances are bounded by $$E^Xh_\theta^2(X) \le c'_U \|\G(\theta)-\G(\theta_\star)\|_{L^2_\lambda}^2 \le c_U' \sigma^2_s$$ with $\sigma_s$ as in (\ref{wimpy}) and some $c'_U>0$. We see that all bounds required to sum (\ref{sum1}) as in (\ref{sum11}) apply to the process $Z_N'$ as well, and hence the series in (\ref{peel}) is indeed bounded as required in the proposition, completing the proof.
\end{proof}

From the third part of Condition \ref{ganzwien} we now deduce (\ref{mapinvrate}). By the preceding Proposition \ref{fwdrate} we can restrict to events 
\begin{equation} \label{TN}
T_N := \big\{\|\mathscr G(\hat \theta_\MAP)- \mathscr G(\theta_\true)\|_{L^2_\lambda}^2 \le 2C \delta_N^2, \|\hat \theta_\MAP\|_{\tilde H^\alpha} \le \sqrt{2C}\big\}
\end{equation}
of sufficiently high $P^N_{\theta_\true}$-probability. Next by (\ref{bias}) and since $\theta_\true \in \tilde H^\alpha$, we have for large enough $M$ that 
$$\|\mathscr G(\hat \theta_\MAP)- \mathscr G(\theta_{\true,D})\|_{L^2_\lambda}^2 \le M \delta_N^2$$ as well as $$\|\hat \theta_\MAP\|_{\tilde H^\alpha} + \|\theta_{\true,D}\|_{\tilde H^\alpha} \le M.$$
We can thus invoke the third hypothesis of Condition \ref{ganzwien} with $\theta=\hat \theta_\MAP,\theta'=\theta_{\true,D}\in E_D$ to deduce the final conclusion (\ref{mapinvrate}).

\subsubsection{Posterior contraction (proof of (\ref{postinvrate}))}

We start with an application of Theorem \ref{twotrickpony} to the posterior measure (\ref{postbus}) arising from the Gaussian prior $$\Pi=\Pi_N=\mathcal L(\theta'/(\sqrt{N}\delta_N))= \mathcal{N}(0, (N \delta_N^2 )^{-1}\Lambda_\alpha)$$ on $E_D \subset \Theta$ from (\ref{gprior}). Verification of the small ball condition (\ref{smballgen}) will ultimately be reduced to the small deviation estimate 
\begin{equation}\label{uncent}
\mathrm{Pr} (\|\theta'\|_{E_D} < b' \sqrt N \delta_N^2) \ge e^{-bN \delta_N^2}
\end{equation}
for our choice of $\delta_N$, all $b'>0$ and some $b=b(b', \alpha, d)$, as follows from Theorem 1.2 in \cite{LL99} with Banach space $E=L^2_\zeta$, RKHS $\tilde H^\alpha \cap E_D \subset \tilde H^\alpha$ and the metric entropy bound  (\ref{mettopa}) below.

The last estimate implies that for $\theta_{\true,D}$ the $L^2_\zeta$-projection of $\theta_\true \in \tilde H^\alpha$ onto $E_D$ and any $d''>0$ there exists $d'=d'(\alpha, d, d'', \|\theta_\true\|_{\tilde H^\alpha})>0$ such that
\begin{align}\label{gensmb} 
\Pi_N(\theta \in E_D: \|\theta-\theta_{\true,D}\|_{E_D} <d''\delta_N) & \ge e^{-N\delta_N^2 \|\theta_{\true,D}\|^2_{\tilde H^\alpha}/2} \mathrm {Pr} (\|\theta'\|_{E_D} < d'' \sqrt N \delta_N^2)  \ge e^{-d' N\delta_N^2}
\end{align}
where we have used Corollary 2.6.18 in \cite{GN16}.  Now we use Lemma \ref{getreal}, (\ref{bias}), Condition \ref{ganzwien}b) and the preceding bound with $d''=1/(2u C_\mathscr G)$ to deduce that
\begin{align} \label{smball}
\Pi_N(\mathcal B_N) &\ge \Pi_N(\theta \in E_D: \|\mathscr G(\theta)-\mathscr G(\theta_\true)\|_{L^2_\zeta} < \delta_N/u)\notag \\
& \ge  \Pi_N(\theta \in E_D: \|\mathscr G(\theta)-\mathscr G(\theta_{\true,D})\|_{L^2_\lambda} <\delta_N/(2u)) \notag \\
&\ge \Pi_N(\theta \in E_D: \|\theta-\theta_{\true,D}\|_{E_D} <\delta_N/ (2u C_\mathscr G))  \ge  e^{-\bar d N\delta_N^2}
\end{align}
for the present choice of $\delta_N$ and some $\bar d = \bar d(u, C_\mathscr G, \alpha, d, \|\theta_\true\|_{\tilde H^\alpha})>0$, verifying (\ref{smballgen}) for our prior and $a=1, A=\bar d$.

\smallskip

We next construct suitable sets $\Theta_N \subset E_D \cap \Theta$ in Theorem \ref{twotrickpony}. From (\ref{uncent}) and Borell's Gaussian iso-perimetric inequality (in the form of Theorem 2.6.12 in \cite{GN16}) one shows that given $B>0$ we can find $m$ large enough (depending on $\bar d, B$) such that 
$$\Pi_N\big(\theta=\theta_1 + \theta_2 \in E_D: \|\theta_1\|_{E_D} \le m \delta_N, \|\theta_2\|_{\tilde H^\alpha} \le m\big) \ge 1- e^{-BN\delta_N^2}.$$ The proof proceeds, e.g., as in Lemma 5.17 in \cite{MNP21} (or Lemma 17 in \cite{GN20},with $\kappa=0$), and will not be repeated here.  Next (\ref{specdim}) and the hypothesis on $D \le A_1 N \delta_N^2$ imply that for $\underline C > m$ large enough we have
\begin{equation} \label{rkhstrick}
\|\theta_1\|_{\tilde H^\alpha} \lesssim D^{\alpha/d} \|\theta_1\|_{E_D} \leq (A_1 N \delta_N^2)^{\alpha/d}m \delta_N \le \underline C/2
\end{equation}
and then also
\begin{equation}
\Pi_N(\Theta_N^c) \le e^{-BN\delta_N^2} \text{ where }\Theta_N :=\{\theta \in E_D: \|\theta\|_{\tilde H^\alpha} \le \underline C\},
\end{equation}
so that (\ref{excess}) can be verified for any $B>0$ by choosing $m$ and then $\underline C$ large enough. To bound the Hellinger-complexity of $\Theta_N$ as required in (\ref{mettopq}), we use Conditions \ref{ganzwien}a), b), Proposition \ref{lucienoldhand} and (\ref{mettopa}) to the effect that
\begin{align*}
\log N(\{p_\theta: \theta \in \Theta_N\}, h, \delta_N) &\lesssim \log N(\{\mathscr G(\theta): \theta \in \Theta_N\}, \|\cdot\|_{L^2_\lambda}, \delta_N/c_1) \\
&\lesssim \log N(\{\Theta_N, \|\cdot\|_{L^2_\zeta}, \delta_N/(c_1C_\mathscr G)) \\
& \lesssim \log N(\{\theta: \|\theta\|_{\tilde H^\alpha} \le \underline C\}, \|\cdot\|_{L^2_\zeta}, \delta_N/(c_1 C_\mathscr G)) \le C N \delta_N^2
\end{align*}
for some $C=C(\underline C, c_1, b_1, C_\mathscr G, \alpha, d)$. Having verified the hypotheses of Theorem \ref{twotrickpony} we deduce from (\ref{hellcont}) and again Proposition \ref{lucienoldhand} that for every $b>0$ we can choose $L,  m,\underline C$ large enough so that as $N \to \infty$ and for some $\bar b>0$,
\begin{equation}\label{rate1}
P^N_{\theta_\true}\Big(\Pi(\{\theta: \|\mathscr G(\theta)-\mathscr G(\theta_\true)\|_{L_\lambda^2} \le c_0L \delta_N\} \cap \Theta_N|(Y_i, X_i)_{i=1}^N) \ge 1- e^{-bN\delta_N^2}\Big) =O(1-e^{-\bar b N\delta_N^2}).
\end{equation}
To complete the proof note that for $M = M(c_0L, \underline C, u)$ large enough and using (\ref{bias}), the sets $$\Theta_N := E_D \cap \{\theta: \|\mathscr G(\theta)-\mathscr G(\theta_{\true,D})\|_{L^2_\lambda} \le M \delta_N\} \cap \{\theta: \|\theta\|_{\tilde H^\alpha} + \|\theta_{\true,D}\|_{\tilde H^\alpha} \le M \}$$ contain the set featuring in (\ref{rate1}) and hence have at least as much probability. We can finally invoke the stability Condition \ref{ganzwien}c) to deduce that
\begin{equation}
\Theta_N \subset \{\theta: \|\theta-\theta_{\true,D}\|_{L^2_\zeta} \le \bar C \delta_N^\gamma, \|\theta\|_{h^\alpha} \le \bar C \},
\end{equation}
for some $\bar C>0$ large enough. We conclude that the posterior probability of the last event is  lower bounded by $1 - e^{-bN\delta_N^2},$ with the desired $P^N_{\theta_\true}$-probability, proving (\ref{postinvrate}) with the exception of the probability bounds for the events $\mathcal C_{N,K}$ covered by the next lemma.

\subsubsection{Lower bounds for the posterior normalising factor}

The following auxiliary result with centring $\vartheta=\theta_\true$ gives the required bound on the $\mathcal C_{N,K}$-sets in Theorem \ref{recycle}, and with centring $\vartheta=\hat \theta_\MAP$ will be used in the proof of Theorem \ref{onetrickpony} that follows next.
\begin{lem}\label{expsmall}
Let $\mathscr G$ be as in Theorem \ref{recycle} and let $\nu$ be a probability measure on some (measurable) set 
\begin{equation} \label{ischgl}
B_N \subseteq \big\{\theta \in \Theta: \|\mathscr G(\theta)-\mathscr G(\theta_\true)\|_{L^2}^2 \le \delta_N^2/u^2 \big\}, 
\end{equation}
where $u=2(U_\mathscr G +1)$. Let either $\vartheta=\theta_\true \in \Theta$ or $\vartheta=\hat \theta_\MAP$ from (\ref{maptheta}). Then for $\ell_N$ from (\ref{llemp}) we can find $K=K(u)>0$ large enough such that for some $\bar b>0$
\begin{equation}
P^N_{\theta_\true}\left(\int_{B_N} e^{\ell_N(\theta)-\ell_N(\vartheta)}d\nu(\theta) \le e^{-K N\delta_N^2} \right) \lesssim e^{-\bar b N \delta_N^2}.
\end{equation}
\end{lem}
\begin{proof}
We proceed as in Lemma 7.3.2 in \cite{GN16}. From Jensen's inequality (applied to $\log$ and $\int (\cdot)d\nu$) and recalling (\ref{llemp}), (\ref{empm}), (\ref{densities}), the probability in question is bounded by
$$P^N_{\theta_\true}\left(\int \int_{B_N} \log \frac{p_\theta}{p_{\vartheta}}d\nu(\theta)d(P_N-P_{\theta_\true}) \le -K \delta_N^2 - \int \int_{B_N} \log  \frac{p_\theta}{p_{\vartheta}}d\nu(\theta)dP_{\theta_\true} \right).$$  Now a standard computation with likelihood ratios (e.g., Lemma 23 in \cite{GN20}) we see that 
\begin{align*}
-\int \log  \frac{p_\theta}{p_{\vartheta}}dP_{\theta_\true} &= -\int \log  \frac{p_\theta}{p_{\theta_\true}}dP_{\theta_\true} + \int \log  \frac{p_{\theta_\true}}{p_{\vartheta}}dP_{\theta_\true} \\
& = \frac{1}{2}\|\mathscr G(\theta)-\mathscr G(\theta_\true)\|_{L^2}^2 - \frac{1}{2} \|\mathscr G(\vartheta) - \mathscr G(\theta_\true)\|_{L_2}^2 \le  \delta_N^2/(2u^2)
\end{align*}
 so that for $K=K(u)$ large enough and using also Fubini's theorem, the last probability can be bounded by 
\begin{align*}
&P^N_{\theta_\true}\left(\sqrt N \int \int_{B_N} \log \frac{p_{\theta_\true}}{p_\theta}d\nu(\theta)d(P_N-P_{\theta_\true}) \ge K \sqrt N\delta_N^2/4\right) \\
& + P^N_{\theta_\true}\left(\sqrt N \int \log \frac{p_{\vartheta}}{p_{\theta_\true}}d(P_N-P_{\theta_\true}) \ge K \sqrt N \delta_N^2/4\right).
\end{align*}
For the first probability we use (\ref{densities}) and proceed as in and after (\ref{twoproc}) and need to consider $Z_N$ as in Lemma \ref{mixchain} for singleton $\mathcal H$ consisting of either $$h_{1,j}(x) = \int_{B_N} (\mathscr G(\theta)_j(x)- \mathscr G(\theta_\true)_j(x)) d\nu(\theta),~~j=1, \dots, d_V \text{ fixed},$$
in the multiplier case $Z_N= \sum_{i=1}^N h_{1,j}(X_i) \varepsilon_{i,j}$ from (\ref{multipl}), and
$$h_2(x) =  \int_{B_N} |\mathscr G(\theta)(x)- \mathscr G(\theta_\true)(x)|_V^2 d\nu(\theta)$$ in the centred case (\ref{centp}). We apply Bernstein's inequality (\ref{bernie}) with fixed envelopes $U=U(U_\mathscr G)>0$ in view of $\|\mathscr G(\theta)-\mathscr G(\theta_\true)\|_\infty \le 2U_\mathscr G$ from Condition \ref{ganzwien}, with variance estimates $$E^Xh^2_{1,j}(X) \le  \delta_N^2/u^2 \equiv \sigma_1^2$$ in the first case (using again Jensen's inequality) and $$E^Xh^2_2(X) \le 4U_\mathscr G^2 \int_{B_N}\|\mathscr G(\theta)-\mathscr G(\theta_\true)\|_{L^2}^2 d\nu(\theta) \le 4(U_\mathscr G/u)^2 \delta_N^2 \equiv \sigma_2^2$$ for the second case, as well as $x= N\sigma_i^2, =1,2$. In both cases we can choose $K=K(U, L)$ large enough such that $L(\sigma_i \sqrt x + Ux/\sqrt N) \le K \sqrt N \delta_N^2/4$ in (\ref{bernie}) and hence the first of the last two displayed probabilities is bounded as desired, completing the proof of $\vartheta=\theta_\true$.

When bounding the second probability with $\vartheta=\hat \theta_\MAP$ we can in view of Proposition \ref{fwdrate} restrict to bounding 
$$P^N_{\theta_\true}\left(\sup_{\theta \in E_D: \|\theta\|_{\tilde H^\alpha} \le  2C, \|\mathscr G(\theta)-\mathscr G(\theta_\true)\|^2_{L^2_\lambda} \le  2C \delta^2_N} \sqrt N\Big|\int \log \frac{p_\theta}{p_{\theta_\true}}d(P_N-P_{\theta_\true}) \Big| \ge K \sqrt N\delta_N^2/4\right).$$ This term corresponds to the empirical process bounded in and after (\ref{peel}) for $s=1$. Choosing $K$ large enough the proof there now applies directly, giving the desired exponential bound and completing the proof of the lemma.
\end{proof}

\subsection{Proof of Theorem \ref{onetrickpony}}

The construction of $\tilde \ell_N$ in (\ref{surrogate}) is such that
\begin{equation}\label{38}
\tilde \ell_N(\theta)=\ell_N(\theta) \text{ for any } \theta ~s.t.~\|\theta-\theta_{\true,D}\|_{E_D}\le \frac{3}{8}\eta.
\end{equation}
for any given $\eta>0$, so Part A) follows directly from the definition of $\tilde \pi(\cdot|(Y_i, X_i)_{i=1}^N)$ in (\ref{surro}). We also note that a maximiser $\hat \theta_\MAP$ exists under Condition \ref{ganzwien}: Indeed, the function $q_N$ to be maximised over $E_D$ in (\ref{maptheta}) then satisfies $q_N(\theta)<q_N(0)$ for all $\theta$ such that $\|\theta\|_{\tilde H^\alpha}$ exceeds some positive constant $k$ (and $(Y_i, X_i)_{i=1}^N$ fixed). Then on the compact set $M=\{\theta \in E_D: \|\theta\|_{\tilde H^\alpha} \le k\}$, $q_N$ is continuous, and hence attains its maximum at some $\hat \theta_M \in M$, which must be a global maximiser of $q_N$ since $q_N(\hat \theta_M) \ge q_N(0) > \inf_{\theta \in M^c} q_N(\theta)$.  [The maximiser can be taken to be measurable, cf.~Exercise 7.2.3 in \cite{GN16}).]

\subsubsection{Geometry of the surrogate posterior.}

\smallskip

We first note that by Lemma \ref{youtube} (and Condition \ref{bock} with $\mathcal B$ the set from (\ref{eq:B:ass})) the event $\mathcal S_N^{(i)} \subset (V \times \mathcal X)^N$ defined by
\begin{equation*}
	\begin{split}
		\mathcal S^{(i)}_N := & \Big\{\inf_{\theta \in \mathcal B} \lambda_\min(-\nabla^2 \ell_N(\theta)) \ge \frac{c_0}{2}ND^{-\kappa_0} \Big\}  \\
		& \cap  \Big\{\sup_{\theta\in \mathcal B}\Big[|\ell_N(\theta)|+\|\nabla \ell_N(\theta)\|_{E_D}+\|\nabla^2\ell_N(\theta)\|_\mathrm{op}\Big]\le N(5c_1D^{\kappa_1}+1) \Big\}
	\end{split}
\end{equation*}
satisfies $P_{\theta_\true}^N(\mathcal S_N^{(i)}) \ge 1 - O(e^{-CND^{-2\kappa_0 - 4 \kappa_2}})$. On $
\mathcal S_N^{(i)}$ and for our choice of $K$ we can invoke Proposition 3.6 (and Remark 3.11) in \cite{NW20} (with $M=Id$ and appropriate choices for $c_{min}, c_{max}$ from Condition \ref{goldfisch}) to verify the required gradient Lipschitz property in Part B) as well as
\begin{equation}\label{hesse}
\sup_{\theta, \vartheta \in \mathbb R^D, \|\vartheta\|_{\mathbb R^D} = 1} \vartheta^T [\nabla^2 \log \tilde \pi_N(\theta|(Y_i, X_i)_{i=1}^N)] \vartheta \le \sup_{\theta, \vartheta \in \mathbb R^D, \|\vartheta\|_{\mathbb R^D} = 1} \vartheta^T [\nabla^2 \tilde \ell_N(\theta)] \vartheta \le - \frac{c_0}{2} ND^{-\kappa_0}.
\end{equation}

The rest of the proof is concerned with part C), without loss of generality for $N$ large enough. This will be done on the event $$\mathcal S_N \equiv \mathcal S_N^{(i)} \cap \mathcal S_N^{(ii)} \subset (V \times \mathcal X)^N,$$ where, for $\bar C$ from Theorem \ref{recycle},
$$ \mathcal S_N^{(ii)} = \Big\{\text{any } \hat \theta_\MAP \text{ satisfies }~\|\hat \theta_\MAP - \theta_{\true,D}\|_{E_D} \le \bar C\delta_N^{\gamma} \Big\}.$$ Note that $\mathcal S_N$ has the desired $P_{\theta_\true}^N$-probability in view of the preceding bound for $P_{\theta_\true}^N(\mathcal S_N^{(i)})$ and (\ref{mapinvrate}). By (\ref{hesse}), Condition \ref{bock} and Part A), on $\mathcal S_N$ any maximiser $\hat \theta_\MAP$ satisfies
\begin{equation} \label{score}
0=\nabla \log \pi(\hat \theta_\MAP|(Y_i, X_i)_{i=1}^N) = \nabla \log \tilde \pi(\hat \theta_\MAP|(Y_i, X_i)_{i=1}^N),
\end{equation}
from which we conclude that $\hat \theta_\MAP$ necessarily equals the \textit{unique} global maximiser of the strongly concave function $\log \tilde \pi(\cdot|(Y_i, X_i)_{i=1}^N)$ (and then also of $\pi(\cdot|(Y_i, X_i)_{i=1}^N)$) over $E_D$. In the rest of the proof we fix this unique global maximiser $\hat \theta_\MAP$.

\smallskip

\subsubsection{Decomposition of the Wasserstein distance.}  Define Euclidean balls $$\hat {\mathcal B}(r) =\{\theta \in E_D: \|\theta - \hat \theta_\MAP\|_{E_D} \le r\}, \text{ of radius } r>0.$$ Then by Theorem 6.15 in \cite{V09} with $x_0=\hat \theta_\MAP$
\begin{align*}
& W_2^2(\tilde \Pi(\cdot|(Y_i, X_i)_{i=1}^N), \Pi(\cdot|(Y_i, X_i)_{i=1}^N)) \\
&\le 2 \int_{E_D} \|\theta-\hat \theta_\MAP\|_{E_D}^2 d|\tilde \Pi(\cdot|(Y_i, X_i)_{i=1}^N)-\Pi(\cdot|(Y_i, X_i)_{i=1}^N)|(\theta).
\end{align*}
This can be further bounded, for $m>0$ to be chosen and $\tilde \delta_N$ from Condition \ref{bock}, by
\begin{align*}
& \le  2m^2\tilde \delta_N^2~ \int_{\hat {\mathcal B}(m\tilde \delta_N)} d|\Pi(\cdot|(Y_i, X_i)_{i=1}^N)-\tilde \Pi(\cdot|(Y_i, X_i)_{i=1}^N)|(\theta)\\
&~~~~~+ 2\int_{\|\theta-\hat \theta_\MAP\|_{E_D} >  m\tilde \delta_N}\|\theta-\hat \theta_\MAP\|_{E_D}^2 d\tilde \Pi(\theta|(Y_i, X_i)_{i=1}^N) \\
&~~~~~+2\int_{\|\theta-\hat \theta_\MAP\|_{E_D} >  m\tilde \delta_N}\|\theta-\hat \theta_\MAP\|_{E_D}^2 d \Pi(\theta|(Y_i, X_i)_{i=1}^N)  \equiv  I +II + III,
\end{align*}
and we now bound $I,II, III$ in separate steps.

\smallskip

\textbf{Term II.} We write the surrogate posterior density (\ref{surro}) as $$\tilde \pi(\theta|(Y_i, X_i)_{i=1}^N) = \frac{e^{\tilde \ell_N(\theta)- \tilde \ell_N(\hat \theta_\MAP)}\pi(\theta)}{\int_{E_D}e^{\tilde \ell_N(\theta)- \tilde \ell_N(\hat \theta_\MAP)}\pi(\theta)d\theta},~~\theta \in E_D.$$ Since $\delta_N =o(\tilde \delta_N)$ we have for $\eta$ from Condition \ref{bock}, $u$ from Lemma \ref{getreal} and some $c=c(u, C_\mathscr G)>0$ small enough, the set inclusion $$B_N \equiv \{\|\theta - \theta_{\true,D}\|_{E_D} \le c \delta_N\}\subset \big\{\|\theta- \theta_{\true,D}\|_{E_D} \le 3\eta/8 \big\} \cap \big\{\|\mathscr G(\theta)-\mathscr G(\theta_\true)\|_{L^2_\lambda} \le \delta_N/u \big\},$$ using also (\ref{bias}) and that $\mathscr G: E_D \to L^2_\lambda$ is Lipschitz by Condition \ref{ganzwien}.  Then since $\ell_N(\theta)=\tilde \ell_N(\theta)$ on the last set (cf.~(\ref{38})) we can apply Lemma \ref{expsmall} with $\nu = \Pi(\cdot)/\Pi(B_N)$ as well as the small ball estimate (\ref{gensmb}) with $d''=c$ to deduce that for $\bar c = K+d'$
\begin{align*}
\int_{E_D} e^{\tilde \ell_N(\theta)-\tilde \ell_N(\hat \theta_\MAP)}d\Pi(\theta) &\ge \int_{B_N} e^{\tilde \ell_N(\theta)-\tilde \ell_N(\hat \theta_\MAP)} d\Pi(\theta) \\
& = \int_{B_N} e^{\ell_N(\theta)- \ell_N(\hat \theta_\MAP)} d\nu(\theta) \times \Pi(B_N)  \ge e^{-\bar c N \delta_N^2}
\end{align*}
with $P_{\theta_\true}^N$-probability of order at least $1-O(e^{-\bar bN\delta_N^2})$.

Recalling the prior (\ref{gprior}) we define scaling constants
\begin{equation} \label{abnormal}
V_N = (2\pi)^{-D/2}\sqrt{\det(N \delta_N^2 \Lambda^{-1}_\alpha)} \times e^{\bar cN\delta_N^2} \text{ and note that } V_N \leq e^{c' N\delta_N^2 \log N}
\end{equation} for some $c'>0$ by (\ref{specdim}) and since $D \leq A_1 N \delta_N^2$. Then on the preceding events the term II can be bounded, using a second order Taylor expansion of $\log \tilde \pi(\cdot|(Y_i, X_i)_{i=1}^N)$ around its maximum $\hat \theta_\MAP$ combined with (\ref{hesse}), (\ref{score}), as
\begin{align*}
 &\int_{\|\theta-\hat \theta_\MAP\|_{E_D} >  m\tilde \delta_N}\|\theta-\hat \theta_\MAP\|_{E_D}^2 \tilde \pi(\theta|(Y_i, X_i)_{i=1}^N)d\theta \\
 &\le e^{\bar cN\delta_N^2} \int_{\|\theta-\hat \theta_\MAP\|_{E_D} >  m\tilde \delta_N}\|\theta-\hat \theta_\MAP\|_{E_D}^2 e^{\tilde \ell_N(\theta)- \tilde \ell_N(\hat \theta_\MAP)}\pi(\theta)d\theta \\
 &\le V_N \times \int_{\|\theta-\hat \theta_\MAP\|_{E_D} >  m\tilde \delta_N}\|\theta-\hat \theta_\MAP\|_{E_D}^2 e^{\tilde \ell_N(\theta)-\frac{N\delta_N^2}{2}\|\theta\|^2_{\tilde H^{\alpha}} - \tilde \ell_N(\hat \theta_\MAP) + \frac{N\delta_N^2}{2}\|\hat \theta_\MAP\|^2_{\tilde H^{\alpha}}}d\theta  \\
 &= V_N \times \int_{\|\theta-\hat \theta_\MAP\|_{E_D} >  m\tilde \delta_N}\|\theta-\hat \theta_\MAP\|_{E_D}^2 e^{\log \tilde \pi(\theta|(Y_i, X_i)_{i=1}^N)-\log \tilde \pi(\hat \theta_\MAP|(Y_i, X_i)_{i=1}^N)}d\theta \\
 & \le V_N \times \int_{\|\theta-\hat \theta_\MAP\|_{E_D} >  m\tilde \delta_N}\|\theta-\hat \theta_\MAP\|_{E_D}^2 e^{-\frac{c_0}{4}ND^{-\kappa_0} \|\theta- \hat \theta_\MAP\|^2_{E_D}}d\theta \\
 &\leq 2 V_N \times \big(\frac{8\pi}{c_0 ND^{-\kappa_0}}\big)^{D/2}  \mathrm{Pr} \big(\|Z\|_{E_D}>m \tilde \delta_N \big) \le V_N \mathrm{Pr} \big(\|Z\|_{E_D}>m \tilde \delta_N \big)
\end{align*}
where we have used $x^2 e^{-cx^2} \le  2e^{-cx^2/2}$ for all $x \in \R$, $c \ge 1$, that $c_0 ND^{-\kappa_0} \ge c_0 N \delta_N^2 \to \infty$ by Condition \ref{bock} and $\eta 
\le 1$, and where $$Z \sim N_{E_D} \Big(0, \frac{4}{c_0 ND^{-\kappa_0}} I_{D \times D}\Big).$$ Now by (\ref{laboum}) and $D \le A_1 N \delta_N^2$ we have $$E\|Z\|_{E_D} \le \sqrt{E\|Z\|_{E_D}^2} \le \sqrt{4D/(c_0 N D^{-\kappa_0})} \leq 2 \sqrt {A_1} c_0^{-1/2} D^{\kappa_0/2} \delta_N \le (m/2) \tilde \delta_N$$ for $m$ large enough, so that for some $c=c(c_0)$,
$$\mathrm{Pr} \big(\|Z\|_{E_D}>m \tilde \delta_N \big)\le \mathrm{Pr} \big(\|Z\|_{E_D}-E\|Z\|_{E_D}>(m/2) \tilde \delta_N \big) \le e^{-m^2 c ND^{-\kappa_0} \tilde \delta_N^2}$$ by a concentration inequality for Lipschitz-functionals of $D$-dimensional Gaussian random vectors (e.g., Theorem 2.5.7 in \cite{GN16} applied to $(c_0 N D^{-\kappa_0}/4)^{1/2}Z \sim \mathcal{N}(0,I_{D \times D})$ and $F=\|\cdot\|_{\mathbb R^D}$). Using this bound, (\ref{abnormal}) and that (\ref{laboum}) implies $ND^{-\kappa_0} \tilde \delta_N^2 \ge (\log N)N \delta_N^2$ we conclude that for all $m$ large enough the term $II$ is bounded by
$$V_N  \times e^{-m^2 c ND^{-\kappa_0}  \tilde \delta_N^2} \leq e^{-m^2 ND^{-\kappa_0} \tilde \delta_N^2/32} \le \frac{1}{8}e^{-N\delta_N^2}$$ with sufficiently large $P_{\theta_\true}^N$-probability.  

\smallskip

\textbf{Term III:} We first note that  Theorem \ref{recycle} and (\ref{laboum}) imply that for every $a>0$ we can find $m$ large enough such that
\begin{align*}
\Pi(\|\theta-\hat \theta_\MAP\|_{E_D} >  m\tilde \delta_N |(Y_i, X_i)_{i=1}^N) &\le \Pi(\|\theta-\theta_{\true,D}\|_{E_D} >  m\delta_N^\gamma - \|\hat \theta_\MAP - \theta_{\true,D}\|_{E_D} |(Y_i, X_i)_{i=1}^N) \\
&\le \Pi(\|\theta- \theta_{\true,D}\|_{E_D} >  m\delta_N^\gamma/2 |(Y_i, X_i)_{i=1}^N) \le e^{-aN \delta_N^2} 
\end{align*}
on events $\mathcal S_N' \subset \mathcal S_N$ of sufficiently high $P_{\theta_\true}^N$-probability. Moreover, again by Theorem \ref{recycle}, we can further intersect with the event $\mathcal C_{N, K}$ from (\ref{CN}) for some $K>0$ and with $B_N=\{\theta \in E_D:\|\mathscr G(\theta)-\mathscr G(\theta_\true)\|_{L^2_\lambda}\le \delta_N/u\}$. Now using the Cauchy-Schwarz inequality, the small ball estimate for $\Pi$ in (\ref{smball}), the identity $E^N_{\theta_\true}e^{\ell_N(\theta) - \ell_N(\theta_\true)}=1$ and Markov's inequality, we have 
\begin{align*}
&P^N_{\theta_\true}\Big(\mathcal C_{N, K} \cap \mathcal S'_N, \int_{\|\theta-\hat \theta_\MAP\|_{E_D} >  m\tilde \delta_N}\|\theta-\hat \theta_\MAP\|_{E_D}^2 d \Pi(\cdot|(Y_i, X_i)_{i=1}^N)>e^{-N\delta_N^2}/8\Big) \le  \\
&P^N_{\theta_\true}\Big(\mathcal C_{N, K} \cap \mathcal S'_N, \Pi(\|\theta-\hat \theta_\MAP\|_{E_D} >  m\tilde \delta_N |(Y_i, X_i)_{i=1}^N) E^\Pi[\|\theta-\hat \theta_\MAP\|_{E_D}^4 |(Y_i, X_i)_{i=1}^N]>\frac{e^{-2N\delta_N^2}}{64}\Big) \\
& \le P_{\theta_\true}^N\Big(\mathcal S_N', e^{(K+\bar d+2-a)N\delta_N^2} \int_{E_D} \|\theta-\hat \theta_\MAP\|_{E_D}^4 e^{\ell_N(\theta) - \ell_N(\theta_\true)} d\Pi(\theta) >1/64\Big) \\
&\lesssim e^{(K+\bar d+2-a)N \delta_N^2} \int_{E_D} (1+\|\theta\|_{E_D}^4 )d\Pi(\theta) \le e^{-aN \delta_N^2/2}
\end{align*}
whenever $m$ and then $a$ are large enough, since $\Pi$ has uniformly bounded fourth moments and since $\|\hat \theta_\MAP\|_{E_D}$ is uniformly bounded by a constant depending only on $\|\theta_\true\|_{L^2_\zeta}$ on the events $\mathcal S_N$.

\smallskip

\textbf{Term I:} On the events $\mathcal S_N$ we have from Condition \ref{bock} that for fixed $m>0$ and all $N$ large enough $$\hat {\mathcal B}(m\tilde \delta_N)\subseteq \{\theta:\|\theta-\theta_{\true,D}\|_{E_D}\le 3 \eta/8 \}.$$ On the latter set, by (\ref{38}), the probability measures $\tilde \Pi(\cdot|(Y_i, X_i)_{i=1}^N)$ and $\Pi(\cdot|(Y_i, X_i)_{i=1}^N)$ coincide up to a normalising factor, and thus we can represent their Lebesgue densities as $$\tilde \pi(\theta|(Y_i, X_i)_{i=1}^N)=p_N \pi(\theta|(Y_i, X_i)_{i=1}^N),~~\theta \in \hat {\mathcal B}(m\tilde \delta_N),$$ for some $0<p_N<\infty$. Moreover, by the preceding estimates for terms II and III (which hold just as well without the integrating factors $\|\theta-\hat \theta_\MAP\|_{\mathbb R^D}^2$), we have both
$$ p_N \Pi(\hat {\mathcal B}(m\tilde \delta_N)|(Y_i, X_i)_{i=1}^N) = \tilde \Pi(\hat {\mathcal B}(m\tilde \delta_N)|(Y_i, X_i)_{i=1}^N)  \ge 1 - e^{-N\delta_N^2}/8~\Rightarrow~ 1 - e^{-N\delta_N^2}/8 \le p_N,$$
$$  p^{-1}_N \tilde \Pi(\hat {\mathcal B}(m\tilde \delta_N)|(Y_i, X_i)_{i=1}^N) = \Pi(\hat {\mathcal B}(m\tilde \delta_N)|(Y_i, X_i)_{i=1}^N) \ge 1 - e^{-N\delta_N^2}/8~\Rightarrow~ 1 - e^{-N\delta_N^2}/8 \le \frac{1}{p_N}$$ on events of sufficiently high $P_{\theta_\true}^N$-probability. On these events necessarily $p_N \in [1- \frac{e^{-N\delta_N^2}}{8},[1-\frac{e^{- N\delta_N^2}}{8}]^{-1}]$
and so for $N$ large enough
\begin{align*}
&\int_{\hat {\mathcal B}(m\tilde \delta_N)} d|\Pi(\cdot|(Y_i, X_i)_{i=1}^N)-\tilde \Pi(\cdot|(Y_i, X_i)_{i=1}^N)|(\theta) \\
& =|1-p_N|\int_{\hat {\mathcal B}(m\tilde \delta_N)}  \pi(\theta|(Y_i, X_i)_{i=1}^N) d\theta \le |1-p_N| \le e^{-N \delta_N^2}/4,
\end{align*} which is obvious for $p_N \le 1$ and follows from the mean value theorem applied to $f(x)=(1-x)^{-1}$ near $x=0$ also for $p_N>1$. Collecting all the bounds we have shown that $I + II + III \le  e^{-N\delta_N^2}$ with the desired $P_{\theta_\true}^N$-probability, completing also the proof of part C) of Theorem \ref{onetrickpony}.

\subsection{Auxiliary results}
	
\subsubsection{The curvature lemma from \cite{NW20}}

\begin{lem}\label{youtube}
	Let $\ell_N: \R^D \to \R$ be given by (\ref{llemp}). Suppose Assumptions \ref{u4}, \ref{goldfisch} are satisfied for $\mathcal B$ from (\ref{eq:B:ass}) and some $\eta \le 1$. There exists a constant $C=C(c_0, c_1, c_2)>0$ such that if
	\begin{equation}\label{simpler}
	\mathcal R_N:=CN D^{-2\kappa_0- 4\kappa_2},
	\end{equation}
	then for any $D,N\ge 1$ satisfying $D\le \mathcal R_N$, we have for constants $c,c'$ depending only on $d_V$	
	\begin{equation}\label{eq:emp:lb}
	\begin{split}
	P_{\theta_\true}^N\Big(\inf_{\theta\in \mathcal B}\lambda_\min\big[-\nabla^2\ell_N(\theta)\big]< \frac {c_0}{2}N D^{-\kappa_0}\Big) &\le ce^{-\mathcal R_N}, ~\text{ as well as}
	\end{split}
	\end{equation}
	\begin{equation}\label{eq:emp:ub}
	\begin{split}
	&P_{\theta_\true}^N\Big( \sup_{\theta\in \mathcal B}\Big[|\ell_N(\theta)|+\|\nabla \ell_N(\theta)\|_{\R^D}+\|\nabla^2\ell_N(\theta)\|_\mathrm{op}\Big]> N(5c_1D^{\kappa_1}+1)\Big) \le c' (e^{-\mathcal R_N} + e^{-N/8}).
	\end{split}
	\end{equation}
	\end{lem}
	\begin{proof}
	When bounding deviation of probabilities involving $$\ell_N(\theta) =-\frac{1}{2} \sum_{k=1}^{d_V}\sum_{i=1}^N (Y_{ik} - \mathscr G(\theta)_k(X_i))^2$$ and its gradient and Hessian after centring them at their expectations, we can follow the proof of Lemma 3.4 in \cite{NW20} for each fixed vector coordinate $k=1, \dots, d_V$ if we can verify Assumptions 3.2 and 3.3 in \cite{NW20} with $\mathcal G(\theta)$ there equal to our $\mathscr G(\theta)_k, k$ fixed (see also Remark 3.11 in \cite{NW20}). This is provided by Conditions \ref{u4}, \ref{goldfisch} with constants in \cite{NW20} chosen as $k_i=m_i = c_2 D^{\kappa_2}$ for $i=0,1, 2$; $c_\min = c_0D^{-\kappa_0}$ and $c_\max =c_1 D^{\kappa_1}$. These choices imply that we can take $C_\mathcal G,  C'_\mathcal G, C''_\mathcal G, C'''_\mathcal G$ can all be taken to equal a constant multiple of $5(c_2D^{\kappa_2})^2$ in that lemma, and the dominant term in the definition of $\mathcal R_N$ there then becomes $c_\min^2/C'_\mathcal G \simeq D^{-2\kappa_0-4\kappa_2}$, using also that $D^{\kappa_1} \gtrsim D^{-\kappa_0}$ and $\eta \le 1$. The bound then follows for fixed $k$ from Lemma 3.4 in \cite{NW20} and carries over to the sum over $k=1, \dots, d_V$ by a simple union bound.
\end{proof}
		
\subsubsection{A general purpose posterior contraction theorem in Hellinger distance}

We follow here general ideas from Bayesian non-parametric statistics \cite{GV17} as developed in \cite{MNP21}. Consider the setting from \textsection \ref{bayesset}, specifically data $(Y_i, X_i)_{i=1}^N$ from (\ref{model}) with model probability densities $\{p_\theta: \theta \in \Theta\}$ from (\ref{densities}) on $V \times \mathcal X$. We can define a \textit{Hellinger distance} between such densities (where $dy$ is Lebesgue measure on $V$) as 
$$h^2(p_\theta, p_{\theta'}) = \int_{V \times \mathcal X} \big[\sqrt{p_{\theta}(y,x)} - \sqrt{p_{\theta'}(y,x)} \big]^2 dy d\lambda(x).$$
By uniform boundedness of $\mathscr G$ one can show the following.

\begin{prop}\label{lucienoldhand}
Suppose $\mathscr G$ satisfies Condition \ref{ganzwien}a). Then there exist constants $c_0, c_1$ depending on $U_\mathscr G$ and $d_V$ such that
$$\frac{1}{c_0} \|\mathscr G(\theta)-\mathscr G(\theta')\|_{L^2_\lambda} \le h(p_\theta, p_{\theta'}) \le c_1\|\mathscr G(\theta)-\mathscr G(\theta')\|_{L^2_\lambda}.$$
\end{prop}
\begin{proof}
The result follows as in Lemma 5.14 in \cite{MNP21} or also Lemma 22 in \cite{GN20}.
\end{proof}

For any set $\mathcal A$ of probability densities on $V \times \mathcal X$, let $N(\eta;\mathcal A,h), \ \eta>0,$ be the minimal number of Hellinger balls of $h$-radius $\eta$ needed to cover $\mathcal A$.

\begin{thm}\label{twotrickpony}	Let $\Pi_N$ be a sequence of prior Borel probability measures on some Borel subset $\Theta \subset L^2_\zeta(\mathcal Z, W)$, and let $\Pi_N(\cdot |(Y_i, X_i)_{i=1}^N)$ be the resulting posterior distribution arising from observations in model \eqref{model} with forward map $\mathscr G: \Theta \to L^2_\lambda(\mathcal X, V)$ satisfying Condition \ref{ganzwien}a) and such that the maps $((y,x),\theta) \mapsto p_\theta(y,x)$ are jointly Borel-measurable from $(V \times \mathcal X) \times \Theta \to \R$. Assume that for some fixed $\theta_\true\in\Theta$, and a sequence $\delta_N\to0$ s.t.~$ \sqrt{N}\delta_N \to\infty$ as $N\to\infty$, the sets 
\begin{equation}
\label{Been}
	\mathcal B_N
	:=
		\Big\{\theta \in \Theta: \ E^1_{\theta_\true}\Big[\log\frac{p_{\theta_\true}(Y_1,X_1)}
		{p_\theta(Y_1,X_1)}\Big]\le \delta_N^2,
		\ E^1_{\theta_\true}\Big[\log\frac{p_{\theta_\true}(Y_1,X_1)}
		{p_\theta(Y_1,X_1)}\Big]^2\le \delta_N^2\Big\},
\end{equation}
satisfy for all $N$ large enough 
\begin{equation}
\label{smballgen}
	\Pi_N(\mathcal B_N )\ge a e^{-A N\delta_N^2},
	\quad \textrm{some $a,A>0$}.
\end{equation}
Further assume that there exists a sequence of Borel sets $\Theta_N\subset \Theta$  for which
\begin{equation}
\label{excess}
	\Pi_N(\Theta_N^c)\lesssim e^{-B N\delta_N^2},
	\quad \textrm{some $B>A+2$,}
\end{equation}
as well as
\begin{equation}
\label{mettopq}
	\log N(\{p_\theta: \theta \in \Theta_N\}, h, \delta_N)\le C N\delta_N^2,
	\quad
	\textrm{some $C>0$.}
\end{equation}
Then, for sufficiently large $L=L(B,C)$, all $0<b<B-A-2$ and some $c>0$ we have that as $N\to\infty$,
\begin{equation}
\label{hellcont}
	P_{\theta_\true}^N\big(\Pi_N(p_\theta, \theta \in \Theta_N, h(p_\theta,p_{\theta_\true})\le L \delta_N |(Y_i, X_i)_{i=1}^N)\ge 1-e^{- b N\delta_N^2 }\big) = O(1- e^{-cN\delta_N^2}).
\end{equation}
\end{thm}
\begin{proof}
This is proved as in Theorem 13 in \cite{GN20} or Theorem 5.13 in \cite{MNP21}, except that we need to quantify the convergence rate to zero on the r.h.s.~of (\ref{hellcont}). But this follows from tracking the bounds in the preceding proof and i) the exponential bound proved in Lemma \ref{expsmall} as well as ii) the exponential bounds for the relevant testing errors via Theorem 7.1.4 in \cite{GN16}.
\end{proof}

The next basic lemma is helpful to verify (\ref{smballgen}).

\begin{lem}\label{getreal}
Let $\mathcal B_N$ be as in (\ref{Been}) and suppose $\mathscr G$ verifies Condition \ref{ganzwien}a). Then for $u=2(U_\mathscr G^2 +1)$ we have $$\{\theta \in \Theta: \|\mathscr G(\theta)-\mathscr G(\theta_\true)\|^2_{L^2_\lambda} \le \delta^2_N/u\} \subset \mathcal B_N.$$
\end{lem}
\begin{proof}
See Lemma 23 in \cite{GN20} (with obvious modifications pertaining to the $V$-valued case).
\end{proof}

\subsubsection{Metric entropy estimates}

The following standard result controls metric entropies of balls in the spaces $\tilde H^\alpha$, relevant in our proofs. We include a proof for completeness.

\begin{prop}\label{mettop}
Let $\mathcal H(\alpha, B)$ be a norm-ball of radius $B$ in $\tilde H^\alpha(\mathcal Z, W), \alpha > 0$. Suppose that (\ref{specdim}) holds for some $d \in \mathbb N$. Then the $\epsilon$-log-covering numbers of $\mathcal H(\alpha, B)$ for the $L^2_\zeta(\mathcal Z, W)$-distance satisfy
\begin{equation}\label{mettopa}
\log N(\mathcal H(\alpha,B), \|\cdot\|_{L_\zeta^2}, \epsilon) \le \Big(\frac{A B}{\epsilon} \Big)^{d/\alpha},~~0 < \epsilon < AB,
\end{equation}
where $A=A(d,\alpha, b_1, d_W)<\infty$ is a fixed constant. Moreover, assuming also (\ref{unifef}), we have for all $D \in \mathbb N, \bar \tau = d(\tau+(1/2))/\alpha$ and some $A'=A'(A)<\infty$
\begin{equation}\label{mettopb}
\log N(\mathcal H(\alpha,B) \cap E_D, \|\cdot\|_{\infty}, \epsilon) \le D^{\bar \tau}\Big(\frac{A' B}{\epsilon} \Big)^{d/\alpha},~~0 < \epsilon < A'B.
\end{equation}
\end{prop}
\begin{proof}
We set $d_W=1$, the general case is the same up to adjusting constants. Also a standard scaling argument for norms allows to restrict to $B=1$. Let us write $f_n=\langle f, e_n \rangle_{L^2_\zeta}$ for the `Fourier' coefficients in the basis $e_n$ throughout the proof. We first prove (\ref{mettopa}):  For any $f$ contained in $\mathcal H(\alpha,1) =\big\{f: \sum_{n} \lambda_{n}^{\alpha} f_n^2 \le 1 \big\},$ using (\ref{specdim}) and Parseval's identity,  we can estimate the error of $L^2_\zeta$-approximation $P_{2^{\ell_0}}f$ of $f$ from all frequencies up to the largest integer $n < 2^{\ell_0}$, as
\begin{align}
\|f-P_{2^{\ell_0}}f\|_{L^2_\zeta}^2 &= \sum_{n \ge 2^{\ell_0}} f_n^2 \lambda_n^{\alpha} \lambda_n^{-\alpha}  \le b_1^{-\alpha} 2^{-2\alpha \ell_0/d} \le (\epsilon/4)^2
\end{align}
where we have chosen $\ell_0$ as 
\begin{equation}\label{ell0}
\ell_0 =  \frac{d}{\alpha}\log_2\Big( \frac{4 \cdot 2^{(\alpha/d)+1} b_1^{-\alpha/2}}{\epsilon}\Big)
\end{equation}
The remaining indices $n:0 \le n < 2^{\ell_0}$ are now decomposed into dyadic brackets $$N_l = \{n: 2^{l-1}\le n \le 2^l-1\}, ~l=1, \dots, ;~ |N_l|= 2^{l-1},$$ where by convention we set $N_0=\{0\}$. Thus using Parseval's identity, (\ref{specdim}) and the preceding estimate we can bound, for any $f,g \in \mathcal H(\alpha,1)$, 
\begin{align*}
\|f-g\|_{L^2_\zeta} &\le \sqrt{\sum_{0 \le l \le \ell_0+1} \sum_{n \in N_l} \lambda_n^{-\alpha} \lambda_n^{\alpha}(f_{n}-g_{n})^2} + 2\epsilon/4 \\
&\leq 2^{\alpha/d} b_1^{-\alpha/2} \sum_{0 \le l \le \ell_0+1} 2^{-l\alpha/d} \sqrt{\sum_{n \in N_l}  (\lambda_n^{\alpha/2} f_{n}- \lambda_n^{\alpha/2} g_{n})^2} + \epsilon/2.
\end{align*}
Now by definition of $\mathcal H(\alpha, 1)$, for every $l \in \mathbb N \cup \{0\}$ the vectors $\{\lambda_n^{\alpha/2} f_{n}: n \in N_l\},~\{\lambda_n^{\alpha/2} g_{n}: n \in N_l\}$ lie in the unit ball of a Euclidean space of dimension $|N_l| \leq  2^l$. The Euclidean $\epsilon'$-covering numbers of such a unit ball $B_l$ for the standard Euclidean norm are $$N(l)\equiv N(B_l, \|\cdot\|_{\mathbb R^{|N_l|}}, \epsilon') \le (3/\epsilon')^{2^l},~0<\epsilon'<1,$$ (see, e.g., Proposition 4.3.34 in \cite{GN16}). By choosing $\epsilon'_l = 2^{l ((\alpha/d)+1)} 2^{-\ell_0 ((\alpha/d)+1)}$-coverings of $B_l$ for each $l$ centred at points $(f_{n,l,i}: n \in N_l, l \le \ell_0+1)$, and setting $f_{n,l,i} =0$ for $n \ge 2^{\ell_0+1}$, we obtain a covering $$\Big ( \bar f^{(i)}= \lambda_n^{-\alpha/2}f_{n,l,i}: i=1, \dots, \prod _{0 \le l \le \ell_0+1}N(l)\Big)$$ of $\mathcal H(\alpha,1)$ of radius bounded by
\begin{align*}
\|f-\bar f^{(i)}\|_{L^2_\zeta} &\le \frac{\epsilon}{2} + 2^{\alpha/d} b_1^{-\alpha/2} \sum_{0 \le l \le \ell_0+1} 2^{-l\alpha/d}2^{l ((\alpha/d)+1)} 2^{-\ell_0 ((\alpha/d)+1)} \\
& \le \frac{\epsilon}{2} + 2^{(\alpha/d)+2} b_1^{-\alpha/2} 2^{-\alpha \ell_0/d}  \le \epsilon,
\end{align*}
by choice of $\ell_0$. We thus obtain, for some $c'=c'(\alpha, d)>0$,
\begin{align*}
\log_2 N(\mathcal H(\alpha,1), \|\cdot\|_{L^2_\zeta}, \epsilon) &\le \sum_{0 \le l \le \ell_0+1} \log_2 N(l) \leq  \sum_{0 \le l \le \ell_0+1} 2^{l}\Big[ \log_2 3 + (\ell_0-l)\big(\frac{\alpha}{d}+1\big)\Big] \\
&\le c' 2^{\ell_0} \le (A/\epsilon)^{d/\alpha}
\end{align*}
so that (\ref{mettopa}) follows. To prove (\ref{mettopb}) we estimate, for $f, g \in E_D$, and using Parseval's identity as well as (\ref{unifef}),
$$\|f-g\|_\infty = \sup_z \left|\sum_{n \le D} (f_{n}-g_{n}) e_{n}(z) \right| \le \|f-g\|_{L_\zeta^2} \|e_{n}\|_\infty  \sqrt{\sum_{n \le D}1} \lesssim D^{\tau+1/2} \|f-g\|_{L_\zeta^2},$$ so that the result follows from (\ref{mettopa}) after adjusting the constant $A$.
\end{proof}

\subsubsection{A concentration inequality for empirical processes}

The following is Lemma 3.12 from \cite{NW20}. The result is formulated for countable $\Theta$; when applied in the context of this article, the relevant suprema can all be shown to be attained as countable suprema.

\begin{lem}\label{mixchain}
	Let $\Theta$ be countable. Suppose a class of real-valued measurable functions $$\mathcal H=\{h_\theta: \mathcal X \to \mathbb R, \theta \in \Theta\}$$ defined on a probability space $(\mathcal X, \mathscr X, P^X)$ is uniformly bounded by $U\ge \sup_\theta \|h_\theta\|_\infty$ and has variance envelope $\sigma^2 \ge \sup_\theta E^Xh_\theta^2(X)$ where $X \sim P^X$. Define metric entropy integrals $$J_2(\mathcal H) = \int_0^{4\sigma} \sqrt{\log N(\mathcal H, d_2,\rho)}d\rho,~~d_2(\theta,\theta'):=\sqrt{E^X[h_\theta(X)-h_{\theta'}(X)]^2},$$
	$$J_\infty(\mathcal H) = \int_0^{4U} \log N(\mathcal H, d_\infty,\rho) d\rho,~~d_\infty(\theta,\theta'):=\|h_\theta-h_{\theta'}\|_{\infty}.$$
	For $N \in \mathbb N$ and $X_1, \dots, X_N$ drawn i.i.d.~from $P^X$ and $\varepsilon_i \sim^{\iid} \mathcal{N}(0,1)$ independent of all the $X_i$'s, consider empirical processes arising either as 
	\begin{equation}\label{multipl}
	Z_{N}(\theta)=\frac{1}{\sqrt N}\sum_{i=1}^N h_\theta(X_i)\varepsilon_i,~~\theta \in \Theta,~\text{ or as }
	\end{equation}
	 \begin{equation} \label{centp}
	 Z_N(\theta)=\frac{1}{\sqrt N}\sum_{i=1}^N (h_\theta(X_i)-Eh_\theta(X)), ~~\theta \in \Theta.
	 \end{equation}
	  We then have for some universal constant $L>0$ and all $x\ge 1$,
	$$\mathrm{Pr} \left(\sup_{\theta \in \Theta}|Z_{N}(\theta)| \ge L \Big[J_2(\mathcal H) + \sigma \sqrt{x} + (J_\infty(\mathcal H) +Ux)/\sqrt N   \Big] \right) \le 2e^{-x}.$$ We also record here the `pointwise' bound for singleton $\mathcal H = \{h_\theta\}$, and $Z_N\equiv Z_N(\theta)$,
	\begin{equation} \label{bernie}
	\mathrm{Pr} \left(|Z_{N}| \ge L \sigma \sqrt{x} + LUx/\sqrt N   \right) \le 2e^{-x}.
	\end{equation}
	\end{lem}

\section{Proofs for Section \ref{naxray}} \label{finalsection}

Based on Theorem \ref{zernikestability}, \ref{zernikeforward} (and also the global stability estimate from \cite{MNP21}), and loosely following the corresponding work in 
 \cite[\textsection 4]{NW20} for the Schr{\"o}dinger equation, 
we verify in this section that the non-Abelian X-ray transform with identifications from (\ref{thetaxray}) satisfies Conditions \ref{linsp} up to  \ref{bock}  with
\begin{equation}
d=2,\quad \tau= 1/4,\quad \alpha> 5,\quad \gamma = 1-1/\alpha,
\end{equation}
as well as
\begin{equation} \label{stepsizex}
\eta = D^{-2}/\log N, \quad  \kappa_0 =1/2,\quad \kappa_1=0 \quad \text{ and } \quad \kappa_2=3/2,
\end{equation}
for all $N$ large enough, see also Remark \ref{ponytamer}.

\begin{proof}[Proof of Theorems \ref{xraypony} and \ref{xraypony2}] Granted Conditions \ref{linsp} up to  \ref{bock} are verified, the general theory from \textsection \ref{genth} and \cite[\textsection 3, Remark 3.11]{NW20} becomes available.  Specifically Theorem \ref{xraypony} follows directly from Theorem \ref{onetrickpony},  when using $\theta_\init=\theta_{\true,D}$ as base point.
Further,  Theorem \ref{xraypony2} follows from Theorem 3.7 in \cite{NW20} in conjunction with Lemma \ref{youtube} and again Theorem \ref{onetrickpony} -- now with base point $\theta_\init$ satisfying the hypothesis of Theorem \ref{xraypony2} -- and with stepsizes $\gamma$ chosen as in (\ref{stepsize}) for corresponding choices of $\bar m$ and $\Lambda$.
\end{proof}

\subsection{Real-valued setting and approximation spaces}  \label{realval} 

An orthonormal basis of $L^2(\DD,\R)$ can be obtained by taking real and imaginary parts of Zernike polynomials, noting that \eqref{zsymmetry} reveals redundant polynomials that can be dropped; precisely we put
\begin{equation}
W_{nk}=\begin{cases}
\sqrt 2 \Re Z_{nk} & n-2k> 0\\
Z_{nk} & n-2k=0\\
\sqrt 2 \Im Z_{n,n-k} & n-2k<0
\end{cases},\quad \hat W_{nk} = \sqrt{\frac{n+1}{\pi}}\cdot W_{nk} ,
\end{equation}
for indices $n\in \N_0$ and $k\in \Z$ with $0\le k \le n$.  
Ordering the double indices as $(0,0),(1,0),(1,1),(2,0),\dots,(n_\ell,k_\ell),\dots$ (i.e.\,rearranged in a pyramid scheme)  we can pass to a single index $\ell\in \N_0$ for the basis.  Then the pair
\begin{equation}\label{singularvalues}
e_\ell = \hat W_{n_\ell,k_\ell}\quad \text{ and } \quad \lambda_\ell = \left(\frac{1+n_\ell}{4\pi}\right)^2 \qquad (\ell\in\N_0)
\end{equation}
obeys
\begin{equation}\label{singleindexbound}
\sup_{0\le \ell \le D} \Vert e_\ell \Vert_{L^\infty(\DD)} \lesssim D^{1/4}\quad\text{ and } \quad   \ell \lesssim \lambda_\ell\lesssim \ell,
\end{equation}
such that the asymptotic conditions \eqref{unifef} and \eqref{specdim} in Conditions \ref{linsp} are satisfied with parameters $\tau=1/4$ and $d=2$ respectively. Indeed,  \eqref{singleindexbound} follows from \eqref{znorms}, \eqref{zsymmetry} and the asymptotic equality $n_\ell \sim \ell^{1/2}$. 

\smallskip

In order to define the finite dimensional approximation spaces $E_D$, we write $d_m =  \dim \so(m)$  $\equiv m(m-1)/2$ and $D_m = \lfloor D/d_m \rfloor$ and  fix a (Frobenius-)orthonormal basis $(A_i:1\le i \le d_m)$ of $\so(m)$. Then 
\begin{equation}\label{eddef2}
E_D=E_D(\DD,\so(m)) =\left \{\Phi = \sum_{\ell =0}^{D_m}\sum_{i=1}^{d_m}  \phi_{\ell,i}  e_\ell  A_i: \phi_{\ell,i} \in \R \right\}\subset C^\infty(\bar \DD,\so(m))
\end{equation}
is a subspace of dimension $D_md_m\approx D$ as in \eqref{eddef3}. 
If $D =  d_m \cdot (D'+1)(D'+2)/2$ for some $D'\in \N_0$, then $\mathrm{span} \{e_\ell: \ell \le D_m\}\subset C^\infty(\DD,\R)$ consists of polynomials of degree $\le D'$, and the earlier definition in \eqref{eddef} can be seen to coincide with the present one by comparing dimensions.

\smallskip

Note that the scale $\tilde H^s(\DD,\R)$ from \eqref{scale} coincides with the intersection $\tilde H^s(\DD,\C)\cap L^2(\DD,\R)$, where the complex Zernike scale is as in \eqref{zscaledef}. Similarly the analogous scale of $\so(m)$-valued maps is obtained as $\tilde H^s(\DD,\so(m))=\tilde H^s(\DD,\u(m))\cap L^2(\DD,\so(m))$. 

\smallskip

On $E_D$ the norms of all function spaces are comparable and we record here some inequalities that are frequently used below. Recall from \textsection 3.2 that $\Vert \cdot \Vert_{E_D}$ is just the $L^2(\DD)$-norm, restricted to $E_D$.

\begin{lem} \label{normcomparison} For $h\in E_D$ ($D\in \N$) and $\alpha \ge 0$ we have:
\begin{align}
\Vert h \Vert_{L^\infty(\DD)} &\lesssim D^{3/4} \cdot \Vert h \Vert_{E_D}\\
\Vert h \Vert_{\tilde H^\alpha(\DD)} &\lesssim D^{\alpha/2} \cdot \Vert h \Vert_{E_D} \\
\Vert h \Vert_{E_D} &\lesssim 
D^{\alpha/2}\cdot \Vert h \Vert_{\tilde H^{-\alpha}(\DD)} 
\end{align}
\end{lem}

\begin{proof} 
Let $h_{\ell,i}:=\langle h, e_\ell A_i\rangle_{L^2(\DD)}\in \R$ 
be the coefficients of $h$ and write $h_\ell = \sum_{i=1}^{d_m} h_{i\ell} A_i \in \so(m)$. Then for 
the first inequality we use Cauchy-Schwarz and \eqref{singleindexbound} to obtain
\begin{equation*}
\Vert h \Vert_{L^\infty} \le \sum_{0\le \ell \le D_m} \vert h_\ell\vert_F \cdot \Vert e_\ell \Vert_{L^\infty} \le \left( \sum_{\ell=0}^{D_m} \Vert e_\ell \Vert^2_{L^\infty(\DD)}\right)^{1/2} \cdot \Vert h \Vert_{E_D} \lesssim D^{3/4} \Vert h \Vert_{E_D},
\end{equation*}
noting that $D_m\lesssim D$.
For the second inequality use the upper bound on the eigenvalue $\lambda_\ell$ in \eqref{singleindexbound} to obtain
\begin{equation}
\Vert h \Vert_{\tilde H^\alpha(\DD)}^2 = \sum_{\ell=0}^{D_m} \lambda_\ell^\alpha \vert h_\ell \vert_{F}^2 \lesssim D^\alpha \Vert h \Vert_{E_D}^2,
\end{equation}
which implies the result. The last inequality follows similarly.
\end{proof}

\subsection{Global conditions for log-concave approximation}
We now show that the non-Abelian X-ray transform $C:\Theta\rightarrow L^2_\lambda(\partial_+S\bar \DD)$ (given in \eqref{yetagain}, with $\Theta$ as in \eqref{thetaxray}) satisfies the `global' Condition \ref{ganzwien}. 
\smallskip
\begin{enumerate}[label=\alph*)]
\item The {\it uniform boundedness property} is a direct consequence of $C_\Phi$ taking values in the compact group $SO(m)$; precisely we have
\begin{equation} \label{globalbound}
\sup_{\Phi \in \Theta} \Vert C_\Phi \Vert_{L^\infty(\partial_+S\bar \DD)} \le \sqrt m.
\end{equation}
\item The {\it global Lipschitz property} is a special case of Theorem 2.2 in \cite{MNP21}, again the involved constants only depend on $m$.
\end{enumerate}

\smallskip 

Combining the stability estimate in \cite{MNP21} with our forward estimate in the Zernike scale (Theorem \ref{zernikeforward}) we also have\smallskip {\rm
\begin{itemize}
\item[c)] The {\it inverse continuity modulus property}  from Condition \ref{ganzwien} holds on $E_D$ for any $D\in \N_0$ and for any pair $(\alpha,\gamma)$ with  $\alpha>5/2$ and $\gamma \in (0,1-1/\alpha]$,
\end{itemize} }

\smallskip

by virtue of the following lemma.

\begin{lem} \label{inversecm} Suppose  $\Phi,\Psi\in C^\infty(\bar \DD,\so(m))$. Then for all $s>5/2$,
\begin{equation*}
\Vert \Phi - \Psi \Vert_{L^2(\DD)} \le \omega_s(\Vert \Phi \Vert_{\tilde H^s(\DD)} \vee \Vert \Psi \Vert_{\tilde H^s(\DD)}) \cdot \Vert C_\Phi - C_\Psi \Vert_{L^2(\DD)}^{1-1/s}
\end{equation*}
for a non-decreasing function $\omega_s:[0,\infty)\rightarrow [0,\infty)$ (of exponential growth), which only depends on $m$ and $s$. 
\end{lem}

\begin{proof}
Starting with the stability estimate from \cite[Thm.\,5.3]{MNP21} we have
\begin{equation*}
\Vert \Phi - \Psi \Vert_{L^2(\DD)} \le C_0(\Phi,\Psi) \cdot \Vert I_\Xi(\Phi - \Psi ) \Vert_{H^1(\partial_+S\bar \DD)},
\end{equation*}
where $\Xi:\DD\rightarrow \so(m^2)\subset \End(\R^{m\times m})$ is defined pointwise by $\Xi (A) = \Phi A - A \Psi$ (as below \eqref{pseudolinearisation}) and the constant $C_0(\Phi,\Psi)$ grows at most exponentially in the $C^1$-norms of $\Phi$ and $\Psi$. By the interpolation inequality for Sobolev spaces (see e.g. \cite[Lemma 7.4]{Boh20}), we can bound the right hand side in the previous display by
\begin{equation*}
\lesssim C_0(\Phi,\Psi) \cdot \Vert I_\Xi(\Phi - \Psi)  \Vert_{L^2_\lambda(\partial_+S\bar \DD)}^{1-1/s} \cdot \Vert I_\Xi(\Phi - \Psi)  \Vert_{H^s(\partial_+S\bar \DD)}^{1/s},
\end{equation*}
where $1<s<\infty$ can be chosen freely.
Now the first norm equals $\Vert C_\Phi C_\Psi^{-1} - I \Vert_{L^2(\DD)}= \Vert C_\Phi - C_\Psi \Vert_{L^2(\DD)}$ by the pseudo-linearisation identity (see \eqref{pseudolinearisation}) and the second norm can be estimated by means of Theorem \ref{zernikeforward} such that we have
\begin{equation*}
\Vert \Phi - \Psi \Vert_{L^2(\DD)} \le C_1(\Phi,\Psi,s) \cdot \Vert C_\Phi - C_\Psi \Vert_{L^2(\DD)}^{1-1/s},
\end{equation*}
with constant obeying
\begin{equation*}
C_1(\Phi,\Psi,s) \lesssim e^{c_1 \left( \Vert \Phi \Vert_{C^1(\bar \DD) } 
\vee \Vert \Psi \Vert_{C^1(\bar \DD)} \right) 
}
\cdot \left(\Vert \Phi \Vert_{\tilde H^s(\DD)} \vee \Vert \Psi \Vert_{\tilde H^s(\DD)}\right)^{1/s}.
\end{equation*} 
Choosing $s>5/2$,  the Zernike space $\tilde H^s(\DD)$ embeds into $C^1(\DD)$ \cite[Lem.\,13,14]{Mon20} and thus its norm dominates the $C^1$-norm. This concludes the result.
\end{proof}

\subsection{Local conditions for log-concave approximation} Next, we consider the non-Abelian X-ray transform on finite-dimensional Euclidean  balls 
\begin{equation}
\B(\eta,D)=\{\Phi \in E_D(\DD,\so(m)): \Vert \Phi  - \Phi_{\truep,D}\Vert_{E_D} \le \eta \},  \quad (D\in \N_0, \eta>0)
\end{equation}
where $\Phi_{\truep,D}$ is the $L^2$-orthogonal projection of $\Phi_\truep \in L^2(\DD,\so(m))$ onto $E_D$.

\smallskip

The first result concerns local regularity on $\B(\eta,D)$ and follows from our computations in Lemma \ref{derivatives} and standard forward estimates.  As a matter of fact, the bounds below are uniform in $\eta$ and can thus be formulated as global results.  To this end, we use the notation $C^{2,1}(E_D,\R^{m\times m})$ for the space of all twice differentiable maps $g:E_D\rightarrow \R^{m\times m}$ with finite norm
\begin{equation*}
\begin{split}
\Vert g \Vert_{C^{2,1}(E_D,\R^{m\times m})} =& \sup_{i=1,\dots,m^2}
\Big \{\Vert g^i \Vert_{L^\infty(E_D,\R)} + \Vert \nabla g^i \Vert_{L^\infty(E_D,\R^D,\vert \cdot \vert_{\R^D})}   \\
&+ \Vert \nabla^2 g^i \Vert_{L^\infty(E_D,\R^{D\times D},\vert \cdot \vert_{\mathrm{op}})} + \sup_{\Phi, \Psi \in E_D, \Phi\neq \Psi} \frac{\vert \nabla^2 g^i(\Phi) -\nabla^2 g^i(\Psi) \vert_{\mathrm{op}}}{\vert \Phi -\Psi\vert_{\R^D}} \Big\},
\end{split}
\end{equation*}
where $\nabla$ and $\nabla^2$ denote gradient and Hessian on $E_D$ and $g=(g^1,\dots,g^{m^2})$ in an orthonormal basis of $(\R^{m\times m},\vert \cdot \vert_F)$ (cf.\,\eqref{c2lip}).


\begin{lem}[Local regularity] 
Let $D\in \N_0$.  For fixed $(x,v)\in \partial_+S\bar \DD$, the map  $\Phi \mapsto C_\Phi(x,v)\in SO(m)$ belongs to the space $C^{2,1}(E_D,\R^{m\times m})$ and has norm 
\begin{equation}\label{locreg}
\sup_{(x,v)\in \partial_+S\bar \DD} \Vert C_\bullet (x,v) \Vert_{C^{2,1}(E_D,\R^{m\times m})} \le c_2 D^{3/2}
\end{equation}
for some constant $c_2=c_2(m)>0$. In particular Condition \ref{u4} is satisfied with $\kappa_2 =3/2$.
\end{lem}

\begin{proof} From Lemma \ref{derivatives} we already know that the map $\Phi \mapsto C_\Phi(x,v)$ is twice differentiable on $E_D$; we will prove for $\Phi,\Psi,h\in E_D$ and $(x,v)\in \partial_+S\bar \DD$ that
\begin{align}
\vert C_\Phi(x,v) \vert_F &\lesssim 1 \label{locreg1}\\
\vert \dot C_\Phi[h](x,v) \vert_F & \lesssim \Vert h \Vert_{L^\infty(\DD)} \label{locreg2}\\
\vert \ddot C_\Phi[h](x,v) \vert_F & \lesssim \Vert h \Vert_{L^\infty(\DD)}^2 \label{locreg3}\\
\vert \ddot C_\Phi[h](x,v) - \ddot C_\Psi[h](x,v) \vert_F &\lesssim \Vert h \Vert_{L^\infty(\DD)}^2 \cdot \Vert \Phi - \Psi \Vert_{L^2(\DD)} \label{locreg4},
\end{align}
with implicit constants only depending on $m$. Let us first verify that this is sufficient for \eqref{locreg}: Fix an orthonormal basis $A^1,\dots,A^{m^2}$ of $(\R^{m\times m},\vert\cdot \vert_F)$ and write $g^i(\Phi)=\langle C_\Phi(x,v),A^i\rangle_F$, then
\begin{equation}
\vert g^i(\Phi) \vert \le \vert C_\Phi(x,v) \vert_F,\quad \vert \nabla g^i(\Phi) \vert_F \le \sup_{\Vert h \Vert_{E_D}\le 1} \vert \dot C_\Phi[h](x,v) \vert_F.
\end{equation}
Further,  noting that the operator norm of a symmetric matrix equals the absolute value of its largest eigenvalue, we have
\begin{equation}
\vert \nabla^2 g^i(\Phi) \vert_{\mathrm{op}} \le \sup_{\Vert h \Vert_{E_D}\le 1} \vert \ddot C_\Phi[h](x,v) \vert_F,
\end{equation}
with similar bounds for the difference $\nabla^2 g^i(\Phi) - \nabla^2 g^i(\Psi)$. Overall this implies \begin{equation}
 \Vert C_\bullet (x,v) \Vert_{C^{2,1}(E_D,\R^{m\times m})} \lesssim \sup_{\Vert h \Vert_{E_D}\le 1}\left( 1 + \Vert h \Vert_{L^\infty(\DD)} +   2 \Vert h \Vert_{L^\infty(\DD)}^2 \right) \lesssim D^{3/2},
\end{equation}where the $L^\infty$-norms of $h\in E_D$ have been bounded by means of Lemma \ref{normcomparison}.

\smallskip

On to the proof of \eqref{locreg1} - \eqref{locreg3}. The first inequality was already stated in \eqref{globalbound} and the second follows from noting that $ \vert \dot C_\Phi[h](x,v) \vert_F = \vert I_\Xi h (x,v)\vert_F $ (with $\Xi$ as in Lemma \ref{derivatives}) and the fact that $I_\Xi$ is $L^\infty$-$L^\infty$ bounded (with operator norm uniform in $\Xi$, see e.g. \cite[Lem.\,5.2]{MNP21}). To prove \eqref{locreg3} and \eqref{locreg4}, note that 
\begin{equation}
\ddot C_\Phi[h]=  I_\Xi(hV_h) \cdot C_\Phi =  I(R_\Xi^{-1} hV_h) \cdot C_\Phi ,
\end{equation}
where $\Xi$ and $V_h$ are as in Lemma \ref{derivatives} and $R_\Xi:S\bar \DD\rightarrow SO(m)$ satisfies  
$(X+\Xi)R_\Xi = 0$ on $S\DD$ and $R_\Xi = \id $ on $\partial_-S\bar \DD$.  Using $L^\infty$-$L^\infty$-boundedness of $I_\Xi$ we obtain
\begin{equation}
\Vert \ddot C_\Phi[h] \Vert_{L^\infty(\partial_+S\bar \DD)} \lesssim \Vert h V_h \Vert_{L^\infty(S\DD)} \le \Vert h \Vert_{L^\infty(\DD)} \cdot \Vert V_h \Vert_{L^\infty(\DD)} \lesssim \Vert h \Vert_{L^\infty(\DD)}^2
\end{equation}
where the last inequality follows from \eqref{horse3} -- this proves \eqref{locreg3}. For \eqref{locreg4} we estimate, using the notation $\Theta$ and $W_h$ for $\Psi$ in place of $\Xi$ and $V_h$ for $\Phi$,
\begin{equation}
\begin{split}
\Vert \ddot C_\Phi[h] - \ddot C_\Psi[h]  \Vert_{L^\infty(\partial_+S\bar \DD)} \le & ~ ~\Vert  I_\Xi( hV_h) \cdot \left ( C_\Phi - C_\Psi \right)  \Vert_{L^\infty(\partial_+S\bar \DD)}\\
& +  \Vert I_\Xi(hV_h- hW_h) \cdot C_\Psi \Vert_{L^\infty(\partial_+S\bar \DD)}\\
& + \Vert I\left((R_\Xi^{-1} - R_\Theta^{-1} )(hW_h) \right)\cdot C_\Psi \Vert_{L^\infty(\partial_+S\bar\DD)}\\
 = & (i) + (ii) + (iii)
\end{split}
\end{equation}
The terms $(i)$-$(iii)$ are bounded seperately; for the first one note that
\begin{equation}
	(i) \le \Vert I_\Xi(hV_h) \Vert_{L^\infty(\partial_+ S\bar \DD)} \cdot \Vert C_\Phi - C_\Psi \Vert_{L^\infty(\partial_+S\bar \DD)} \lesssim \Vert h \Vert_{L^\infty(\DD)}^2\cdot \Vert \Phi - \Psi \Vert_{L^\infty(\DD)},
\end{equation}
 where the first factor has been estimates as in the argument leading to \eqref{locreg3} and the second one by the $L^\infty$-$L^\infty$-Lipschitz estimate from \cite[Thm.\,2.2]{MNP21}. Similarly,
 \begin{equation}
 (ii) \le \Vert h \Vert_{L^\infty(\DD)} \cdot \Vert V_h - W_h \Vert_{L^\infty(S\DD)} \lesssim \Vert h \Vert_{L^\infty(\DD)}^2 \cdot \Vert \Phi - \Psi \Vert_{L^\infty(\DD)}
 \end{equation}
 by estimate \eqref{horse3} above. Finally we have
 \begin{equation}
 (iii) \le \Vert R_\Xi^{-1} - R_\Theta^{-1} \Vert_{L^\infty(S\DD)}\cdot \Vert h W_ h \Vert_{L^\infty(S \DD)} \le \Vert h \Vert_{L^\infty(\DD)}^2 \cdot \Vert \Phi - \Psi \Vert_{L^\infty(\DD)}
 \end{equation}
 by \eqref{horse1} and \eqref{horse3}. The proof is complete.
\end{proof}

Next we verify `local curvature' Condition \ref{goldfisch}, which is stated in terms of the generic log-likelihood function (cf.\,\eqref{llgen})
\begin{equation}
\ell(\Phi) = \ell(\Phi, (X,V,Y)) = -\frac12 \vert Y - C_\Phi(X,V) \vert_F^2\quad \text{where } (X,V,Y)\in \partial_+S\bar \DD\times \R^{m\times m}
\end{equation}
and its expectation $E_{\Phi_\truep}$ under the assumption that $(X,V)$ is drawn uniformly at random from $\partial_+S\bar \DD$ and $Y$ follows the regression model $Y = C_{\Phi_\truep}(X,V)+\varepsilon$, where $\varepsilon \in \R^{m\times m}$ is independent Gaussian noise (cf.\,\eqref{model0}).  Recalling the discussion preceding Condition \ref{goldfisch} (and the identification of $\Phi\in E_D$ with its coefficients $\theta\in \R^D$) we use the notation 
\begin{equation}
\nabla \ell(\Phi,(X,V,Y)) \in \R^D\quad \text{ and }\quad  \nabla^2 \ell(\Phi,(X,V,Y)) \in \R^{D\times D}
\end{equation}
for gradient and Hessian of the map $\Phi \mapsto \ell(\Phi,(X,V,Y))$ on $E_D$ (with dimension $D$ suppressed, by abuse of notation.)

\begin{lem}[Local curvature]  \label{xraylocalcurvature} Given a compact set $K\subset \DD$ and numbers $M>0,\alpha>5$ there are constants
$\epsilon,c_0,c_1>0$ (depending on $M,K,\alpha$)
with the following property: Suppose $\Phi_\truep:\DD\rightarrow \so(m)$ is of regularity $H^\alpha(\DD)$ and obeys
\begin{equation} \label{phi0conditions}
\supp \Phi_\truep \subset K,\quad \Vert \Phi_\truep \Vert_{H^\alpha(\DD)}\le M\quad \text{and} \quad  \Vert C_{\Phi_\truep} -C_{\Phi_{\truep,D}} \Vert_{L^2_\lambda(\partial_+S\bar \DD)} \le \epsilon D^{-2}
\end{equation}
for all $D\in \N$.  Then on balls $\B(\eta,D)=\{\Phi \in E_D: \Vert \Phi_{\truep,D} - \Phi \Vert_{E_D} < \eta \}$ of radius
$
0 < \eta \le \epsilon D^{-2}
$ 
the averaged log-likelihood satisfies
\begin{eqnarray}
\inf_{\Phi \in \B(\eta,D)}\lambda_\min\left(-E_{\Phi_\truep} \nabla^2 \ell(\Phi) \right)&\ge& c_0 D^{-1/2}\label{zlocalcurvature1}\\
\sup_{\Phi \in \B(\eta,D)} \left[\vert E_{\Phi_\truep} \ell(\Phi) \vert  + \vert E_{\Phi_\truep} \nabla \ell(\Phi) \vert_{\R^D} +  \vert E_{\Phi_\truep} \nabla^2 \ell(\Phi) \vert_{\mathrm{op}}\right]  &\le& c_1. \label{zupperbound}
\end{eqnarray} 
\end{lem}

\begin{rem}\label{toolate} The requirement that $\Phi_\truep$ has Sobolev-regularity $H^\alpha(\DD)$ can be replaced by $\Phi_\truep \in F$ for any normed function space $F$ that embeds compactly into $C^4(\bar \DD,\so(m))$,  another valid choice being H{\"o}lder-space $F = C^{4+\epsilon}(\bar \DD)$ ($\epsilon>0$). Due to compactness, the constant $C(\Phi_\truep,K)$ from Theorem \ref{zernikestability} achieves its infimum over $\{\Phi_\truep'\in F: \supp \Phi_\truep' \subset K, \Vert \Phi_\truep' \Vert_{F} \le M \}$; in particular Theorem \ref{findimstability} follows from Theorem \ref{zernikestability} and Lemma \ref{normcomparison}.
\end{rem}

\begin{rem}[Non-sharp version of Theorem \ref{findimstability}]\label{nonsharp2}
We briefly sketch a proof of the non-sharp version of Theorem  \ref{findimstability}, as discussed in Remark \ref{nonsharp}.
Let $\eta>0$, then for $h\in E_D$ and $\Phi\in C^1(\bar \DD,\so(m))$ with $C^1$-norm $\le M$ we have
\begin{equation}
\Vert \mathbb I_\Phi h \Vert_{L^2(\partial_+S\bar \DD)} \gtrsim_\eta \frac{\Vert \mathbb I_\Phi h \Vert^{1+1/\eta}_{H^1(\partial_+S\bar \DD)}}{\Vert \mathbb I_\Phi h \Vert^{1/\eta}_{H^{1+\eta}(\partial_+S\bar \DD)}} \gtrsim_{M,\eta} \left[ \frac{\Vert h \Vert_{E_D}}{\Vert h \Vert_{\tilde H^{1+\eta}(\DD)}}\right]^{1/\eta} \cdot \Vert h \Vert_{E_D},
\end{equation}
where we first interpolated $H^1$ between $L^2$ and $H^{1+\eta}$ and then used the stability estimate
\cite[Theorem 5.3]{MNP21} and the forward estimate from Theorem 4.2. Now, using Lemma \ref{normcomparison}, the bracket on the right hand side can be lower bounded by $\gtrsim_\eta D^{-\frac 12[1+1/\eta]}$ and the result follows by taking $\eta$ sufficiently large.
\end{rem}

\begin{rem}\label{ponytamer}
Given $K$ and $M,\alpha$ as in the preceding Lemma, as well as some $A_1>0$, there is a threshold $N_\min(M,K,\alpha,A_1)\in \N$ as follows: Suppose a pair of integers $(D,N)\in \N^2$ satisfies $D\le A_1 N^{1/(\alpha+1)}$ and $N\ge N_\min$; then, setting
\begin{equation}
\delta_N = N^{-\frac{\alpha}{2\alpha +2}}\quad \text{ and } \quad \tilde \delta_{N} = \log N \cdot \max\left(\delta_N^\gamma, \delta_N D^{\kappa_0/2} \right)
\end{equation}
as in Condition \ref{bock}, where $\gamma = 1-1/\alpha$ and $\kappa_0 =1/2$,  the prerequisites of Theorem \ref{onetrickpony} are satisfied for 
$
\eta = D^{-2}/\log N
$
and
every $\Phi_\truep\in H^\alpha(\DD,\so(m))$, obeying
\begin{equation}\label{phi0conditions2}
\supp \Phi_\truep \subset K,\quad \Vert \Phi_\truep \Vert_{H^\alpha(\DD)} \le M,\quad \Vert C_{\Phi_\truep} -C_{\Phi_{\truep,D}} \Vert_{L^2_\lambda(\partial_+S\bar \DD)} \le \frac{\delta_N}{2 \sqrt{m}+ 1}.
\end{equation}
Specifically,  on the ball $\B=\{\Phi \in E_D: \Vert \Phi_{\truep,D} - \Phi \Vert_{E_D} < \eta \}$ Condition \ref{goldfisch} is satisfied for exponents $\kappa_0=1/2$ and $\kappa_1=0$ and constants $c_0,c_1$ only depending on $K$ and $M,\alpha$.
To see this, let $\epsilon=\epsilon(M,K,\alpha)$ be as in Lemma \ref{xraylocalcurvature}. 
Since $\alpha>5$, choosing $N$ sufficiently large implies
\begin{align*}
&D^2 \delta_N \le A_1^2 N^{\frac{2}{\alpha+1} } \cdot N^{- \frac{\alpha }{2\alpha +2}} \le (2\sqrt m + 1) \epsilon, \\
&D^2( \log N)^3\delta_N^\gamma \le A_1^{2} N^{\frac{2}{\alpha+1}}\cdot (\log N)^3 N^{-\frac{\alpha-1}{2\alpha+2}} = A_1^{2 } (\log N)^3 N^{-\frac{\alpha-5}{2\alpha+2}} \le 1\\
&D^2 (\log N)^3 \delta_N {D^{\kappa_0/2}}
\le A_1^{\frac{9}4}  \cdot (\log N)^3 N^{-\frac{2\alpha -9}{4\alpha +4}} \le 1
\end{align*}
Hence, for  $N_\min(M,K,A_1,\alpha)\in \N$ large enough, we have 
\begin{equation}
\frac{\delta_N}{2\sqrt m + 1} \le \epsilon D^{-2} \quad \text{ and } \quad \log N\cdot  \tilde \delta_{N} \le \eta \le \epsilon D^{-2} \quad (N\ge N_\min),
\end{equation}
so that Lemma \ref{xraylocalcurvature} applies. In summary we have verified the hypotheses of Theorem \ref{onetrickpony}.
\end{rem}

\begin{proof}[Proof of Lemma \ref{xraylocalcurvature}] 
We start with computing the expected values of the log-likelihood and its derivatives. For a fixed direction $h\in E_D$ and in view of the regularity stated in Lemma \ref{derivatives}, we have
\begin{align}
\ell(\Phi,(X,V,Y)) &= -  \vert C_\Phi(X,V) - Y \vert_F^2/2\\
h^T \nabla \ell(\Phi,(X,V,Y)) & = - \langle C_{\Phi}(X,V) - Y , \dot C_\Phi[h](X,V)\rangle_{F}\\
h^T \nabla^2 \ell(\Phi,(X,V,Y)) h  & = - \vert \dot C_\Phi[h](X,V) \vert_F^2 + \langle C_\Phi(X,V) - Y , \ddot C_\Phi[h](X,V)\rangle_F.
\end{align}
Now  $C_\Phi(X,V) - Y = C_\Phi(X,V) - C_{\Phi_\truep}(X,V) - \varepsilon$, where $\varepsilon$ and $(X,V)$ are independent; in particular $E_{\Phi_\truep}\langle \varepsilon, \varphi(X,V)\rangle_F=0$ for any measurable function $\varphi(X,V)$ and thus we have the following expected values:
\begin{align}
- E_{\Phi_\truep} \ell(\Phi) & =\frac 12 \left( \Vert C_\Phi - C_{\Phi_\truep} \Vert_{L^2_\lambda(\partial_+S\bar \DD)}^2 + 1 \right)\\
- E_{\Phi_\truep } h^T\nabla \ell(\Phi) & = \langle C_\Phi - C_{\Phi_\truep},\dot C_\Phi[h] \rangle_{L^2_\lambda(\partial_+S\bar \DD)}\\
- E_{\Phi_\truep} h^T \nabla^2\ell(\Phi) h  &= \Vert \dot C_\Phi[h] \Vert_{L^2_\lambda(\partial_+S\bar \DD)}^2 - \langle C_{\Phi_\truep} - C_\Phi , \ddot C_\Phi[h]\rangle_{L^2_\lambda(\partial_+\bar \DD)}. \label{seealsothis}
\end{align}

 {\it Curvature bound.} Our first goal is to prove the curvature bound in \eqref{zlocalcurvature1}.  To this end note that the first derivative is given by $\dot C_\Phi[h] = I_{\Xi} h \cdot C_\Phi$ (Lemma \ref{derivatives}) and admits a lower bound, at least at the compactly supported $\Phi_\truep$, due to Theorem \ref{findimstability}. 
To make use of this, we apply standard inequalities (and that $C_\Phi^T C_\Phi=\id$) to the effect that
\begin{equation}\label{zlc2}
\begin{split}
- E_{\Phi_\truep} h^T \nabla^2\ell(\Phi) h  \ge &~\frac 12 \Vert I_{\Xi_\truep} h \Vert_{L^2_\lambda(\partial_+S\bar \DD)}^2\\ & -  \Vert I_{\Xi_\truep} h  - I_{\Xi}h\Vert_{L^2_\lambda(\partial_+S\bar \DD)}^2 - \langle C_{\Phi_\truep}-C_\Phi,\ddot C_\Phi[h] \rangle_{L^2_\lambda(\partial_+S\bar \DD)}
\end{split}
\end{equation}
By Theorem \ref{findimstability}, the first term satisfies $\frac12 \Vert I_{\Xi_\truep} h \Vert_{L^2_\lambda(\partial_+S\bar \DD)}^2 \ge \delta  D^{-1/2}\Vert h \Vert_{E_D}^2$ for some $\delta=\delta(M,K,\alpha)>0$.
Assuming w.l.o.g. that $\Vert h \Vert_{E_D}= 1$ and decreasing  $\delta$ if necessary, we have
\begin{equation}\label{zlc3}
\begin{split}
- E_{\Phi_\truep} h^T \nabla^2 \ell(\Phi) h  \gtrsim_{M,K} &~ \delta D^{-1/2} \cdot \Vert  h \Vert_{E_D}^2\\
 & -\left[\underbrace{\Vert \Phi_\truep - \Phi_{\truep,D} \Vert_{L^\infty(\DD)}^2}_{(i)} + \underbrace{\Vert C_{\Phi_\truep} - C_{\Phi_{\truep,D}} \Vert_{L^2(\DD)} \Vert h \Vert_{L^\infty(\DD)}^2}_{(ii)} \right]\\
& -  \left[\underbrace{\Vert \Phi_{\truep,D} - \Phi \Vert_{L^\infty(\DD)}^2}_{(iii)}  +\underbrace{ \Vert \Phi_{\truep,D} - \Phi \Vert_{L^2(\DD)} \cdot \Vert h \Vert_{L^\infty(\DD)}^2}_{(iv)} \right],
\end{split}
\end{equation}
where the terms $(i)$-$(iv)$ arise from bounding the error terms in \eqref{zlc2} by means of standard forward estimates; we will give a detailed proof below, but first continue with the main line of argument. \\
In order to bound $(i)$,  we use the chain of embeddings $\tilde H^t(\DD)\subset H^{1+\gamma}(\DD) \subset L^\infty(\DD)$ (valid for $t>2$ and some $\gamma=\gamma(t)>0$, see \eqref{zernikeembedding}) and an interpolation argument to estimate the difference $\Psi = \Phi_\truep - \Phi_{\truep,D} $ as follows:
\begin{equation}
\Vert \Psi\Vert_{L^\infty(\DD)} \lesssim_t \Vert\Psi  \Vert_{\tilde H^t(\DD)} \lesssim_t \Vert\Psi \Vert_{L^2(\DD)}^{1-t/5} \cdot \Vert \Psi \Vert_{\tilde H^5(\DD)}^{t/5},\quad  2< t \le 5.
\end{equation}
The $\tilde H^5$-norm can be estimated by $\Vert \Psi \Vert_{\tilde H^5(\DD)} \le \Vert \Phi_\truep \Vert_{\tilde H^5(\DD)} +  \Vert \Phi_{\truep,D} \Vert_{\tilde H^5(\DD)}\le 2\Vert \Phi_\truep \Vert_{\tilde H^5(\DD)}\le 2 M$. To bound the $L^2$-norm we use Lemma \ref{inversecm} for $s=5$. Choosing $t$ such that $(1-t/5)(1-1/5) = 1/4$,  we get
\begin{equation}
(i) \lesssim_{M} \Vert C_{\Phi_\truep} - C_{\Phi_{\truep,D}} \Vert_{L^2_\lambda(\partial_+S\bar \DD)}^{1/2} \le \epsilon ^{1/2} D^{-1},
\end{equation}
where we used the bias estimate in \eqref{phi0conditions}, for $\epsilon>0$ to be chosen.\\
To bound $(ii)$ we use \eqref{phi0conditions} and the fact $\Vert h \Vert_{L^\infty(\DD)}^2 \lesssim D^{3/2}$ (Lemma \ref{normcomparison}); then
\begin{equation}
(ii) \le \epsilon D^{-1/2}
\end{equation}
Similarly, for $\eta>0$ to be chosen, we have
\begin{equation}
(iii) \lesssim \eta^2 D^{3/2}\quad \text{ and } \quad (iv) \lesssim \eta D^{3/2}, 
\end{equation}
where we used Lemma \ref{normcomparison} both for $h$ and for $h'=\Phi_{\truep,D}-\Phi\in E_D$. \\
Combining the previous displays, we obtain (again for $\delta>0$ decreased if necessary to compensate for all implicit constants) the following inequality:
\begin{equation}
- E_{\Phi_\truep} h^T \nabla^2\ell(\Phi) h\gtrsim_{M,K} \delta D^{-1/2} -\left(\epsilon^{1/2} D^{-1} + \epsilon D^{-1/2} + \eta^2 D^{3/2} + \eta D^{3/2} \right),
\end{equation}
valid for $h\in E_D$ with $\Vert h \Vert_{E_D}=1$.  Taking the infimum over $h$, we obtain the lowest eigenvalue of $-E_{\Phi_\truep} \nabla^2 \ell(\Phi)$; overall we have
\begin{equation}\label{zlc4}
\lambda_{\min} \left( - E_{\Phi_\truep}  \nabla^2\ell(\Phi)  \right) \ge CD^{-1/2} \cdot \left[ \delta - \left(\epsilon^{1/2} D^{-1/2} + \epsilon + \eta^2 D^2 + \eta D^2 \right)  \right]
\end{equation}
for constants $C(M,K,\alpha)>0$, $0<\delta(M,K,\alpha)<1 $ and with $\epsilon,\eta>0$ still free to choose.  Under the assumption that
\begin{equation}
\epsilon \le \epsilon^{1/2} \le \delta/5\quad \text{ and } \quad \eta^2 \le \eta \le D^{-2} \delta/5 
\end{equation}
we have
\begin{equation}
\lambda_{\min} \left( - E_{\Phi_\truep}   \nabla^2\ell(\Phi)\right) \ge \frac{C\delta}{5 } D^{-1/2}.
\end{equation}
In particular  \eqref{zlocalcurvature1} follows for $c_0(M,K,\alpha)=C(M,K,\alpha)\delta(M,K,\alpha)/5$ and $\epsilon(M,K,\alpha)= (\delta(M,K,\alpha)/5)^2$.\\
Finally, we demonstrate how to derive \eqref{zlc3} from \eqref{zlc2}; this consists of proving
\begin{align}
\Vert I_{\Xi_\truep} h - I_\Xi h \Vert_{L^2_\lambda(\partial_+S\bar \DD)}^2 &\lesssim_{M,K} (i) + (iii) \label{zlc5}\\
\vert \langle C_{\Phi_\truep}-C_\Phi,\ddot C_\Phi[h] \rangle_{L^2_\lambda(\partial_+S\bar \DD)} \vert &\lesssim_{M,K} (ii) +(iv). \label{zlc6}
\end{align}
For the first inequality we note that $I_\Xi  h= I(R_\Xi^{-1} h)$,  and $R_\Xi \in C(S\bar\DD,\so(m)))$ is an integrating factor for $\Xi$, i.e. a solution to $(X+\Xi)R_\Xi=0$ with $S\bar\DD$ and $R_\Xi =\id$ on $\partial_-S\bar \DD$.  Using a similar notation for $\Xi_\truep$, we have
\begin{equation*}
\begin{split}
 \Vert I_{\Xi_\truep} h  - I_{\Xi}h\Vert_{L^2_\lambda(\partial_+S\bar \DD)}^2 &= \Vert I\left((R_{\Xi_\truep}^{-1} - R_\Xi^{-1}) h )\right)  \Vert_{L^2_\lambda(\partial_+S\bar \DD)}^2\\
&\le \Vert R^{-1}_{\Xi_\truep}  -  R_\Xi^{-1} \Vert_{L^\infty(S\DD)}^2 ,
\end{split}
\end{equation*}
where we used boundedness of $I:L^2(S\DD)\rightarrow L^2_\lambda(\partial_+S\bar \DD)$ \cite[Thm.\,4.2.1]{Sha94}. By \eqref{horse1} we have $\Vert R^{-1}_{\Xi_\truep}  -  R_\Xi^{-1} \Vert_{L^\infty(S\DD)} \lesssim \Vert \Xi_\truep - \Xi \Vert_{L^\infty(\DD)} \lesssim  \Vert \Phi_\truep - \Phi \Vert_{L^\infty(\DD)}$, such that \eqref{zlc6} follows from the triangle inequality.\\
Next, for any $g\in L^2_\lambda(\partial_+S\bar \DD)$ we have (using H{\"o}lder's inequality)
\begin{equation*}
\vert \langle g, \ddot C_\Phi[h]\rangle_{L^2_\lambda(\partial_+S\bar \DD)} \vert =\vert \langle g\cdot C_\Phi,  I_\Xi[hV_h]\rangle_{L^2_\lambda(\partial_+S\bar \DD)}  \vert \le \Vert g \Vert_{L^2_\lambda(\partial_+S\bar\DD)} \cdot \Vert I_\Xi[hV_h] \Vert_{L^\infty(\partial_+S\bar\DD)},
\end{equation*}
where $V_h$ is as in Lemma \ref{derivatives} and we made use of the fact that $C_\Phi$ takes values in $SO(m)$. Using the representation of $I_\Xi$ via integrating factors from above, one sees that it is bounded $L^\infty(S\DD)\rightarrow L^\infty(\partial_+S\bar \DD)$ with operator norm independent of $\Xi$. The previous display can thus be continued with  
\begin{equation}
\lesssim \Vert g \Vert_{L^2_\lambda(\partial_+S\bar\DD)} \cdot \Vert h V_h \Vert_{L^\infty(S\DD)}  
\lesssim \Vert g \Vert_{L^2_\lambda(\partial_+S\bar\DD)} \cdot \Vert h\Vert_{L^\infty(\DD)}^2,
\end{equation}
where we made use of \eqref{horse3}. Setting $g=C_{\Phi_\truep} - C_\Phi$ we get
\begin{equation}
\vert \langle C_{\Phi_\truep} - C_\Phi, \ddot C_\Phi[h]\rangle_{L^2_\lambda(\partial_+S\bar \DD)} \vert \lesssim_m \Vert C_{\Phi_\truep} - C_\Phi \Vert_{L^2_\lambda(\partial_+S\bar\DD)} \cdot \Vert h \Vert_{L^\infty(\DD)}^2. 
\end{equation}
Now the $L^2$-norm on the right hand side can be estimated with the triangle inequality in terms of $C_{\Phi_\truep} - C_{\Phi_{\truep,D}}$, which gives term $(ii)$, and $C_{\Phi_{\truep,D}} - C_\Phi$, which gives term $(iv)$ after applying the global $L^2$-Lipschitz estimate for $\Phi \mapsto C_\Phi$. Together, we have also proved \eqref{zlc6}.\\

{\it Upper bounds.} Finally, we prove the upper bounds \eqref{zupperbound}. Using the triangle inequality and the $L^2$-Lipschitz estimate for $\Phi \mapsto C_\Phi$ we obtain
\begin{equation}\label{zlc7}
\Vert C_\Phi - C_{\Phi_\truep} \Vert_{L^2_\lambda(\partial_+S\bar \DD)} \lesssim \Vert \Phi - \Phi_{\truep,D} \Vert_{L^2_\lambda(\partial_+S\bar \DD)}  + \Vert C_{\Phi_\truep,D} - C_{\Phi_\truep}\Vert_{L^2_\lambda(\partial_+S\bar \DD)} \le 2  \epsilon D^{-2}
\end{equation}
for $\Phi \in \B(\eta,D)$ with $\eta \le \epsilon D^{-2}$. In particular,  assuming without loss of generality that $\epsilon <1$, we have 
\begin{equation}
\vert E_{\Phi_\truep} \ell(\Phi)  \vert \lesssim 1 + \epsilon^2 D^{-4} \lesssim 1.
\end{equation}
Next, for $h\in E_D$ with $\Vert h \Vert_{E_D}\le1$ we have
\begin{equation}
\vert E_{\Phi_\truep} h^T\nabla \ell(\Phi) \vert \le \Vert C_\Phi - C_{\Phi_\truep} \Vert_{L^2_\lambda(\partial_+S\bar \DD)} \cdot \Vert I_\Xi h  \Vert_{L^2_\lambda(\partial_+S\bar \DD)} \lesssim \epsilon D^{-2} \Vert h \Vert_{L^2(\DD)} \lesssim 1,
\end{equation}
where we again used \eqref{zlc7}, as well as $L^2$-boundedness of $I_\Xi$ (with operator norm bounded uniformly in $\Xi$). Finally, using additionally \eqref{zlc6}, we have
\begin{equation*}
\begin{split}
\vert E_{\Phi_\truep}h^T \nabla^2 \ell(\Phi) h  \vert  &\le \Vert I_\Theta h \Vert_{L^2_\lambda(\partial_+S\bar \DD)}^2 + \vert \langle C_{\Phi_\truep}-C_\Phi,\ddot C_\Phi[h] \rangle_{L^2_\lambda(\partial_+S\bar \DD)} \\
&\lesssim_{M,K} \Vert h \Vert_{L^2(\DD)}^2 + \Vert C_{\Phi_\truep} - C_{\Phi_{\truep,D}} \Vert_{L^2_\lambda(\partial_+S\bar\DD)} \Vert h \Vert_{L^\infty(\DD)}^2 +  \Vert \Phi_{\truep,D} - \Phi \Vert_{L^2(\DD)} \cdot \Vert h \Vert_{L^\infty(\DD)}^2\\
 &\lesssim 1 + \epsilon D^{-2} \cdot \Vert h \Vert_{L^\infty(\DD)}^2.
\end{split}
\end{equation*}
As above, we have $\Vert h \Vert_{L^\infty(\DD)}^2\lesssim D^{3/2}$, which is tempered by $D^{-2}$.  Combining the previous displays we obtain, for $\Phi \in \B(\eta,D)$ and $h\in E_D$ with $\Vert h \Vert_{E_D} \le 1$
\begin{equation}
\vert E_{\Phi_\truep} \ell(\Phi) \vert + \vert E_{\Phi_\truep} h^T \nabla \ell(\Phi) \vert + \vert E_{\Phi_\truep} h^T \nabla^2 \ell(\Phi) h  \vert \le c_2, 
\end{equation}
where $c_2=c_2(M,K)$ incorporates all implicit constants from above. Finally,
\begin{equation*}
\sup_{\Vert h \Vert_{E_D}\le 1} \vert E_{\Phi_\truep} h^T \nabla \ell(\Phi) \vert = \vert E_{\Phi_\truep} h^T \nabla \ell(\Phi) \vert_{\R^D},\quad  \sup_{\Vert h \Vert_{E_D}\le 1} \vert E_{\Phi_\truep} h^T \nabla^2 \ell(\Phi) h \vert = \vert E_{\Phi_\truep} h^T \nabla^2 \ell(\Phi) h  \vert_{\mathrm{op}}
\end{equation*}
and the proof is complete.
\end{proof}

\textbf{Acknowledgements} JB would like to thank François Monard for fruitful discussions leading to the proof of Theorem \ref{zernikeforward}. Both authors would like to thank Gabriel Paternain and Sven Wang and the two anonymous referees for various helpful remarks.

This article was part of the Doctoral thesis of the first author, who was supported by the EPSRC Centre for Doctoral Training and the Munro-Greaves Bursary of Queens' College Cambridge. 
The second author was supported by ERC grant No.647812 (UQMSI). \\

\bibliographystyle{plain}
\bibliography{revision07_22.bbl}

\end{document}